\documentclass[11pt,leqno]{amsart}

\usepackage{amsmath}
\usepackage{amssymb}
\usepackage{a4wide}
\usepackage{bbm}
\usepackage{graphics}
\usepackage{epsfig}
\usepackage{mathrsfs}
\usepackage[sans]{dsfont}

\newtheorem{definition}{Definition}[section]
\newtheorem{theorem}[definition]{Theorem}
\newtheorem{corollary}[definition]{Corollary}
\newtheorem{lemma}[definition]{Lemma}

\newtheorem{proposition}[definition]{Proposition}
\newtheorem{remark}[definition]{Remark}
\newtheorem{example}[definition]{Example}
\newtheorem{assumption}{Assumption}

\newcounter{hypo}
\newcounter{hypoa}

\newcounter{hypoaa}
\newcounter{hypobb}

\def\C{{\mathbb C}}

\def\N{{\mathbb N}} 
\def\R{{\mathbb R}} 

\def\Z{{\mathbb Z}}

\def\T{{\mathbb T}}
\def\S{{\mathbb S}}

\def\CC{\mathcal {C}}

\def\CE{\mathcal {E}}
\def\CF{\mathcal {F}}

\def\CI{{\mathcal I}}
\def\CJ{{\mathcal J}}

\def\CM{\mathcal {M}}

\def\CO{\mathcal {O}}
\def\CP{\mathcal {P}}
\def\CQ{\mathcal {Q}}

\def\CU{\mathcal {U}}
\def\CV{\mathcal {V}}

\def\SQ{\mathscr {Q}}

\def\ker{\mathop{\rm Ker}\nolimits}

\def\one{\mathds{1}}

\newcommand{\re}{\operatorname{Re}}
\newcommand{\im}{\operatorname{Im}}

\newcommand{\Id}{\operatorname{Id}}

\newcommand{\supp}{\operatorname{supp}}

\def\dist{\mathop{\rm dist}\nolimits}

\def\sgn{\mathop{\rm sgn}\nolimits}

\def\<{\langle}
\def\>{\rangle}

\def\rank{\mathop{\rm Rank}\nolimits}

\makeatletter
 \@addtoreset{equation}{section}
 \makeatother
 
\title{Semiclassical Schr\"{o}dinger operators with purely imaginary potential}
\author[V. Arnaiz, J.-F. Bony and  L. Michel]{Victor Arnaiz, Jean-Fran\c{c}ois Bony and Laurent Michel} 
\address{Universit\'e de Bordeaux, Institut Math\'ematique de Bordeaux}
\email{victor.arnaiz@math.u-bordeaux.fr}
\email{jean-francois.bony@math.u-bordeaux.fr}
\email{laurent.michel@math.u-bordeaux.fr}

\begin{document}

\begin{abstract}
We consider Schr\"odinger operators with purely imaginary potential $P = - h^{2} \Delta + i V ( x )$ on a bounded domain. Assuming that near its critical points the potential $V$ can be approximated by an homogeneous polynomial, we show that in the limit $h \to 0$ the leftmost eigenvalues of $P$ are asymptotically given by the local model associated to the most degenerated critical points of $V$. We give applications of this result to the associated evolution problem including shear flows in fluid mechanics.
\end{abstract}

\maketitle

\tableofcontents
\section{Introduction}

\subsection{Motivations}
The mathematical analysis of the low-lying eigenvalues of Schr\"odinger operators $P=-h^2\Delta+V(x)$ in the semiclassical limit goes back to the work of Combes--Duclos--Seiler \cite{CoDuSe83}, Simon \cite{Si83_01}. When the minima of the potential $V$ are non-degenerated, one can associate to each minimum a model operator with quadratic potential. Using the particular scaling of this quadratic model, this yields a first order approximation of the bottom of the spectrum of $P$. These asymptotics were next sharpened by Helffer--Sj\"ostrand \cite{HeSj84_01} using WKB methods, with also a study of the tunneling effect (see also the monographs \cite{DiSj99_01}, \cite{He88_01}). The proof of these results uses crucially the selfadjointness of $P$, through maxi-min principle or resolvent estimates.

When the potential $V$ is complex valued, the situation is drastically different since the operator becomes non-selfadjoint and the resolvent is not controlled by the distance to the spectrum. The simple example of complex harmonic oscillator $- \partial_{x}^{2} + i x^{2}$ on $\R$ studied in \cite{Da99_01} captures easily this phenomenon: the spectrum is made of the complex numbers $e^{i \pi / 4} ( 2 \N + 1 )$ and the norm of the resolvent verifies $\lim_{r \rightarrow + \infty} \Vert ( - \partial_{x}^{2} + i x^{2} - r e^{i \theta} )^{- 1} \Vert = + \infty$ for all $\vert \theta \vert < \frac{\pi}{2}$ (see also \cite{DeSjZw04_01} for a microlocal point of view and generalization of this result). The notion of pseudospectrum \cite{Tr97_01}, gives a general mathematical framework to this remark. Of course, this lack of control of the resolvent leads to serious complications if one aims to study more general Schr\"odinger operators. Another difficulty arises when one tries to go beyond the Morse assumption on the potential $V$ for at least two reasons. First, the normal form theory of degenerate critical points is much more complicated \cite{Ar76_01}. Moreover, in dimension at least $2$, the resulting Schr\"odinger operators have neither separated variables nor natural scaling in the general case. Hence, it seems natural to study first the case of homogeneous critical points. In the case of real valued potentials, this is done for instance in \cite{MaRo88_01}. The aim of this paper is to study the case of Schr\"odinger operators with purely imaginary potentials. Such operators appear in the literature in the study of supraconductivity (see \cite{Al08_01} and references therein), in the study of shear flows in fluid mechanics (see \cite{DrRe81_01} for an introduction), and also in sub-Riemannian control problems \cite{BeHeHeRo15_01}. We shall now introduce our precise framework.

\subsection{Framework and main result} \label{s4}

Consider the semiclassical Schr\"odinger operator 
\begin{equation} \label{a12}
P = - h^{2} \Delta + i V ( x ) ,
\end{equation}
acting on $L^{2} ( X )$, where $X$ is either the $d$-dimensional torus $\T^{d} = \R^{d} / 2 \pi \Z^{d}$ or a bounded open connected Lipschitz subset of $\R^{d}$, $d \geq 1$. In the latter case, we assume that $P$ satisfies the Dirichlet or Neumann boundary conditions at $\partial X$. In all the paper, we assume that $V$ is a $C^{1} ( \overline{X} )$ real valued potential. Under this assumption, the operator $P$ with domain $H^{2}( X )$, with boundary condition if $X \subset \R^{d}$, is maximal accretive. It has compact resolvent, its spectrum is made of finite multiplicity eigenvalues and
\begin{equation}
\sigma ( P ) \subset [ 0 , + \infty [ + i V ( \overline{X} ) .
\end{equation}
By analogy with the hypoelliptic theory, the leftmost eigenvalues are intuitively given by the critical points of $V$. In this paper, we will assume that these critical points are homogeneous degenerate of finite order. More precisely,

\begin{assumption}[Degenerate Morse]\sl \label{h1}
Let $E \in \R$. For all $c \in V^{-1} ( E ) \subset \overline{X}$, we have either

$i)$ $c$ is a regular point (i.e. $\nabla V ( c ) \neq 0$).

$ii)$ $c$ is a critical point (i.e. $\nabla V ( c ) = 0$). In that case, we assume $c \in X$ and
\begin{equation} \label{a4}
\nabla \big( V ( x ) - V_{c} ( x ) \big) = o \big( \vert x - c \vert^{\alpha_{c} - 1} \big) ,
\end{equation}
for $x$ near $c$, where $V_{c}$ can be written
\begin{equation} \label{a83}
V_{c} ( x ) = E + v_{c} ( \theta ) r^{\alpha_{c}} ,
\end{equation}
in the spherical coordinates $( r = \vert x - c \vert , \theta ) \in \R^{+} \times \S^{d -1}$ centered at $c$, with $\alpha_c>1$ and $v_c\in C^{1} ( \S^{d - 1} )$. Moreover, we make the non-degeneracy assumption
\begin{equation} \label{a81}
\forall x \in \R^{d} \setminus \{ c \} , \quad \nabla_{x} V_{c} ( x ) \neq 0 .
\end{equation}
\end{assumption}

It follows immediately from this assumption that, near a critical point $c \in X$, we have
\begin{equation} \label{a57}
V ( x ) = V_{c} ( x ) + o ( \vert x - c \vert^{\alpha_{c}} ) ,
\end{equation}
with $V_c$ homogeneous of degree $\alpha_{c}$. Nevertheless, \eqref{a57} does not imply \eqref{a4}. The assumptions $\alpha_{c} > 1$ and $v_{c} \in C^{1} ( \S^{d- 1 } )$ guarantee that the potential $V_{c} ( x )$ is always in $C^{1} ( \R^{d} )$, but \eqref{a4} does not imply that $V ( x )$ is better than $C^{1}$ near $c$ even for large $\alpha_{c}$ and $v_{c} \in C^{\infty} ( \S^{d - 1 } )$. When $V \in C^{\infty} ( \overline{X} )$, we have $\alpha_{c} \in \N$, $\alpha_{c} \geq 2$ and $V_{c}$ is the first non-zero homogeneous part in the Taylor expansion of $V$ at $c$.

\begin{remark}\sl \label{a85}
For $g \in C^{1} ( \R^{d} )$ and $x = r \theta \neq 0$, we have $\nabla_{x} g ( x ) = 0$ iff $\partial_{r} g ( r \theta ) = 0$ and $\nabla_{\theta} g ( r \theta )= 0$. Then, since $\partial_{r} V_{c} = \alpha_{c} r^{\alpha_{c} - 1} v_{c} ( \theta )$ and $\nabla_{\theta} V_{c} = r^{\alpha_{c}} \nabla_{\theta} v_{c} ( \theta )$, the non-degeneracy assumption \eqref{a81} is equivalent in dimension $d \geq 2$ to
\begin{equation} \label{a82}
\forall \theta \in \S^{d - 1} , \quad v_{c} ( \theta ) = 0 \quad \Longrightarrow \quad \nabla_{\theta} v_{c} ( \theta ) \neq 0 .
\end{equation}
\end{remark}

Depending on the setting, it may be more appropriate to check \eqref{a81} or \eqref{a82} to verify Assumption~\ref{h1}. By homogeneity, \eqref{a81} implies $\vert \nabla V_{c} ( x ) \vert \gtrsim \vert x - c \vert^{\alpha_{c} - 1}$. Combined with \eqref{a4}, it shows that the critical points in $V^{- 1} ( E )$ are isolated under Assumption~\ref{h1}. Since $\overline{X}$ is compact, the set of critical points of energy $E$ is finite. Moreover, the energies close to $E$ have only regular points. Then, \eqref{a81} is not a technical assumption, but needed for the results below (see also Example \ref{a93}).

\begin{example}\rm
In dimension $d = 1$, $\S^{d - 1} = \{ + , - \}$ and \eqref{a83} writes
\begin{equation} \label{a84}
V_{c} ( x ) = E + \left\{ \begin{aligned}
&v_{c}^{+} \vert x - c \vert^{\alpha_{c}}   &\text{for } x \geq c , \\
&v_{c}^{-} \vert x - c \vert^{\alpha_{c}}   &\text{for } x < c  ,
\end{aligned} \right.
\end{equation}
for some $\alpha_{c} > 1$ and $v_{c}^{\pm} \in \R^{*}$. In dimension $1$, the present proof allows to treat the critical points at the boundary of $X$. When $c \in \partial X$, we assume $V_{c} ( x ) = E + v_{c} \vert x - c \vert^{\alpha_{c}}$ for some $\alpha_{c} > 1$ and $v_{c} \in \R^{*}$.
\end{example}

Under \eqref{a82}, the zeros of $v_{c}$ form a finite number of non-intersecting closed hypersurfaces of $\S^{d - 1}$. Moreover, $v_{c}$ changes sign at each of these hypersurfaces. By \eqref{a57}, the critical point $c$ is isolated in $V^{- 1} ( E )$ in dimension $d \geq 2$ iff $v_{c}$ is globally positive or globally negative.

\begin{example}\rm
In $X = B ( 0 , 1 ) \subset \R^{2}$, we consider the potential
\begin{equation*}
V ( x ) = \sin ( 4 t ) r^{2} = 4 \frac{x_{1}^{3} x_{2} - x_{1} x_{2}^{3}}{x_{1}^{2} + x_{2}^{2}} ,
\end{equation*}
in the variables $r \theta = r ( \cos ( t ) , \sin ( t ) )$. It has a single critical point $c = 0$ and $V_{c} = V$. Since $v_{c} ( \theta ) = \sin ( 4 t ) = 4 \theta_{1}^{3} \theta_{2} - 4 \theta_{1} \theta_{2}^{3}$ is non-degenerate, this potential satisfies Assumption~\ref{h1}. Note that $v_{c}$ vanishes $8$ times in $\S^{1}$.
\end{example}

By the max-min principle, the asymptotic of the low-lying eigenvalues of the selfadjoint operator $P = - h^{2} \Delta + V ( x )$ requires only an equivalent of $V$ near its global minimum. Thus, $V \in L^{2}_{\rm loc}$ and \eqref{a57} (with $v_{c} > 0$) is enough. For purely imaginary potential, we require that $V \in C^{1}$ and \eqref{a4}. This is natural since the leftmost eigenvalues are now given by all the critical points of $V$ (and not just the global minima) and \eqref{a57} does not exclude that $V$ has an infinite number of critical points (of order greater than $\alpha_{c}$) in any neighborhood of $c$. See Example \ref{b31} for such a situation. In other words, the shape of the potential is not sufficient to get the asymptotic of the small eigenvalues unlike the selfadjoint case.

Let us denote by
\begin{equation*}
\CC ( E )  = \{ c \in X ; \ V ( c ) = E \text{ and } \nabla V ( c ) = 0 \} ,
\end{equation*}
the set of critical points of $V$ at energy $E$. The idea is that the eigenvalues closest to the imaginary axis near $i E$ are given by the most degenerate critical points. Then, we define $\alpha = \alpha ( E ) = \max \{ \alpha_{c} ; \ c \in \CC ( E ) \}$ the maximal order of vanishing and the set $\CC_{\max} ( E ) = \{ c \in \CC ( E ) ; \ \alpha_{c} = \alpha ( E ) \}$ of critical points of maximal vanishing at energy $E$. We also denote
\begin{equation} \label{a19}
\beta = \frac{2}{\alpha+ 2} \qquad \text{ and } \qquad \sigma = \alpha \beta = \frac{2 \alpha}{\alpha + 2} .
\end{equation}
Roughly speaking, $\beta$ is the scaling order, whereas $\sigma$ corresponds to the order of the first eigenvalues.

For $c \in \CC_{\max} ( E )$, the leading part of $P$ near $c$ is defined by
\begin{equation} \label{a48}
P_{c} = - h^{2} \Delta  + i V_{c} ( x ) ,
\end{equation}
on $\R^{d}$. If $c \in \partial X$ in dimension $d = 1$, the operator $P_{c}$ acts on $[ c , + \infty [$ or $] - \infty , c ]$ depending on the situation with the same boundary condition than $P$ at $c$. Using the conjugation operator $U_{c}$, which is the isometry on $L^{2} ( \R^{d} )$ defined by
\begin{equation} \label{a5}
U_{c} ( u ) ( x ) = h^{\frac{d \beta}{2}} u ( c + h^{\beta} x ) ,
\end{equation}
a direct computation gives
\begin{equation} \label{a32}
U_{c} P_{c} U_{c}^{*} = h^{\sigma} \CP_{c} + i E ,
\end{equation}
where $\CP_{c}$ is the model operator acting on $L^{2} ( \R^{d} )$
\begin{equation} \label{b49}
\CP_{c} = - \Delta + i \CV_{c} ( x )  \qquad \text{ with } \qquad \CV_{c} ( x ) = V_{c} ( c + x ) = v_{c} ( \theta ) r^{\alpha} ,
\end{equation}
in the spherical coordinates $x = r \theta$. If $c \in \partial X$ in $1$D, $\CP_{c}$ is defined on $\R^{\pm}$ with the same boundary condition at $0$ than $P$ at $c$. The spectral properties of this operators are described in details in Section \ref{s2}. To state our main result, we only need the following proposition which is a consequence of Propositions \ref{a46} and \ref{a47}.

\begin{proposition}\sl \label{a6}
The operator $\CP_{c}$  with domain $C^{\infty}_{0} ( \R^{d} )$ is closable. The spectrum of its closed extension, still denoted $\CP_{c}$, consists of a sequence (possibly empty or finite) of eigenvalues of finite multiplicity whose real part is positive and goes to $+ \infty$ (in the case of infinitely many eigenvalues). In the sequel, we denote
\begin{equation*}
\Lambda_{c} = \sigma ( \CP_{c} ) .
\end{equation*}
\end{proposition}

We do not know if $\Lambda_{c}$ is always an infinite set or at least a non-empty set. Nevertheless, if $V_{c}$ has a sign (see Proposition \ref{a65}) or if $V_{c} = v ( x - c )^{n}$ is a polynomial in $1$D (see Proposition \ref{a65} and Theorem \ref{a66}), $\Lambda_{c}$ is an infinite set. Note that the last setting covers the case of smooth $V \in C^{\infty} ( \overline{X} )$ in $1$D with critical points of finite order. The semiclassical distribution of the eigenvalues of $P$ near $E$ is given by the following.

\begin{theorem}[Asymptotic of eigenvalues]\sl \label{a7}
Suppose that $P$ and $E$ satisfy Assumption~\ref{h1} and let $\varepsilon > 0$ be small enough. Fix $M > 0$ such that $M \notin \bigcup_{c \in \CC_{\max} ( E )} \re \Lambda_{c}$ and denote $\Omega = [ 0 , M h^{\sigma} ] + i [ E - \varepsilon , E + \varepsilon ] $. Then, for $h$ small enough, there exist a finite set $J$, a labelling $( \lambda_{j} )_{j \in J}$ of $\sigma ( P ) \cap \Omega$ and a labelling $( \mu_{j} )_{j \in J}$ of $\cup_{c \in \CC_{\max} ( E )} \Lambda_{c} \cap \{ \re z \leq M \}$ such that for all $j \in J$
\begin{equation}\label{b35}
\lambda_{j} = i E + h^{\sigma} \mu_{j} + o ( h^{\sigma} ) .
\end{equation}
Here the labelling of eigenvalues is done according to their algebraic multiplicity.

Moreover, for all $\gamma > 0$, there exists $C > 0$ such that
\begin{equation} \label{a8}
\Vert ( P - z )^{- 1} \Vert \leq C h^{- \sigma} ,
\end{equation}
for $h$ small enough and $z \in \Omega$ with $\dist ( z , \cup_{c \in \CC_{\max} ( E ) } ( i E+ h^{\sigma} \Lambda_{c} ) ) \geq \gamma h^{\sigma}$.
\end{theorem}

In this result, the set $J$ is independent of $h$ small enough. It is possible to describe the sum of the spectral projectors associated to the eigenvalues of $P$ close to an element of $\cup_{c \in \CC_{\max} ( E )} ( i E + h^{\sigma} \Lambda_{c} )$ in terms of the spectral projectors of $\CP_{c}$ (see \eqref{a38}). When $\sigma ( \CP_{c} ) \neq 0$, the eigenvalue free region of size $h^{\sigma}$ and the resolvent estimate \eqref{a8} are optimal. Moreover, the asymptotic \eqref{b35} shows that the leftmost eigenvalues of $P$ are the one given by the critical values of $V$ which are the more degenerated. Note that the multiplicity of $\lambda \in \sigma ( P )$ may arise either from a multiple eigenvalue of a model operator $P_{c}$, or from several model operators $P_{c}$ sharing the same eigenvalue.

To the best of our knowledge, there are very few results concerning spectral asymptotics for the operator $P$. Let us mention the paper by Almog--Henry \cite{AlHe16_01} for non-critical potential $V$ (see also \cite{AlGrHe19_01}, \cite{AlHe20_01} where various boundary conditions are considered and \cite{AlGrHe18_01} on exterior domains). In all these papers the potential does not have any critical point and the spectrum is described by some Airy operator. As far as we know, results obtained when critical points are present appear to be even rarer. In his PhD thesis, Henry \cite{He13_02} obtained an asymptotic equivalent of the real part of the leftmost eigenvalue in the case of Morse potentials. In dimension $1$, Averseng--Frantz--H\'erau--Raymond \cite{AvFrHeRa25_01} have studied the tunneling effect (and the asymptotic of eigenvalues) between two symmetric wells. In a relatively close framework, let us also mention the work by Morin--Raymond--V{\~u} Ng{\d{o}}c \cite{MoRaVu23_01} on magnetic Schr\"odinger operator with complex potential.

The resolvent estimate \eqref{a8} is the natural estimate one hopes to get in this situation, generalizing the estimate of the resolvent of order $\dist ( z , \sigma ( P ) )^{- 1}$ in the selfadjoint case. Similar estimates have been obtained by Malkov \cite{Ma26_01} for Schr\"odinger operators with complex potentials (not necessarily purely imaginary) in non-degenerate case (see also \cite{HeHiSj08_02} for general results on non-selfadjoint real second order differential operators).
Estimate \eqref{a8} should be compared with the results of Coti Zelati--Gallay \cite{CoGa23_01} where operators like $P$ were considered (see also \cite{BeCoZ17_01} for previous results on related topics). In Proposition 2.4 of \cite{CoGa23_01}, it is proven that there exists a constant $C_*>0$ such that \eqref{a8} holds true for $\re z \leq C_*h^{\sigma}$, under some assumptions on the potential $V$ which are similar to our in dimension $1$ and for Morse potential in higher dimension. 
Theorem \ref{a7} improves their result in several directions. First we identify the best constant $C_*$ as the infimum of the real part of the spectrum of the model operators. Second, we go beyond the vertical line $\re z =C_*h^{\sigma}$ and compute an approximation of the bottom of the spectrum. Let us also mention that our degenerate Morse assumption requires only the potential to be $C^{1}$.
Though these are improvements of the results of \cite{CoGa23_01}, our approach is based on their paper which was a great source of inspiration.

\begin{figure}
\begin{center}
\begin{picture}(0,0)%
\includegraphics{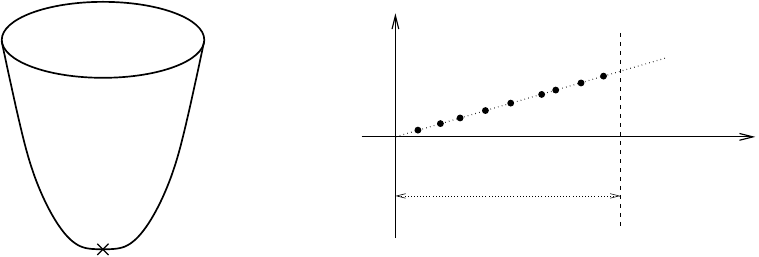}%
\end{picture}%
\setlength{\unitlength}{1184sp}%
\begingroup\makeatletter\ifx\SetFigFont\undefined%
\gdef\SetFigFont#1#2#3#4#5{%
  \reset@font\fontsize{#1}{#2pt}%
  \fontfamily{#3}\fontseries{#4}\fontshape{#5}%
  \selectfont}%
\fi\endgroup%
\begin{picture}(20177,6813)(-10543,-2344)
\put(2701,-361){\makebox(0,0)[lb]{\smash{{\SetFigFont{9}{10.8}{\rmdefault}{\mddefault}{\updefault}$h^{\frac{2 \alpha}{\alpha + 2}}$}}}}
\put(-8999,2939){\makebox(0,0)[b]{\smash{{\SetFigFont{9}{10.8}{\rmdefault}{\mddefault}{\updefault}$V$}}}}
\put(-299,1064){\makebox(0,0)[rb]{\smash{{\SetFigFont{9}{10.8}{\rmdefault}{\mddefault}{\updefault}$0$}}}}
\put(7576,2864){\makebox(0,0)[lb]{\smash{{\SetFigFont{9}{10.8}{\rmdefault}{\mddefault}{\updefault}$e^{i \frac{\pi}{\alpha + 2}} \R^{+}$}}}}
\put(-7799,-1786){\makebox(0,0)[b]{\smash{{\SetFigFont{9}{10.8}{\rmdefault}{\mddefault}{\updefault}$0$}}}}
\end{picture}%
\end{center}
\caption{The potential $V$ and the eigenvalues of $P$ in Example \ref{a92}.} \label{f4}
\end{figure}
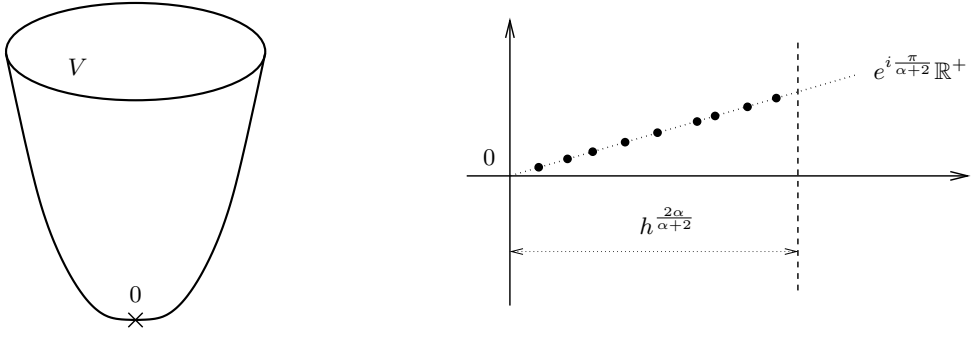

\begin{example}\rm \label{a92}
In $X = B ( 0 , 1 ) \subset \R^{d}$, we consider the potential
\begin{equation*}
V ( x ) = \vert x \vert^{\alpha} ,
\end{equation*}
for some $\alpha > 1$. In the energy surface $E = 0$, it has a single critical point $c = 0$ and $V$ satisfies Assumption~\ref{h1}. Then, Theorem \ref{a7} and Proposition \ref{a65} show that, modulo $o ( h^{\frac{2 \alpha}{\alpha + 2}} )$ terms, the eigenvalues of $P$ near $0$ are given by
\begin{equation*}
e^{i \frac{\pi}{\alpha + 2}} \sigma ( - \Delta + \vert x \vert^{\alpha} ) h^{\frac{2 \alpha}{\alpha + 2}} .
\end{equation*}
This setting is illustrated in Figure \ref{f4}.
\end{example}

\begin{example}\rm \label{a93}
In a neighborhood $X$ of $0 \in \R^{2}$, we consider the potentials
\begin{equation*}
V_{1} ( x_{1} , x_{2} ) = x_{1}^{2} \qquad \text{ and } \qquad V_{1} ( x_{1} , x_{2} ) = x_{1}^{2} + x_{2}^{4} .
\end{equation*}
For these two potentials, $c = 0$ is a critical point of energy $E = 0$ and $V_{c} = x_{1}^{2} = v_{c} ( \theta ) r^{\alpha_{c}}$ with $v_{c} ( \theta ) = \theta_{1}^{2}$ and $\alpha_{c} = 2$. But, $v_{c}$ is degenerate and Assumption~\ref{h1} is not satisfied. Note that the spectrum of $P_{c} = - \Delta + i x_{1}^{2}$ on $L^{2} ( \R^{2} )$ is not discrete and the conclusions of Theorem \ref{a7} cannot be true. Moreover, $V_{1}^{- 1} ( 0 ) =\{ 0 \} \times \R$ and $V_{2}^{- 1} ( 0 ) =\{ 0 \}$ are all critical points of energy $E = 0$, showing that the critical points may or may not be isolated without \eqref{a81}, regardless of $P_{c}$. Nevertheless, it should be possible to give the asymptotic of the small eigenvalues for these potentials by adapting our approach.
\end{example}

\begin{remark}\sl \label{a94}
In the statement of Theorem \ref{a7}, we implicitly assume that $V$ has at least one critical point of energy $E$ (i.e. $\CC ( E ) \neq \emptyset$). If the energy $E$ is regular (i.e. $\CC ( E ) = \emptyset$), there exist $C , m , \varepsilon > 0$ such that, for $h$ small enough, $P$ has no eigenvalue in $[ 0 , m h^{2 / 3} ] + i [ E - \varepsilon , E + \varepsilon ]$ and that
\begin{equation*}
\Vert ( P - z )^{- 1} \Vert \leq C h^{- 2 / 3} ,
\end{equation*}
for $z$ in that set. To prove that, it is enough to follow Section \ref{s1} and to take $\delta = h^{2 / 3}$. Note that this was already proved in \cite{AlHe16_01} and \cite{CoGa23_01}. From a global perspective, if $X=\Pi^d$ then $V$ admits necessarily some extrema and hence there exists some energy $E$ such that the set $\CC ( E )$ is non-empty. However, in the case where $X$ is a bounded subset of $\R^d$, it is possible that the potential $V$ does not have any critical point. In that case, one can apply \cite{AlHe16_01} to get the spectral asymptotics.
\end{remark}

The remainder term in Theorem \ref{a7} is only a $o ( h^{\sigma} )$. Under Assumption~\ref{h1}, it seems difficult to have a better estimate. But if we assume $\nabla ( V ( x ) - V_{c} ( x ) ) = \CO ( \vert x - c \vert^{\alpha_{c} - 1 + \nu} )$ for some $\nu > 0$ instead of \eqref{a4}, it should be possible to replace the $o ( h^{\sigma} )$ by a $\CO ( h^{\sigma + \widetilde{\nu}} )$ for some $\widetilde{\nu} > 0$. For that, one should replace the big constants $R , S, s^{- 1}$ by some negative powers of $h$ in the proof of Theorem \ref{a7}. One may also ask whether the eigenvalues admit a full asymptotic expansion in fractional powers of $h$ when $V \in C^{\infty}$ (or more generally when $V$ has near any critical point $c$ an  asymptotic expansion in fractional powers of $\vert x - c\vert$). However, the advantage of considering $C^{1}$ potentials is that it allows to study the degeneracies in $\vert x - c \vert^{\alpha_{c}}$ for every $\alpha_{c} > 1$ and then to have all the possible values of $\sigma \in ] 2 / 3 , 2[$ for the asymptotic of the eigenvalues.

\begin{figure}
\begin{center}
\begin{picture}(0,0)%
\includegraphics{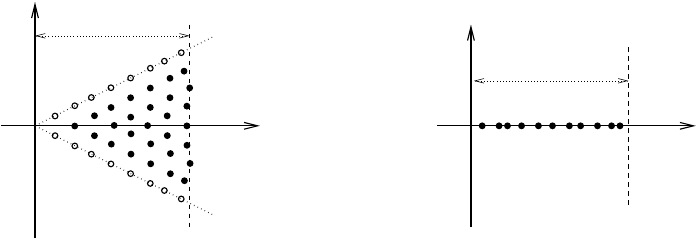}%
\end{picture}%
\setlength{\unitlength}{1184sp}%
\begingroup\makeatletter\ifx\SetFigFont\undefined%
\gdef\SetFigFont#1#2#3#4#5{%
  \reset@font\fontsize{#1}{#2pt}%
  \fontfamily{#3}\fontseries{#4}\fontshape{#5}%
  \selectfont}%
\fi\endgroup%
\begin{picture}(18591,6366)(-932,-2194)
\put(11326,1064){\makebox(0,0)[rb]{\smash{{\SetFigFont{9}{10.8}{\rmdefault}{\mddefault}{\updefault}$0$}}}}
\put(13501,2489){\makebox(0,0)[lb]{\smash{{\SetFigFont{9}{10.8}{\rmdefault}{\mddefault}{\updefault}$h^{6 / 5}$}}}}
\put(5176,3164){\makebox(0,0)[lb]{\smash{{\SetFigFont{9}{10.8}{\rmdefault}{\mddefault}{\updefault}$e^{i \frac{\pi}{6}} \R^{+}$}}}}
\put(5176,-1711){\makebox(0,0)[lb]{\smash{{\SetFigFont{9}{10.8}{\rmdefault}{\mddefault}{\updefault}$e^{- i \frac{\pi}{6}} \R^{+}$}}}}
\put(-299,1064){\makebox(0,0)[rb]{\smash{{\SetFigFont{9}{10.8}{\rmdefault}{\mddefault}{\updefault}$0$}}}}
\put(1801,3614){\makebox(0,0)[lb]{\smash{{\SetFigFont{9}{10.8}{\rmdefault}{\mddefault}{\updefault}$h^{4 / 3}$}}}}
\end{picture}%
\end{center}
\caption{The eigenvalues of $P$ in Example \ref{a91} with $d = 2$, $n = 4$, $v_{1} = - v_{2}$ on the left and $n = 3$ on the right (in filled circles).} \label{f3}
\end{figure}
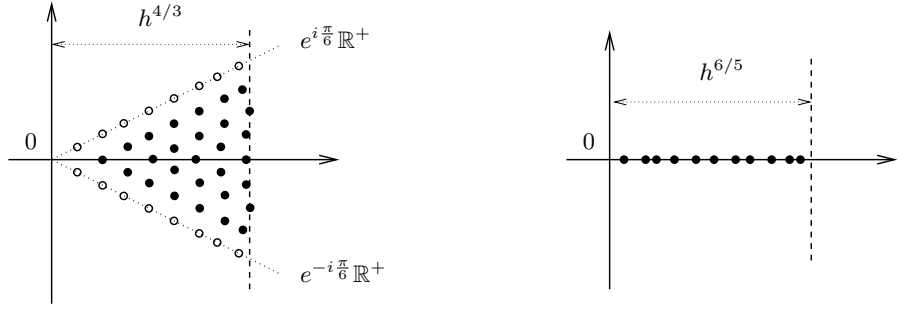

\begin{example}[generalized Morse case]\rm \label{a91}
Let $V \in C^{1} ( \R^{d}) $ defined in a neighborhood of $0$ and satisfying \eqref{a4} near $c = 0$ with
\begin{equation} \label{a86}
V_{c} ( x ) = \sum_{j = 1}^{d} v_{j} x_{j}^{n} ,
\end{equation}
for some $n \in \N$, $n\geq 2$ and $v_{j} \in \R^{*}$. Such a potential satisfies Assumption~\ref{h1} at energy $E = 0$. Since the variables are separated, we have
\begin{equation}\label{b36}
\sigma ( \CP_{c} ) = \sum_{1 \leq j \leq d} \sigma ( - \partial_{x}^{2} + v_{j} x^{n} ) ,
\end{equation}
taking into account the multiplicity and considering $- \partial_{x}^{2} + v_{j} x^{n}$ on $L^{2} ( \R )$. Then, Theorem \ref{a7} and Proposition \ref{a65} show that, modulo $o ( h^{\frac{2 n}{n + 2}} )$ terms, the eigenvalues of $P$ near $0$ are given by
\begin{equation*}
\sum_{1 \leq j \leq d} e^{i \sgn v_{j} \frac{\pi}{n + 2}} \vert v_{j} \vert^{\frac{2}{n + 2}} \sigma ( - \partial_{x}^{2} + x^{n} ) h^{\frac{2 n }{n + 2}} ,
\end{equation*}
for $n$ even and 
\begin{equation*}
\sum_{1 \leq j \leq d} \vert v_{j} \vert^{\frac{2}{n + 2}} \sigma ( - \partial_{x}^{2} + i x^{n} ) h^{\frac{2 n }{n + 2}} ,
\end{equation*}
for $n$ odd (see Figure \ref{f3}). In the previous expressions the operators $- \partial_{x}^{2} + x^{n}$ and $- \partial_{x}^{2} + i x^{n}$ act on $L^{2} ( \R )$. In particular, the number of such eigenvalues is always infinite thanks to Proposition \ref{a65} and Theorem \ref{a66}. The case of Morse functions $V \in C^{2} ( \overline{X} )$ falls within this framework with $n = 2$ and $V_{c}$ being the Taylor expansion of $V$ to order $2$ at the critical point $c$, after a possible unitary (linear) change of variables. In that case taking the real part of equation \eqref{b36}, one recovers the result of Theorem 4.1.3 in \cite{He13_02}. It is also possible to generalize \eqref{a86} in the spirit of \eqref{a84} considering
\begin{equation*}
V_{c} ( x ) = \sum_{j = 1}^{d} \big( v^{-}_{j} \one_{x_{j} < 0} + v^{+}_{j} \one_{x_{j} \geq 0} \big) \vert x_{j} \vert^{\alpha_{c}} ,
\end{equation*}
for some $\alpha_{c} > 1$ and $v_{j}^{\pm} \in \R^{*}$.
\end{example}

Up to now, we have only considered the spectrum of $P$ near a particular energy $E$. But, if $V$ satisfies Assumption~\ref{h1} for all $E \in V ( \overline{X} )$, it is possible to describe all the eigenvalues of $P$ closest to the imaginary axis using Theorem \ref{a7}, Remark \ref{a94} and a compactness argument. That the previous results hold in a $h$-independent interval $[ E - \varepsilon , E + \varepsilon ]$ in imaginary part is crucial here.

\begin{figure}
\begin{center}
\begin{picture}(0,0)%
\includegraphics{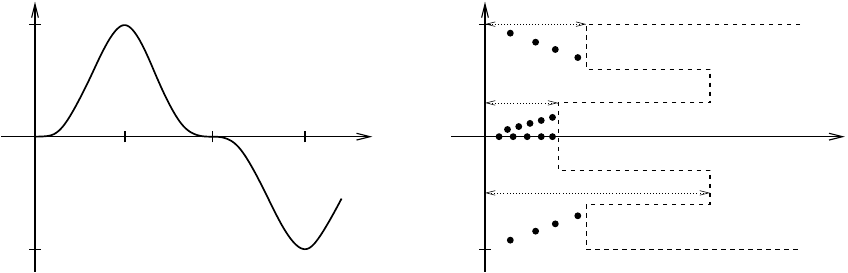}%
\end{picture}%
\setlength{\unitlength}{1184sp}%
\begingroup\makeatletter\ifx\SetFigFont\undefined%
\gdef\SetFigFont#1#2#3#4#5{%
  \reset@font\fontsize{#1}{#2pt}%
  \fontfamily{#3}\fontseries{#4}\fontshape{#5}%
  \selectfont}%
\fi\endgroup%
\begin{picture}(22566,7266)(-12932,-2794)
\put(676,4139){\makebox(0,0)[lb]{\smash{{\SetFigFont{9}{10.8}{\rmdefault}{\mddefault}{\updefault}$h$}}}}
\put(676,1964){\makebox(0,0)[lb]{\smash{{\SetFigFont{9}{10.8}{\rmdefault}{\mddefault}{\updefault}$h^{6 / 5}$}}}}
\put(676,-436){\makebox(0,0)[lb]{\smash{{\SetFigFont{9}{10.8}{\rmdefault}{\mddefault}{\updefault}$h^{2 / 3}$}}}}
\put(-12299,-2236){\makebox(0,0)[rb]{\smash{{\SetFigFont{9}{10.8}{\rmdefault}{\mddefault}{\updefault}$- 1$}}}}
\put(-12299,3764){\makebox(0,0)[rb]{\smash{{\SetFigFont{9}{10.8}{\rmdefault}{\mddefault}{\updefault}$1$}}}}
\put(-12299,1064){\makebox(0,0)[rb]{\smash{{\SetFigFont{9}{10.8}{\rmdefault}{\mddefault}{\updefault}$0$}}}}
\put(-299,-2236){\makebox(0,0)[rb]{\smash{{\SetFigFont{9}{10.8}{\rmdefault}{\mddefault}{\updefault}$- 1$}}}}
\put(-299,3764){\makebox(0,0)[rb]{\smash{{\SetFigFont{9}{10.8}{\rmdefault}{\mddefault}{\updefault}$1$}}}}
\put(-299,1064){\makebox(0,0)[rb]{\smash{{\SetFigFont{9}{10.8}{\rmdefault}{\mddefault}{\updefault}$0$}}}}
\put(-7199,239){\makebox(0,0)[b]{\smash{{\SetFigFont{9}{10.8}{\rmdefault}{\mddefault}{\updefault}$\pi$}}}}
\put(-9599,239){\makebox(0,0)[b]{\smash{{\SetFigFont{9}{10.8}{\rmdefault}{\mddefault}{\updefault}$\pi / 2$}}}}
\put(-4724,239){\makebox(0,0)[b]{\smash{{\SetFigFont{9}{10.8}{\rmdefault}{\mddefault}{\updefault}$3 \pi / 2$}}}}
\end{picture}%
\end{center}
\caption{The potential $V$ and the eigenvalues of $P$ in Example \ref{a90}.} \label{f2}
\end{figure}
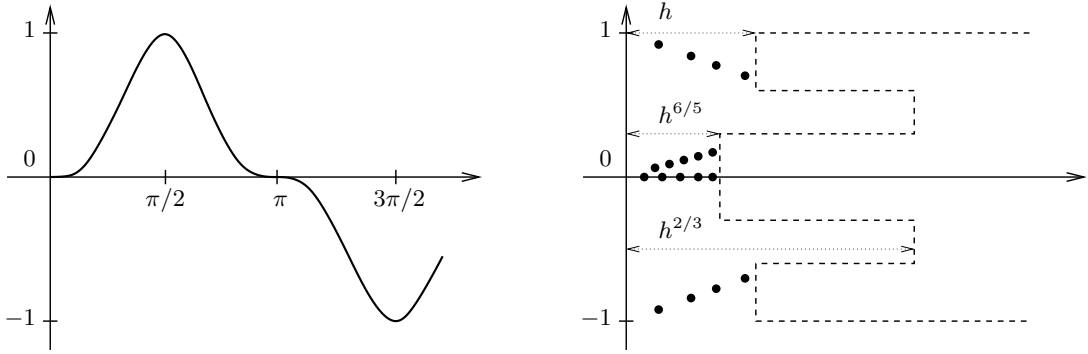

\begin{example}\rm \label{a90}
Consider $V ( x ) = \sin^{3} ( x )$ on $X = [ 0 , 5]$. This smooth potential has $4$ critical points: $c_{1} = 0$, $c_{2} = \pi / 2$, $c_{3} = \pi$ and $c_{4} = 3 \pi / 2$ of order $\alpha_{1} = \alpha_{3} = 3$ and $\alpha_{2} = \alpha_{4} = 2$ corresponding to the energies $E_{1} = E_{3} = 0$, $E_{2} = 1$ and $E_{4} = - 1$. We can apply Theorem \ref{a7}, Proposition \ref{a65} and Theorem \ref{a66} to compute the eigenvalues of $P = - h^{2} \Delta + i V ( x )$ near the imaginary axis (note that $c_{1} \in \partial X$). Modulo lower order terms, they are given by
\begin{equation*}
\left\{ \begin{aligned}
&e^{i \pi / 5} \Lambda_{3}^{D} h^{6 / 5} \cup \Lambda_{3}^{i} h^{6 / 5}   \qquad \qquad  &&\text{near } 0 , \\
&1+ e^{- i \pi / 4} \sqrt{\frac{3}{2}} ( 2 \N + 1 ) h   &&\text{near } 1 , \\
&- 1+ e^{i \pi / 4} \sqrt{\frac{3}{2}} ( 2 \N + 1 ) h   &&\text{near } - 1 ,
\end{aligned} \right.
\end{equation*}
where $\Lambda_{3}^{D} , \Lambda_{3}^{i} \subset ] 0 , + \infty [$ is the (discrete) spectrum of $- \Delta + x^{3}$ on $\R^{+}$ with Dirichlet boundary condition and $- \Delta + i x^{3}$ on $\R$ respectively. Moreover, $P$ has no spectrum in a neighborhood of size $h^{2 / 3}$ of any energy in $\R \setminus \{ - 1 , 0 , 1 \}$ (see Figure \ref{f2}).
\end{example}

We have stated our results for $\overline{X}$ compact, but it seems possible to extend them to $X = \R^{d}$. In that case, we would have to make assumptions on $V$ at infinity. One approach to compute the eigenvalues near an energy $E$ is to assume that
\begin{equation} \label{b32}
\liminf_{x \to \infty} V ( x ) > E \qquad \text{or} \qquad \limsup_{x \to \infty} V ( x ) < E .
\end{equation}
Using hypocoercivity techniques (more precisely, considering $\re \< u , e^{\mp i \varepsilon} ( P - z ) u \>$ with $\varepsilon > 0$ small enough), it has been proved in \cite[Proposition 3.5]{AvFrHeRa25_01} that the spectrum of $P$ is then discrete near $E$. The problem is that \eqref{b32} does not allow to treat potentials that change sign at infinity (for instance, $x^{3}$ on $\R$ or $x_{1}^{2} - x_{2}^{2}$ on $\R^{2}$) and that it seems difficult to get good resolvent estimates for $( P - z)^{- 1}$ as $\vert \im z \vert \to + \infty$ used in the applications (see Section \ref{s6}). Another approach, more in the spirit of our paper, is to suppose that
\begin{equation} \label{b33}
\lim_{x \to \infty} \vert \nabla V ( x ) \vert = + \infty .
\end{equation}
This should allow us to treat the operators $P$ in dimension $d = 1$. However, in higher dimensions, it seems possible to construct a potential $V \in C^{1} ( \R^{2} )$ satisfying \eqref{b33} such that $P$ does not have a compact resolvent. To bypass this obstruction, we could assume, mimicking Assumption~\ref{h1}, that
\begin{equation} \label{b34}
\nabla \big( V ( x ) - V_{\infty} ( x ) \big) = o \big( \vert x \vert^{\alpha_{\infty} - 1} \big) ,
\end{equation}
as $x$ goes to infinity where $V_{\infty} ( x ) = v_{\infty} ( \theta ) r^{\alpha_{\infty}}$ in spherical coordinates $x = r \theta$, $\alpha_{\infty} > 1$, $v_{\infty} \in C^{1} ( \S^{d - 1} )$ and $V_{\infty}$ satisfies a non-degeneracy assumption similar to \eqref{a81}. Under \eqref{b34} and adapting Sections \ref{s7} and \ref{s8}, one could prove that $P$ is maximal accretive, has a compact resolvent, enjoys good resolvent estimates as $\vert \im z \vert \to + \infty$ and get the asymptotic of its leftmost eigenvalues. To keep the presentation concise, we give no precise statement in this direction.

Eventually, we give an example which shows that the shape of $V$ is not enough to get the asymptotic of the leftmost eigenvalues.

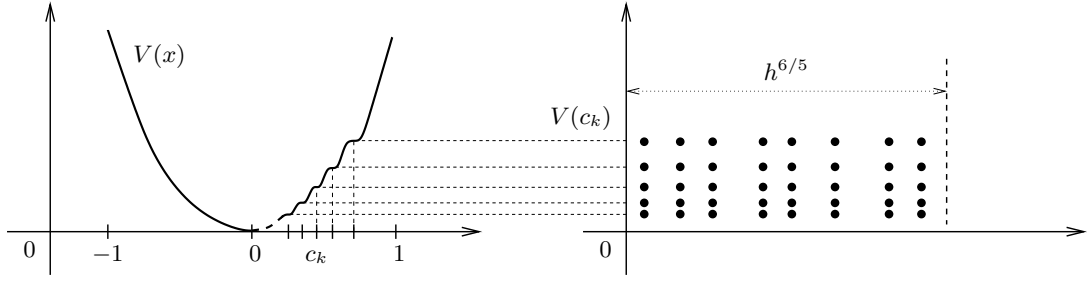
\begin{figure}
\begin{center}
\begin{picture}(0,0)%
\includegraphics{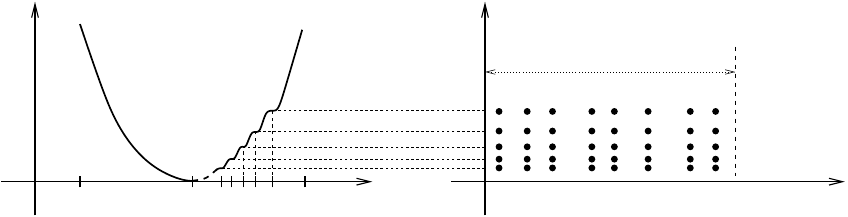}%
\end{picture}%
\setlength{\unitlength}{1184sp}%
\begingroup\makeatletter\ifx\SetFigFont\undefined%
\gdef\SetFigFont#1#2#3#4#5{%
  \reset@font\fontsize{#1}{#2pt}%
  \fontfamily{#3}\fontseries{#4}\fontshape{#5}%
  \selectfont}%
\fi\endgroup%
\begin{picture}(22566,5766)(-12932,-1894)
\put(-10799,-1561){\makebox(0,0)[b]{\smash{{\SetFigFont{9}{10.8}{\rmdefault}{\mddefault}{\updefault}$- 1$}}}}
\put(-7724,-1561){\makebox(0,0)[b]{\smash{{\SetFigFont{9}{10.8}{\rmdefault}{\mddefault}{\updefault}$0$}}}}
\put(-12299,-1486){\makebox(0,0)[rb]{\smash{{\SetFigFont{9}{10.8}{\rmdefault}{\mddefault}{\updefault}$0$}}}}
\put(-4724,-1561){\makebox(0,0)[b]{\smash{{\SetFigFont{9}{10.8}{\rmdefault}{\mddefault}{\updefault}$1$}}}}
\put(-6449,-1561){\makebox(0,0)[b]{\smash{{\SetFigFont{9}{10.8}{\rmdefault}{\mddefault}{\updefault}$c_{k}$}}}}
\put(-299,-1486){\makebox(0,0)[rb]{\smash{{\SetFigFont{9}{10.8}{\rmdefault}{\mddefault}{\updefault}$0$}}}}
\put(-299,1289){\makebox(0,0)[rb]{\smash{{\SetFigFont{9}{10.8}{\rmdefault}{\mddefault}{\updefault}$V ( c_{k} )$}}}}
\put(3339,2189){\makebox(0,0)[b]{\smash{{\SetFigFont{9}{10.8}{\rmdefault}{\mddefault}{\updefault}$h^{6 / 5}$}}}}
\put(-10274,2564){\makebox(0,0)[lb]{\smash{{\SetFigFont{9}{10.8}{\rmdefault}{\mddefault}{\updefault}$V ( x )$}}}}
\end{picture}%
\end{center}
\caption{The potential $V$ and the eigenvalues of $P$ in Example \ref{b31} with $V^{( 3 )} ( c_{k} )$ independent of $k$.} \label{f6}
\end{figure}

\begin{example}\rm \label{b31}
In $X = [ - 1 , 1 ]$, we consider a potential $V \in C^{1} ( \overline{X} )$ as in Figure $\ref{f6}$. In particular, we assume that $x^{2} \leq V ( x ) \leq x^{2} + x^{4}$ and then
\begin{equation*}
V ( x ) = x^{2} + o ( x^{2} ) ,
\end{equation*}
near $0$ (i.e. $\alpha_{0} = 2$). Suppose moreover that $V$ has an infinite number of critical points $c_{k}$ of order $\alpha_{c_{k}} = 3$ with $c_{k} \to 0$ as $k \to + \infty$. Applying Theorem \ref{a7} at each critical energy $V ( c_{k} )$, it follows that $P$ has eigenvalues at distance $h^{6 / 5}$ from the imaginary axis in any neighborhood of size $1$ of $0$. Then the eigenvalues of $P$ near $0$ are not close to the ones of $- h^{2} \Delta + i x^{2}$. This shows that \eqref{a57} is not enough to get the asymptotic of the leftmost eigenvalues.
\end{example}

In this paper, we have only considered purely imaginary potentials. It should be possible to adapt the present approach to the case of complex valued potentials $V$. If $\min \re V = 0$, the asymptotic of the eigenvalues closest to the imaginary axis should be given by the points $c \in \overline{X}$ such that $\re V ( c) = 0$ and $\im V^{\prime} ( c ) = 0$ (and then $V^{\prime} ( c ) = 0$ if $c \in X$).

\subsection{Spectrum of the model operator} \label{s2}

This part is devoted to the model operators appearing in Section \ref{s4}. More precisely, we collect some spectral informations on the operator acting on $L^{2} ( \R^{d} )$
\begin{equation} \label{a44}
\CP = - \Delta + i \CV ( x ) \qquad \text{ with } \qquad \CV ( x ) = v ( \theta ) r^{\alpha} ,
\end{equation}
in the spherical coordinates $x = r \theta$ with $\alpha > 1$, $v \in C^{1} ( \S^{d - 1} )$ and $\CV $ satisfying \eqref{a81}.

\begin{proposition}\sl \label{a46}
The operator $\CP$ of \eqref{a44} with domain $C^{\infty}_{0} ( \R^{d} )$ is accretive and closable. Its closed extension, still denoted $\CP$, is maximal accretive and has a compact resolvent. Moreover, $\sigma ( \CP ) \subset \{ z \in \C ; \ \re z > 0 \}$.
\end{proposition}

Since $\CP$ is a non-selfadjoint operator, it is not clear that it has at least one eigenvalue. For instance, $- \partial_{x}^{2} + i x$ on $L^{2} ( \R )$ is known to have an empty spectrum. Nevertheless, when $\CV$ has a sign, it is possible to compute the spectrum of $\CP$ and show that it has an infinite number of eigenvalues. In the sequel, we will say that ``$\CV$ has a sign'' if $v$ is globally positive or globally negative on $\S^{d - 1}$.

\begin{proposition}\sl \label{a65}
Let $\CP$ be given by \eqref{a44} where $\CV$ has a sign, and let $\pm$ denote the sign of $v$. Then, the domain of $\CP$ is $D ( \CP ) = \{ u \in H^{2} ( \R^{d} ) ; \ \vert x \vert^{\alpha} u \in L^{2} ( \R^{d} ) \}$ and
\begin{equation*}
\sigma ( \CP ) = e^{\pm i \frac{\pi}{\alpha + 2}} \sigma ( \CQ ) ,
\end{equation*}
where $\CQ = - \partial_{x}^{2} \pm \CV$ is selfadjoint and positive on $D ( \CP )$ with compact resolvent.
\end{proposition}

In dimension $1$, we have $D ( \CP ) = \{ u \in H^{2} ( \R ) ; \ \vert x \vert^{\alpha} u \in L^{2} ( \R ) \}$ even if $v^{+}_{c}$ and $v^{-}_{c}$ do not have the same sign (see \eqref{a84}). Technically, it follows from \eqref{a71} which is always valid in that case. If the critical point $c$ belongs to the boundary of $X$, the corresponding operator $\CP$ is defined on $\R^{\pm}$ with the same boundary condition at $0$ than $P$ at $c$, the potential $\CV$ has a sign and the conclusions of Proposition \ref{a65} hold with the same proof mutatis mutandis.

In the previous result, the eigenvalues of $\CP$ and $\CQ$ have the same algebraic multiplicity. In dimension $d \geq 2$, $\CV ( x )$ has a sign iff $v$ does not vanish on $\S^{d - 1}$. Proposition \ref{a65} is proven using complex dilations. The case of the complex harmonic oscillator $- \Delta + i x^{2}$ is well known (see e.g. Section 2.5 of \cite{Sj19_01} and the references therein). When $\CV$ does not have a sign, it is no longer possible to use the complex dilations and we do not have a result like Proposition \ref{a65}. Nevertheless, if
\begin{equation} \label{a75}
\CP = - \partial_{x}^{2} + i v x^{n} ,
\end{equation}
on $L^{2} ( \R )$ with $n \in \N$ odd, \cite{Ar25_01} provides the following result. The case $n = 3$ had already been treated by Sibuya (see Chapter 6 of \cite{Si75_01} and the references therein) and the eigenvalues are known to be all real in that case (see Shin \cite{Sh02_01}). By symmetry, $- \partial_{x}^{2} + i v x^{n}$ and $- \partial_{x}^{2} - i v x^{n}$ have the same spectrum.

\begin{theorem}\sl \label{a66}
Let $\CP$ acting on $L^{2} ( \R )$ be as in \eqref{a75} with $n \geq 3$ odd and $v \in \R^{*}$. Then, there exist $R_{n , v} , q_{n , v} > 0$ such that
\begin{equation*}
\sigma ( \CP ) \cap B ( 0 , R_{n, v} )^{c} = \{ \lambda_{q} ; \ q \geq q_{n , v} \} ,
\end{equation*}
where the eigenvalues $\lambda_{q}$ satisfy the asymptotic
\begin{equation} \label{a67}
\lambda_{q} = \vert v \vert^{\frac{2}{n + 2}} \Big( \frac{ ( 2 q + 1 ) \pi}{K \nu ( n , \mathbf{k} )} \Big)^{\frac{2 n}{n + 2}}( 1 + o ( 1 ) ) ,
\end{equation}
as $q \to + \infty$, with
\begin{equation} \label{a68}
\nu ( n , \mathbf{k} ) = 2 \sin \Big( \frac{2 \pi \mathbf{k}}{n} \Big) + 2 \sin \Big( \frac{2 \pi (\mathbf{k}-1 )}{n} \Big) ,
\end{equation}
$\mathbf{k}$ is the unique integer such that $\frac{n}{4} < \mathbf{k} < \frac{n}{4} + 1$ and
\begin{equation*}
K = \int_{0}^{+ \infty} t^{\frac{n}{2}} \big( ( 1 + t^{-n} )^{\frac{1}{2}} - 1 \big) d t > 0 .
\end{equation*}
Moreover, these eigenvalues are simple and real.
\end{theorem}

Finally, the resolvent of $\CP$ satisfies the following estimates, which will be useful in the proof of Theorem \ref{a7}.

\begin{proposition}[Resolvent estimates]\sl \label{a47}
Let $\CP$ be given by \eqref{a44}.

$i)$ For all $M > 0$, there exist $C , K > 0$ such that $\CP$ has no spectrum and
\begin{equation} \label{b88}
\Vert \partial_{x} ( \CP - z )^{- 1} \Vert + \Vert ( \CP - z )^{- 1} \Vert \leq C ,
\end{equation}
in $\{ z \in \C; \ \re z \leq M , \ \vert z \vert > K \}$.

$ii)$ For all $K > 0$, there exist $C , N > 0$ such that
\begin{equation} \label{b55}
\Vert \partial_{x} ( \CP - z )^{- 1} \Vert + \Vert ( \CP - z )^{- 1} \Vert \leq C + \frac{C}{\dist ( z , \sigma ( \CP ) )^{N}} ,
\end{equation}
for all $z \in B ( 0 , K ) \setminus \sigma ( \CP )$.
\end{proposition}

One could also prove that $z \mapsto \Vert ( \CP - z )^{- 1} \Vert$ decays polynomially on the vertical lines. Using the relation \eqref{a32}, Proposition \ref{a47} implies the following resolvent estimate for the semiclassical operator $P_{c}$ introduced in \eqref{a48}.

\begin{corollary}\sl \label{a49}
For all $M , \gamma > 0$, there exists $C > 0$ such that
\begin{equation*}
\Vert ( P_{c} - z )^{- 1} \Vert + h^{\beta} \Vert \partial_{x} ( P_{c} - z)^{- 1} \Vert \leq C h^{- \sigma} ,
\end{equation*}
for $h > 0$ and $z \in ( [ 0 , M h^{\sigma} ] + i \R ) \setminus ( i E + h^{\sigma} \Lambda_{c} + B ( 0 , \gamma h^{\sigma} ) )$.
\end{corollary}

\subsection{Application to time evolution} \label{s6}

We give in this section an application of our main theorem to the heat equation associated to $P$. We assume that $X$ is as in Section \ref{s4} and we let $P$ be as in \eqref{a12} with either Dirichlet or Neumann boundary condition when $X$ is a bounded subset of $\R^{d}$. Consider the evolution equation in $X$ for $t \geq 0$
\begin{equation} \label{b37}
\partial_{t} u + P u = 0 ,
\end{equation}
with initial data $u_{\vert t = 0 } = u_{0} \in L^{2} ( X )$. Since $P$ is maximal accretive, this equation admits a unique solution $u ( t ) = e^{- t P } u_{0}$. In the next theorem, we assume that Assumption~\ref{h1} holds true for all $E \in \R$. Since $V$ is bounded, the discussion below Remark \ref{a85} implies that the set $\CC^{\rm glob} = \cup_{E \in \R} \CC ( E )$ of critical points of $V$ is finite. Recall that $\alpha_{c} > 1$ is given in Assumption~\ref{h1} for any $c \in \CC^{\rm glob} $. We denote $\boldsymbol{\alpha}=\max \{ \alpha_{c} ; \ c \in \CC^{\rm glob} \}$, $\boldsymbol{\sigma} = \frac {2 \boldsymbol{\alpha}}{\boldsymbol{\alpha} + 2}$, $\boldsymbol{\CC} = \{ c \in \CC^{\rm glob} ; \ \alpha_{c}=\boldsymbol{\alpha}\}$ and $\boldsymbol{\CE} = \{ V ( c ) ; \  c \in \boldsymbol{\CC} \}$ some quantities corresponding to the most degenerate critical points of $V$. Note that $\CC_{\max} ( E ) = \boldsymbol{\CC} \cap V^{- 1} ( E )$ for $E \in \boldsymbol{\CE}$. We define for $E \in \boldsymbol{\CE}$
\begin{equation*}
\boldsymbol{\Lambda} ( E ) = \cup_{c \in \CC_{\max} ( E )} \Lambda_c \qquad \text{and} \qquad \boldsymbol{\Lambda} = \cup_{c \in \boldsymbol{\CC}} \Lambda_{c} = \cup_{E \in \boldsymbol\CE} \boldsymbol{\Lambda} ( E ) ,
\end{equation*}
the union of the spectra of the model operators $\CP_{c}$ associated to the most degenerate critical points respectively at the energy $E$ or globally. The set $\re \boldsymbol{\Lambda}$ is either finite or consists of an increasing sequence of positive numbers which goes to $+ \infty$. For $M \notin  \re \boldsymbol{\Lambda}$, we define
\begin{equation*}
\boldsymbol{\Lambda}^{M} ( E ) = \boldsymbol{\Lambda} ( E ) \cap \{ \re z < M \} \qquad \text{and} \qquad r_{M} = \min \{ \re z ; \ z \in \boldsymbol{\Lambda} , \ \re z > M \} ,
\end{equation*}
with the convention that $r_{M}$ is any number greater than $M$ if $\{ z \in \boldsymbol{\Lambda} ; \ \re z > M \}  = \emptyset$. Observe that thanks to Proposition \ref{a6}, the sets $\boldsymbol{\Lambda}^{M} ( E )$ are finite. For any $\mu \in \boldsymbol{\Lambda}^{M} ( E )$, we denote by $\gamma_{E , \mu} = \frac{1}{4} \min_{\mu^{\prime} \in \boldsymbol{\Lambda}^{M} ( E ) , \; \mu^{\prime} \neq \mu} \vert \mu^{\prime} - \mu \vert > 0$ (with the convention $\gamma_{E , \mu} = 1$ if $\boldsymbol{\Lambda}^{M} ( E ) \setminus \{ \mu \}$ is empty) and we define the spectral projector 
\begin{equation} \label{b39}
\Pi_{E , \mu} =\frac{1}{2 i \pi} \oint_{\vert z - i E - h^{\boldsymbol{\sigma}} \mu \vert =\gamma_{E , \mu} h^{\boldsymbol{\sigma}}} ( z - P )^{- 1} d z .
\end{equation}
Thanks to Theorem \ref{a7}, for $h > 0$ small enough $\Pi_{E , \mu}$ has finite rank equal to the sum over $c \in \CC_{\max} ( E )$ of the algebraic multiplicity of $\mu$ as an eigenvalue of $\CP_{c}$. Moreover, it follows from \eqref{a8}, that there exists $C > 0$ such that $\Vert \Pi_{E , \mu} \Vert \leq C$ for $h > 0$ small enough.

\begin{theorem}[eigenmode expansion]\sl \label{b40}
Suppose that Assumption~\ref{h1} holds true for all $E \in \R$ and let $M \notin \re \boldsymbol{\Lambda}$. Then, for all $\varepsilon>0$, one has uniformly for $h$ small enough and $t \geq 0$
\begin{equation} \label{b41}
e^{- t P} = \sum_{E \in \boldsymbol{\CE} , \; \mu \in\boldsymbol{\Lambda}^{M} ( E )} e^{- t P} \Pi_{E , \mu} + \CO \big( e^{-t h^{\boldsymbol{\sigma}} ( r_{M} - \varepsilon )} \big) ,
\end{equation}
as operators on $L^{2} ( X )$. Moreover, for all $f \in L^{2} ( X )$, there exist some $f_{E , \mu , \lambda , s} \in L^{2} ( X )$ such that
\begin{equation} \label{b42}
e^{- t P} f = \sum_{E , \mu , \lambda} e^{- \lambda t} \sum_{0 \leq s \leq m ( \lambda ) - 1} t^{s} f_{E , \mu , \lambda , s} + \CO \big( e^{-t h^{\boldsymbol{\sigma}} ( r_{M} - \varepsilon )} \big) ,
\end{equation}
uniformly for $h$ small enough and $t \geq 0$, where $m ( \lambda )$ denotes the algebraic multiplicity of $\lambda \in \sigma ( P )$ and the first sum is taken over $E \in \boldsymbol{\CE} , \mu \in \boldsymbol{\Lambda}^{M} ( E ) , \lambda \in \sigma ( P ) \cap B_{E , \mu}$ with $ B_{E , \mu} = B ( i E + h^{\boldsymbol{\sigma}} \mu , h^{\boldsymbol{\sigma}} \gamma_{E , \mu} )$. Finally, one has 
\begin{equation} \label{b43}
f_{E , \lambda , \mu , s} = \Big( \sum_{c \in \boldsymbol{\CC}} \varphi_{c} \Pi_{c} \varphi_{c}  + o ( 1 ) \Big) f_{E , \lambda , \mu , s} ,
\end{equation}
where the cut-off function $\varphi_ {c}$ is defined in Section \ref{s9} and the spectral projector $\Pi_{c}$ of $P_{c}$ associated to the eigenvalue $i E + h^{\boldsymbol{\sigma}} \mu$ is given by \eqref{a37}.
\end{theorem}

Let us discuss the results of this theorem. First, equation \eqref{b41} with $M = 0$ implies that the solution of \eqref{b37} decays exponentially to zero with decay rate given by a spectral gap. More precisely, for all $\varepsilon > 0$, one has for $h$ small enough
\begin{equation} \label{b44}
\Vert e^{- t P} f \Vert \leq C e^{- t h^{\boldsymbol{\sigma}} ( \mu_{0} - \varepsilon )} \Vert f \Vert ,
\end{equation}
where $\mu_{0} = \inf_{\mu \in \boldsymbol{\Lambda}} \re \mu$ and $C > 0$ is a constant independent of $h$. In the case where $\boldsymbol{\Lambda}$ is non-empty, this rate of decay is sharp as shown by taking $f \in L^{2} ( X )$ a normalized eigenfunction of $P$ associated to an eigenvalue of minimal real part. Similar sharpness results were already proved in \cite{AlHe16_01} for non-critical potentials and \cite{CoGa23_01}, \cite{He13_02} for Morse potentials.

For some $M \notin \re \boldsymbol{\Lambda}$, assume that the projector $\Pi_{E , \mu}$ is a rank one operator for all $E \in \boldsymbol{\CE}$ and $\mu \in \boldsymbol{\Lambda}^{M} ( E )$. Then, $\sigma ( P ) \cap B_{E , \mu}$ is made of one simple eigenvalue $\lambda ( \mu )$ and \eqref{b42} writes
 \begin{equation}
e^{- t P} f = \sum_{E , \mu} e^{- \lambda ( \mu ) t} f_{E , \mu} + \CO \big( e^{- t h^{\boldsymbol{\sigma}} ( r_{M} - \varepsilon )} \big) ,
\end{equation}
with $f_{E , \mu} = \Pi_{E , \mu} f$.

More generally, assume that $\sigma ( P ) \cap B_{E , \mu} $ is reduced to one single eigenvalue $\lambda$ of multiplicity $m ( \lambda ) \geq 1$ for some $E \in \boldsymbol{\CE}$ and $\mu \in \boldsymbol{\Lambda}^{M} ( E )$. Then,
\begin{equation*}
f_{E , \mu , \lambda , s} = \frac{1}{2 i \pi} \oint_{\vert z - i E - h^{\boldsymbol{\sigma}} \mu \vert = \gamma_{E , \mu} h^{\boldsymbol{\sigma}}} ( z - \lambda )^{s} ( z - P )^{- 1} f \, d z .
\end{equation*}
In particular, $\Vert f_{E , \mu , \lambda , s} \Vert \leq C h^{\boldsymbol{\sigma} s} \Vert f \Vert$ for $0 \leq s \leq m ( \lambda ) - 1$ since $\lambda = i E + h^{\boldsymbol{\sigma}} \mu + o ( h^{\boldsymbol{\sigma}} )$.

The above theorem gives also more precise informations on the dynamics of shear flows widely studied in the last decade (see e.g. \cite{AlBeNo22_01},  \cite{BeCoZ17_01}, \cite{We21_01}). Consider the equation
\begin{equation} \label{b45}
\partial_{t} u + V ( x ) \partial_{y} u - \nu \Delta_{x , y} u = 0 ,
\end{equation}
where the unknown function $u$ depends on $t \in \R$, $x \in\Pi^{d}$ and $y \in \Pi$ and $\Delta_{x , y}$ denotes the Laplace operator in the $(x , y)$-variables. Expanding $u$ in partial Fourier series in the $y$-variable yields to
\begin{equation} \label{B50}
\partial_{t} u_{k} + i k V ( x ) u_{k} - \nu ( \Delta_{x} - k^{2} ) u_{k} = 0 ,
\end{equation}
where $k \in \Z$ and $\Delta_x$ is the Laplace operator in the $x$-variable. Setting $h = \sqrt{\nu / \vert k \vert}$ and $\tau = \sgn k$, it follows from \eqref{b44} that for all $k \neq 0$ one has
\begin{equation*}
\big\Vert e^{t \nu k^{2}} u_{k} ( t ) \big\Vert \leq C e^{- t \mu_{0} ( \tau ) \nu^{\frac{\alpha}{\alpha + 2}} \vert k \vert^{\frac{2}{\alpha+2}} ( 1 - \varepsilon )} \Vert u_{k} ( 0 ) \Vert ,
\end{equation*}
where $\mu_{0} ( \tau ) = \inf_{\mu \in \boldsymbol{\Lambda}} \re \mu$, the set $\boldsymbol{\Lambda}$ being associated with the potential $\tau V$. This estimate allows to recover and precise the decay rate of solutions of \eqref{B50}. The time scale $\nu^{\frac{\alpha}{\alpha + 2}} \vert k \vert^{\frac{2}{\alpha + 2}}$ was already obtained in \cite{AlBeNo22_01}, \cite{BeCoZ17_01}, \cite{CoGa23_01}, \cite{We21_01}. Here we obtain the precise prefactor $\mu_{0} ( \tau )$ in terms of spectral quantities. Moreover, by taking $u_{k} ( 0 )$ equal to any eigenvector of $- h^{2} \Delta_{x} + i \tau V ( x )$ (when it exists), we obtain immediately that this rate of decay is sharp (see \cite{CoZDr21_01} for similar result by probabilistic approach). Of course, one could also use Theorem \ref{b40} to get an eigenmode expansion of solutions of \eqref{b45}.

\begin{proof}[Proof of Theorem \ref{b40}]
Let us denote $\Pi_{\leq M} = \sum_{E \in \boldsymbol{\CE} , \; \mu \in \boldsymbol{\Lambda}^{M} ( E )} \Pi_{E , \mu}$ and observe that by definition of $\gamma_{E , \mu}$, the ball $B_{E , \mu}$ are disjoints and hence $\Pi_{\leq M}$ is a projector. We write 
\begin{equation} \label{b46}
e^{- t P} = e^{- t P} \Pi_{\leq M} + e^{- t P} ( 1 - \Pi_{\leq M} ) = \sum_{E \in \boldsymbol{\CE} , \; \mu \in \boldsymbol{\Lambda}^{M} ( E )} e^{- t P} \Pi_{E , \mu} + e^{- t P} ( 1 -\Pi_{\leq M} ) .
\end{equation}
To estimate the last term in \eqref{b46}, we use a Gearhart--Pr\"{u}ss type inequality with an explicit bound. Let $Q : \im ( 1 - \Pi_{\leq M} )\to \im ( 1 - \Pi_{\leq M} )$ be the operator $P$ restricted to the Hilbert space $\im ( 1 - \Pi_{\leq M} )$. Since $P$ is maximal accretive, so is $Q$ and one has $e^{- t P} ( 1 - \Pi_{\leq M} ) = e^{- t Q} ( 1 - \Pi_{\leq M} )$ and $\Vert e^{- t Q } \Vert \leq 1$. From Theorem \ref{a7} and by definition of $r_{M}$, one knows that for $0 < \varepsilon < ( r_{M} - M ) / 8$ and $h > 0$ small enough, one has $\sigma ( P ) \cap \{ h^{\boldsymbol{\sigma}} M \leq \re z \leq h^{\boldsymbol{\sigma}} ( r_{M} - \varepsilon ) \} = \emptyset$ and
\begin{equation*}
\sigma ( P ) \cap \{ \re z \leq h^{\boldsymbol{\sigma}} ( r_{M} - \varepsilon ) \} \subset \cup_{E \in \boldsymbol{\CE} ,\; \mu \in \boldsymbol{\Lambda}^{M} ( E )} B_{E , \mu} .
\end{equation*}
Moreover, denoting $\Sigma_{M , \varepsilon} = \{ \re z \leq h^{\boldsymbol{\sigma}} ( r_{M} - \varepsilon ) \}\setminus \cup_{E \in \boldsymbol{\CE} , \; \mu \in \boldsymbol{\Lambda}^{M} ( E )} B_{E , \mu}$, it follows from \eqref{a8} that there exists $C > 0$ such that 
\begin{equation} \label{b47}
\forall z \in \Sigma_{M,\varepsilon} , \qquad \Vert ( P - z )^{- 1} \Vert \leq C h^{- \boldsymbol{\sigma}} .
\end{equation}
This follows actually from \eqref{a8} for $\im z$ near $V ( \overline{X} )$ and from $\Vert ( P - z )^{- 1}\Vert\leq C$ for $\im z$ outside any neighborhood of $V ( \overline{X} )$. In particular, we get
\begin{equation*}
\forall z \in\Sigma_{M , \varepsilon} , \qquad \Vert ( Q - z )^{- 1} \Vert \leq C h^{- \boldsymbol{\sigma}} .
\end{equation*}
Since the application $z \mapsto ( Q - z)^{- 1}$ is holomorphic on $\{ \re z < h^{\boldsymbol{\sigma}} ( r_{M} - \varepsilon ) \}$, the maximum principle implies that
\begin{equation} \label{b48}
\forall \re z \leq h^{\boldsymbol{\sigma}} ( r_{M} - \varepsilon ) , \qquad \Vert ( Q - z )^{- 1} \Vert \leq C h^{-\boldsymbol{\sigma}} .
\end{equation}
Hence, we deduce from \cite[Theorem 1.4]{HeSj21_01} (see \cite[Proposition 2.1]{HeSj10_01} for more details) that for some $C > 0$ and all $t \geq 0$,
\begin{equation*}
\big\Vert e^{- t Q } \big\Vert \leq C \Big( 1 + C h^{\boldsymbol{\sigma}} \sup_{\re z = h^{\boldsymbol{\sigma}} ( r_{M} - \varepsilon )} \Vert ( Q - z )^{- 1} \Vert \Big) e^{- t h^{\boldsymbol{\sigma}} ( r_{M} - \varepsilon )} .
\end{equation*}
 Combined with \eqref{b48}, it implies the existence of $C>0$ such that, for all $t \geq 0$,
\begin{equation*} 
\big\Vert e^{- t P} ( 1 - \Pi_{\leq M} ) \big\Vert \leq C e^{- t h^{\boldsymbol{\sigma}} ( r_{M} - \varepsilon )} ,
\end{equation*}
which completes the proof of the first estimate. The second expansion follows directly from the usual formula for the exponential of a matrix applied to $e^{- t P} \Pi_{\leq M}$. Of course, in this expansion, one has $f_{E , \mu , \lambda , s} \in \im \Pi_{E , \mu}$. Hence, \eqref{b43} follows from \eqref{a38} below.
\end{proof}

\noindent
\textbf{Acknowledgements.} The authors are grateful to F. Sueur and M. Zworski for helpful discussions and for providing references on shear flows. The authors are supported by ANR projects SpecDiMa and DySLoS. The first author is also supported by ANR project STENTOR.

\section{Asymptotic of the eigenvalues}

\subsection{Construction of an auxiliary operator}

The general strategy to prove Theorem \ref{a7} is to use a parametrix for $( P - z )^{- 1}$ with $z$ at distance $h^{\sigma}$ from the spectrum (see \cite[Section 4.4]{He13_02} for a similar idea in this context), we could also have used the famous Grushin problem method. In this part, we construct an auxiliary operator and show that it has no spectrum near the imaginary axis. Let $R \geq 1$ and $g$ be a function in $C_{0}^{\infty} ( \R ; [ 0 , 1 ] )$ such that $g = 1$ on $[- 1 , 1 ]$ and $\supp g \subset ] - 2 , 2 [$. We set, for $c \in \CC_{\max} ( E )$,
\begin{equation} \label{a9}
g_{c} ( x ) = g \Big( \frac{x - c}{R h^{\beta}} \Big) \qquad \text{ and } \qquad G ( x ) = \sum_{c \in \CC_{\max} ( E )} g_{c} ( x ) .
\end{equation}
If $X = \T^{1}$, we identify $\T^{1}$ with $\R$ in local coordinates near $c$ to construct $g_{c}$. We define the auxiliary operator $Q$ by 
\begin{equation}\label{a13}
Q = P + G^{2} ,
\end{equation}
with domain $D ( Q ) = D ( P )$. The idea is that adding the function $G^{2}$ makes the operator elliptic near the critical points of maximal vanishing. Then, $Q$ should have better resolvent estimates than $P$. This is the subject of the following proposition which is an adaptation of Section 2 of \cite{CoGa23_01} and whose proof is postponed to Section \ref{s1}.

\begin{proposition}\sl \label{a33}
Suppose that Assumption~\ref{h1} and \eqref{a9} hold true and that $\varepsilon > 0$ is small enough. Then, for all $M \geq 1$ and $R \geq R_{M}$ large enough, $Q$ has no spectrum in $\Omega = [ 0 , M h^{\sigma} ] + i [ E - \varepsilon , E + \varepsilon ]$ for $h$ small enough and
\begin{equation}
\Vert ( Q - z )^{- 1} \Vert + h^{\beta} \Vert \nabla ( Q - z )^{- 1} \Vert \leq M^{- 1} h^{- \sigma} ,
\end{equation}
for $z$ in that set.
\end{proposition}

\subsection{Construction of a parametrix} \label{s9}

We can now define a parametrix for the resolvent of the operator $P$. For this purpose, we consider a smooth cut-off $\varphi \in C_{0}^{\infty} ( \R^{d} ; [ 0 , 1 ] )$ with $\supp \varphi \subset B ( 0 , 3 )$ and $\varphi = 1$ on $B ( 0 , 2 )$, and a smooth function $\psi \in C^{\infty} ( \overline{X} ; [ 0 , 1 ] )$ which satisfy the following properties. For any $c \in \CC_{\max} ( E )$, we define $\varphi_{c} \in C^{\infty}_{0} ( X )$ by $\varphi_{c} ( x ) = \varphi ( \frac{x - c}{R h^{\beta}} )$. Then, we have 
\begin{equation}
\sum_{c \in \CC_{\max} ( E )} \varphi_{c}^{2} + \psi^{2} = 1 ,
\end{equation}
and, for all $k \in \N^{d}$,
\begin{equation} \label{a23}
\vert \partial_{x}^{k} \psi ( x ) \vert \leq C_{k} R^{- \vert k \vert} h^{- \vert k \vert \beta} .
\end{equation}
Let us observe that, with the above definition, one has $\varphi_{c} = 1$ on $\supp g_{c} $ and hence $\psi g_{c} = 0$. Then, it follows from Corollary \ref{a49} and Proposition \ref{a33} that, for all $M \geq 1$, $R \geq R_{M}$, $\gamma > 0$ and any $z \in \Omega \setminus \bigcup_{c \in \CC_{\max} ( E )} \bigcup_{\mu \in \Lambda_{c}} B ( i E + h^{\sigma} \mu , \gamma h^{\sigma} )$, the parametrix operator
\begin{equation} \label{a24}
R ( z ) = \sum_{c \in \CC_{\max} ( E )} \varphi_{c} ( P_{c} - z )^{- 1} \varphi_{c} + \psi ( Q - z )^{- 1} \psi ,
\end{equation}
is well defined and bounded on $L^{2} ( X )$ with values in $D ( P )$. The following result shows that $R ( z )$ is indeed a good parametrix for $( P - z )^{- 1}$ with $z$ near $E$. In the sequel, $o_{h \to 0}^{R} ( 1 )$ denotes a quantity which goes to $0$ as $h \to 0$ for $R$ fixed.

\begin{proposition}\sl \label{a26}
Suppose that Assumption~\ref{h1} holds true and that $\varepsilon > 0$ is small enough. For any $M \geq 1$, we have
\begin{equation*}
( P - z ) R ( z ) = \Id + o_{h \to 0}^{R} ( 1 ) + o_{R \to \infty} ( 1 ) \quad \text{ and } \quad R ( z ) ( P - z ) = \Id + o_{h \to 0}^{R} ( 1 )  + o_{R \to \infty} ( 1 ) ,
\end{equation*}
uniformly for $z \in \Omega \setminus \bigcup_{c \in \CC_{\max} ( E )} \bigcup_{\mu \in \Lambda_{c}} B ( i E + h^{\sigma} \mu , \gamma h^{\sigma} )$.
\end{proposition}

\begin{proof}
From \eqref{a24}, we have
\begin{align*}
( P - z ) R ( z ) ={}& \sum_{c \in \CC_{\max} ( E )} ( P - z ) \varphi_{c} ( P_{c} - z )^{- 1} \varphi_{c} + ( P - z ) \psi ( Q - z)^{- 1} \psi  \\
={}& \sum_{c \in \CC_{\max} ( E )} \big( \varphi_{c} ( P - z ) ( P_{c} - z)^{- 1} \varphi_{c} - [ h^{2} \Delta , \varphi_{c} ] ( P_{c} - z )^{- 1} \varphi_{c} \big) \\
&+ \psi ( P - z ) ( Q - z )^{- 1} \psi - [ h^{2} \Delta , \psi ] ( Q - z )^{- 1} \psi .
\end{align*}
Note that $\psi ( P - z ) ( Q - z )^{- 1} \psi = \psi ( P + G - z ) ( Q - z )^{- 1} \psi = \psi^{2}$ since $\psi G = 0$. Moreover, \eqref{a57} gives the identity $P = P_{c} + r_{c} ( x )$ with $r_{c} ( x ) = o ( \vert x - c \vert^{\alpha} )$ for all $c \in \CC_{\max} ( E )$. Thus, the last equation becomes
\begin{align}
( P - z ) R ( z ) &= \sum_{c \in \CC_{\max} ( E )} \varphi_{c}^{2} + \psi^{2} + \sum_{c \in \CC_{\max} ( E )} ( \varphi_{c} r_{c} - [ h^{2} \Delta , \varphi_{c} ] ) ( P_{c} - z )^{- 1} \varphi_{c}   \nonumber \\
&\hspace{210pt} - [ h^{2} \Delta , \psi ] ( Q - z)^{- 1} \psi  \nonumber  \\
&= 1 + \sum_{c \in \CC_{\max} ( E )} ( \varphi_{c} r_{c} - [ h^{2} \Delta , \varphi_{c} ] ) ( P_{c} - z )^{- 1} \varphi_{c} - [ h^{2} \Delta , \psi ] ( Q - z )^{- 1} \psi .  \label{a27}
\end{align}
For $c \in \CC_{\max} ( E )$, we first study $\varphi_{c} r_{c} ( P_{c} - z )^{- 1} \varphi_{c}$. Thanks to Corollary \ref{a49}, there exists $C_{M , \gamma} > 0$ such that $\Vert ( P_{c} - z )^{- 1} \varphi_{c} \Vert \leq C_{M , \gamma} h^{- \sigma}$. On the other hand, by \eqref{a19}, the definition of $\varphi_{c}$ and $r_{c} = o_{x \to c} ( \vert x - c \vert^{\alpha} )$, we have $\vert r_{c} \vert = o_{R^{\alpha} h^{\sigma} \to 0} ( R^{\alpha} h^{\sigma} ) = o_{h \to 0}^{R} ( h^{\sigma} )$ on the support of $\varphi_{c}$. Consequently,
\begin{equation} \label{a28}
\varphi_{c} r_{c} ( P_{c} - z )^{- 1} \varphi_{c}= o_{h \to 0}^{R} ( 1 ) .
\end{equation}
Since $[ h^{2} \Delta , \varphi_{c} ] = h^{2} \Delta \varphi_{c} + 2 h^{2} \nabla \varphi_{c} \cdot  \nabla$, the next step is to estimate
\begin{equation*}
h^{2} ( \Delta \varphi_{c} ) ( P_{c} - z )^{- 1} \varphi_{c} \qquad \text{ and } \qquad h^{2} \nabla \varphi_{c} \cdot \nabla ( P_{c} - z )^{- 1} \varphi_{c} .
\end{equation*}
Using $\vert \Delta \varphi_{c} \vert \leq C R^{- 2} h^{- 2 \beta} $, we get as before $h^{2} ( \Delta \varphi_{c} ) ( P_{c} - z )^{- 1} \varphi_{c} = \CO_{M , \gamma} ( R^{- 2} h^{2 - 2 \beta - \sigma} )$. Since $2 - 2 \beta - \sigma = 0$ from \eqref{a19}, it becomes
\begin{equation} \label{a29}
h^{2} ( \Delta \varphi_{c} ) ( P_{c} - z )^{- 1} \varphi_{c} = \CO_{M , \gamma} ( R^{- 2} ) = o_{R \to \infty} ( 1 ) .
\end{equation}
Similarly, it follows from Corollary \ref{a49} that 
\begin{equation*}
\Vert \nabla ( P_{c} - z )^{- 1} \varphi_{c} \Vert \leq C_{M , \gamma} h^{- \beta - \sigma}.
\end{equation*}
Since $\vert \nabla \varphi_{c} \vert \leq C R^{- 1} h^{- \beta}$, it yields $h^{2} \nabla \varphi_{c} \cdot \nabla ( P_{c} - z )^{- 1} \varphi_{c}= \CO_{M , \gamma} ( R^{- 1} h^{2 - \beta - \beta - \sigma} )$ and, using again $2 - 2 \beta - \sigma = 0$, we deduce
\begin{equation} \label{a30}
h^{2} \nabla \varphi_{c} \cdot \nabla( P_{c} - z )^{- 1} \varphi_{c}= \CO_{M , \gamma} ( R^{- 1} ) = o_{R \to \infty} ( 1 ) .
\end{equation}
It remains to estimate $[ h^{2} \Delta , \psi ] ( Q - z )^{- 1} \psi$ where $[ h^{2} \Delta, \psi ] = h^{2} \Delta \psi + 2 h^{2} \nabla \psi \cdot \nabla$. Applying Proposition \ref{a33} instead of Corollary \ref{a49} and using \eqref{a23} with $k = 1 , 2$, we obtain as before
\begin{align}
[ h^{2} \Delta , \psi ] ( Q - z )^{- 1} \psi &= h^{2} ( \Delta \psi ) ( Q - z )^{- 1} \psi + 2 h^{2} \nabla \psi \cdot \nabla ( Q - z )^{- 1} \psi  \nonumber \\
&= \CO ( M^{- 1} R^{- 2} h^{2 - 2 \beta -\sigma} ) + \CO( M^{- 1} R^{- 1} h^{2 - \beta - \beta - \sigma} ) \nonumber \\
&= o_{R \to \infty} ( 1 ) .  \label{a31}
\end{align}
Putting together \eqref{a28}, \eqref{a29}, \eqref{a30}, \eqref{a31} and \eqref{a27}, we obtain 
\begin{equation*}
( P - z ) R ( z ) = 1 + o_{R \to \infty} ( 1 ) .
\end{equation*}
The same way, we can show $R ( z ) ( P - z ) = 1 + o_{R \to \infty} ( 1 )$ which completes the proof.
\end{proof}

\subsection{Asymptotic of the eigenvalues of $P$}

In this section, we prove Theorem \ref{a7} on the distribution of the leading spectrum of $P$ under Assumption~\ref{h1}. Recall that $\Omega = [ 0 , M h^{\sigma} ] + i [ E - \varepsilon , E + \varepsilon ]$ and let
\begin{equation} \label{a34}
\sigma_{0} ( P ) = \bigcup_{c \in \CC_{\max} ( E )} ( i E + h^{\sigma} \Lambda_{c} ) ,
\end{equation}
denote the quasi-eigenvalues of $P$, counted with their multiplicity. We first prove that $P$ has no eigenvalue away from $\sigma_{0} ( P )$. More precisely, for $M \geq 1$, $\gamma >0$, $h$ small enough and $z \in \Omega \setminus ( \sigma_{0} ( P ) + B ( 0 , \gamma h^{\sigma}) )$, Proposition \ref{a26} and \eqref{a24} show that $( P - z )$ is invertible and
\begin{equation} \label{a35}
\Vert ( P - z )^{- 1} \Vert \leq 2 \Vert R ( z ) \Vert = 2 \Big\Vert \sum_{c \in \CC_{\max} ( E )} \varphi_{c} ( P_{c} - z )^{- 1} \varphi_{c} + \psi ( Q - z )^{- 1} \psi \Big\Vert .
\end{equation}
Here, we have chosen $R$ large enough and then $h$ small enough such that the $o_{h \to 0}^{R} ( 1 ) + o_{R \to \infty} ( 1 )$'s in Proposition \ref{a26} have a norm less than $1 / 2$. Thus, $P$ has no eigenvalue in $\Omega \setminus ( \sigma_{0} ( P ) + B ( 0 , \gamma h^{\sigma}) )$. Moreover, the resolvent estimate \eqref{a8} is a direct consequence of Corollary \ref{a49}, Proposition \ref{a33} and \eqref{a35}.

We now prove that $P$ has eigenvalues near $\sigma_{0} ( P )$. For that, we consider $\lambda_{0} \in \Omega \cap \sigma_{0} ( P )$, which writes
\begin{equation*}
\lambda_{0} = i E + \mu_{0} h^{\sigma} ,
\end{equation*}
for some $\mu_{0} \in \Lambda_{c_{0}}$ and $c_{0} \in \CC_{\max} ( E )$. From \eqref{a34}, the distance between two different elements of $\sigma_{0} ( P )$ is at least of order $h^{\sigma}$. Then, the circle $\partial B ( \lambda_{0} , \gamma h^{\sigma} )$, with $\gamma > 0$ small enough, is at distance $h^{\sigma}$ from $\sigma_{0} ( P )$ and $\lambda_{0}$ is the only element of $\sigma_{0} ( P )$ in its interior. Applying Corollary \ref{a49} and Proposition \ref{a33}, we can write
\begin{align}
( P - z )^{- 1} = \sum_{c \in \CC_{\max} ( E )} \varphi_{c} ( P_{c} - z )^{- 1} \varphi_{c} \big( \Id &+ o_{h \to 0}^{R} ( 1 ) + o_{R \to \infty} ( 1 ) \big)   \nonumber \\
&+ \psi ( Q - z )^{- 1} \psi \big( \Id + o_{h \to 0}^{R} ( 1 ) + o_{R \to \infty} ( 1 ) \big) ,     \label{a36}
\end{align}
with $( P_{c} - z )^{-1} = \CO ( h^{- \sigma} )$ and $( Q - z )^{-1} = \CO ( h^{- \sigma} )$ uniformly for $z \in \partial B ( \lambda_{0} , \gamma h^{\sigma} )$.

We denote by
\begin{equation*}
\Pi = - \frac{1}{2 i \pi} \oint_{\partial B ( \lambda_{0} , \gamma h^{\sigma} )} ( P - z )^{- 1} d z \qquad \text{and} \qquad \Pi_{c} = - \frac{1}{2 i \pi} \oint_{\partial B ( \lambda_{0} , \gamma h^{\sigma} )} ( P_{c} - z )^{- 1} d z ,
\end{equation*}
the spectral projector of $P$ and $P_{c}$ associated to the eigenvalues in $B ( \lambda_{0} , \gamma h^{\sigma} )$. Remark that $\Pi$ (resp. $\Pi_{c}$) acts on $L^{2} ( X )$ (resp. $L^{2} ( \R^{d} )$). From \eqref{a32}, the operator $\Pi_{c}$ differs from $0$ iff $\lambda_{0}$ is an eigenvalue of $P_{c}$ iff $\mu_{0}$ is an eigenvalue of $\CP_{c} $. In that case, $\lambda_{0}$ is the unique eigenvalue of $P_{c}$ in $B ( \lambda_{0} , \gamma h^{\sigma} )$ and
\begin{equation} \label{a37}
\Pi_{c} = U_{c}^{*} \widetilde{\Pi}_{c} U_{c} ,
\end{equation}
where $\widetilde{\Pi}_{c}$ is the spectral projector of $\CP_{c}$ associated to the eigenvalue $\mu_{0}$. Integrating \eqref{a36} along $\partial B ( \lambda_{0} , \gamma h^{\sigma} )$ leads to
\begin{equation} \label{a38}
\Pi = \sum_{c \in \CC_{\max} ( E )} \varphi_{c} \Pi_{c} \varphi_{c} + o_{h \to 0}^{R} ( 1 ) + o_{R \to \infty} ( 1 ) .
\end{equation}
Combining \eqref{a37} and \eqref{a38}, there exists $C > 0$ such that
\begin{equation} \label{a39}
\Vert \Pi \Vert + \sum_{c \in \CC_{\max} ( E )} \Vert \Pi_{c} \Vert \leq C ,
\end{equation}
for $h$ small enough. To finish the proof of Theorem \ref{a7}, it remains to show that
\begin{equation} \label{a40}
\rank \Pi = \sum_{c \in \CC_{\max} ( E )} \rank \Pi_{c} .
\end{equation}

Let $( \widetilde{u}_{c}^{j} )_{1 \leq j \leq J_{c}} \in L^{2} ( \R )$ be an orthonormal basis of $\im \widetilde{\Pi}_{c}$ independent of $h$ and define $u_{c}^{j} = U_{c}^{*} \widetilde{u}_{c}^{j} \in L^{2} ( \R )$. Thus, $( u_{c}^{j} )_{j}$ is an orthonormal basis of $\im \Pi_{c}$. We also set $v^{j}_{c} = \varphi_{c} u_{c}^{j} \in L^{2} (\R ) \cap L^{2} ( X )$ and $w_{c}^{j} = \Pi v_{c}^{j} \in L^{2}( X )$. We have $U_{c} \varphi_{c} u_{c}^{j} = \varphi ( x / R ) \widetilde{u}_{c}^{j}$. Since $\widetilde{u}_{c}^{j}$ is independent of $h$, the Lebesgue theorem implies $U_{c} \varphi_{c} u_{c}^{j} = \widetilde{u}_{c}^{j}  + o_{R \to \infty} ( 1 )$ and then
\begin{equation} \label{a41}
v_{c}^{j}= u_{c}^{j} + o_{R \to \infty} ( 1 ) ,
\end{equation}
in $L^{2} ( \R )$. Using $\varphi_{c} \varphi_{c^{\prime}} = 0$ for $c \neq c^{\prime}$, \eqref{a38}, \eqref{a39} and \eqref{a41} give
\begin{align}
w_{c}^{j} &= \varphi_{c} \Pi_{c} \varphi_{c} v_{c}^{j} + o_{h \to 0}^{R} ( 1 ) + o_{R \to \infty} ( 1 ) = \varphi_{c} \Pi_{c} u_{c}^{j} + o_{h \to 0}^{R} ( 1 ) + o_{R \to \infty} ( 1 ) \nonumber \\
&= v_{c}^{j} + o_{h \to 0}^{R} ( 1 ) + o_{R \to \infty} ( 1 ) ,    \label{a42}
\end{align}
in $L^{2} ( X )$. Recalling that $( u_{c}^{j} )_{j}$ are orthonormal and mainly supported near $c$, \eqref{a41}, \eqref{a42} imply that, for $R$ large enough and then $h$ small enough, the family $( w_{c}^{j} )_{c , j}$ is free and, since $w_{c}^{j} \in \im \Pi$,
\begin{equation} \label{a43}
\rank \Pi \geq \sum_{c \in \CC_{\max} ( E )} \rank \Pi_{c} .
\end{equation}

We now show by contradiction that the rank of $\Pi$ cannot be bigger than $\sum \rank \Pi_{c}$. If this is the case, there exists $w \in L^{2} ( X )$ with $\Vert w \Vert = 1$ orthogonal to the $w_{c}^{j}$'s. Using \eqref{a38}, \eqref{a39}, \eqref{a41}, \eqref{a42} and $\varphi_{c} \varphi_{c^{\prime}} = 0$ for $c \neq c^{\prime}$, we get
\begin{align}
0 &= \< w , w_{c}^{j} \>_{L^{2} ( X )} = \Big\< \sum_{c^{\prime} \in \CC_{\max} ( E )} \varphi_{c^{\prime}} \Pi_{c^{\prime}} \varphi_{c^{\prime}} w + o_{h \to 0}^{R} ( 1 ) + o_{R \to \infty} ( 1 ) , w_{c}^{j} \Big\>_{L^{2} ( X )}   \nonumber  \\
&= \< \varphi_{c} \Pi_{c} \varphi_{c} w , w_{c}^{j} \>_{L^{2} ( X )} + o_{h \to 0}^{R} ( 1 ) + o_{R \to \infty} ( 1 )   \nonumber \\
&= \< \Pi_{c} \varphi_{c} w , \varphi_{c} w_{c}^{j} \>_{L^{2} ( \R )} + o_{h \to 0}^{R} ( 1 ) + o_{R \to \infty} ( 1 ) \nonumber \\
&= \< \Pi_{c} \varphi_{c} w , u_{c}^{j} \>_{L^{2} ( \R )} + o_{h \to 0}^{R} ( 1 ) + o_{R \to \infty} ( 1 ) .
\end{align}
Since $( u_{c}^{j} )_{j}$ is an orthonormal basis of $\im \Pi_{c}$, it yields $\Pi_{c} \varphi_{c} w = o_{h \to 0}^{R} ( 1 ) + o_{R \to \infty} ( 1 )$ for all $c \in \CC_{\max} ( E )$. Using \eqref{a38}, we eventually deduce $w = o_{h \to 0}^{R} ( 1 ) + o_{R \to \infty} ( 1 )$ which provides a contradiction for $R$ large enough and then $h$ small enough with $\Vert w \Vert = 1$. Summing up,
\begin{equation}
\rank \Pi \leq \sum_{c \in \CC_{\max} ( E )} \rank \Pi_{c} ,
\end{equation}
and the theorem follows.

\subsection{Proof of Proposition \ref{a33}} \label{s1}

This proof follows the strategy of \cite{CoGa23_01}, Section 2. The main idea to estimate the resolvent of $Q - i \lambda$ is to use the ellipticity of $V - \lambda$ outside the level set $V^{- 1} ( \lambda )$ and to apply a Poincar\'e inequality near $V^{- 1} ( \lambda )$. Then, we first need to study the level sets of $V$. If $X$ is an open set of $\R^{d}$, we would like to treat the points at the boundary $\partial X$ as the points in the interior $X$. Thus, following \cite{CoGa23_01}, we extend the potential $V$. More precisely, since $V \in C^{1} ( \overline{X} )$ and $X$ is a Lipschitz domain, the Whitney extension theorem (see \cite{Wh34_01}) shows that $V$ can be extended as a $C^{1}$ function in a small neighborhood $Y$ of $\overline{X}$ and that the extension has no critical point in $Y \setminus X$. If $X = \T^{d}$, we take $Y = \T^{d}$.

\subsubsection{Description of the levels sets} \label{s3}

The first step is to understand the structure and the size of the levels set of $V$ near $E$. For $\lambda \in \R$, we recall that $V^{- 1} ( \lambda ) = \{ x \in Y ; \ V ( x ) = \lambda \}$. Following \cite{CoGa23_01}, we define for $\rho , \delta > 0$
\begin{equation} \label{a63}
M_{\lambda} = \{ x \in Y ; \ \vert V ( x ) - \lambda \vert \leq \rho \} ,
\end{equation}
the union of the level sets close to $\lambda$, and
\begin{equation} \label{a59}
\CM_{\lambda} = \{ x \in Y ; \ \dist ( x , M_{\lambda} ) \leq \delta \} .
\end{equation}
In the applications, $\lambda$ will be close to $E$ and $\rho , \delta$ will be powers of $h$. Since $V$ is continuous on the compact set $\overline{X}$, a compactness argument shows that, for all $\CU$ neighborhood of $V^{- 1} ( E ) \cap \overline{X}$ in $Y$, there exists $\varepsilon > 0$ such that $V^{- 1} ( [ E - 2 \varepsilon , E + 2 \varepsilon ] ) \cap \overline{X} \subset \CU$. Thus, taking $\varepsilon > 0$ small enough and $\lambda \in [ E - \varepsilon , E + \varepsilon ]$, it is enough to work near $V^{- 1} ( E ) \cap \overline{X}$ to study the set $M_{\lambda}$ near $\overline{X}$.

\underline{Case 1} Let $c \in V^{- 1} ( E ) \cap \overline{X}$ be such that $\nabla_{x} V ( c ) \neq 0$. Since $V \in C^{1}$, there exists a $C^{1}$ change of variables $x = \varphi ( y )$ with $c = \varphi ( 0 )$ and
\begin{equation} \label{b8}
V ( \varphi ( y ) ) = E + y_{1} ,
\end{equation}
for $y$ near $0$. Moreover, there exists $L > 0$ such that
\begin{equation} \label{b9}
L^{- 1} \vert y - \widetilde{y} \vert \leq \vert \varphi ( y ) - \varphi ( \widetilde{y} ) \vert \leq L \vert y - \widetilde{y} \vert ,
\end{equation}
for $y , \widetilde{y}$ near $0$. In particular, the level set of energy $\lambda$ near $E$ in the variables $y$ satisfies
\begin{equation} \label{b10}
V^{- 1} ( \lambda ) \subset \{ ( \lambda - E , y_{2} , \ldots , y_{d} ) ; \ ( y_{2} , \ldots , y_{d} ) \in \R^{d -1} \} .
\end{equation}
For any $\widetilde{x} = \varphi ( \widetilde{y} ) \in M_{\lambda}$, \eqref{b8} and \eqref{b10} imply $\vert \widetilde{y}_{1} - ( \lambda - E ) \vert \leq \rho$ and then $\widetilde{y} \in V^{- 1} ( \lambda ) + B ( 0 , \rho )$. Eventually, consider a point $x = \varphi ( y )$ at distance less than $\delta$ from the part of $M_{\lambda}$ described previously. Thus, there exists $\widetilde{x} = \varphi ( \widetilde{y} ) \in M_{\lambda}$ with $\vert x - \widetilde{x} \vert\leq \delta$. From \eqref{b9}, we have $\vert y - \widetilde{y} \vert \leq L \delta$ and then
\begin{equation} \label{b11}
\vert y_{1} - ( \lambda - E ) \vert \leq \rho + L \delta ,
\end{equation}
for all $x = \varphi ( y ) \in \CM_{\lambda}$. Note that these arguments are similar to those of Section 2.2 of \cite{CoGa23_01}.

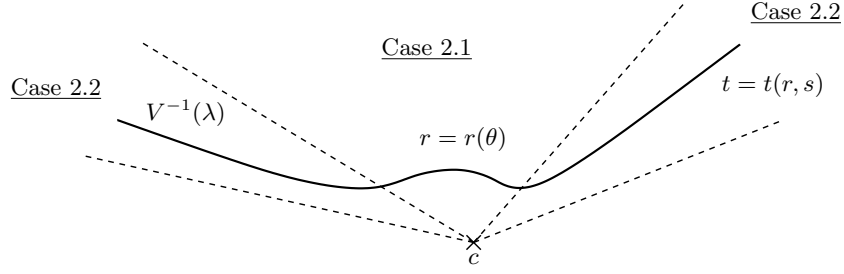
\begin{figure}
\begin{center}
\begin{picture}(0,0)%
\includegraphics{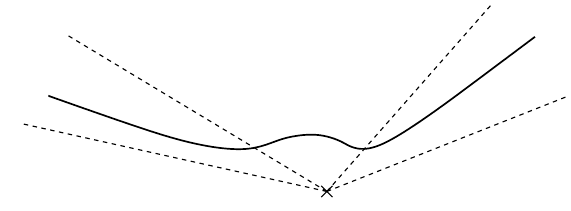}%
\end{picture}%
\setlength{\unitlength}{1184sp}%
\begingroup\makeatletter\ifx\SetFigFont\undefined%
\gdef\SetFigFont#1#2#3#4#5{%
  \reset@font\fontsize{#1}{#2pt}%
  \fontfamily{#3}\fontseries{#4}\fontshape{#5}%
  \selectfont}%
\fi\endgroup%
\begin{picture}(15198,5563)(-8714,334)
\put(  1,389){\makebox(0,0)[b]{\smash{{\SetFigFont{9}{10.8}{\rmdefault}{\mddefault}{\updefault}$c$}}}}
\put(5176,3989){\makebox(0,0)[lb]{\smash{{\SetFigFont{9}{10.8}{\rmdefault}{\mddefault}{\updefault}$t = t ( r , s )$}}}}
\put(-1124,2864){\makebox(0,0)[lb]{\smash{{\SetFigFont{9}{10.8}{\rmdefault}{\mddefault}{\updefault}$r = r ( \theta )$}}}}
\put(5776,5489){\makebox(0,0)[lb]{\smash{{\SetFigFont{9}{10.8}{\rmdefault}{\mddefault}{\updefault}\underline{Case 2.2}}}}}
\put(-974,4739){\makebox(0,0)[b]{\smash{{\SetFigFont{9}{10.8}{\rmdefault}{\mddefault}{\updefault}\underline{Case 2.1}}}}}
\put(-8699,3914){\makebox(0,0)[b]{\smash{{\SetFigFont{9}{10.8}{\rmdefault}{\mddefault}{\updefault}\underline{Case 2.2}}}}}
\put(-6824,3389){\makebox(0,0)[lb]{\smash{{\SetFigFont{9}{10.8}{\rmdefault}{\mddefault}{\updefault}$V^{- 1} ( \lambda )$}}}}
\end{picture}%
\end{center}
\caption{The energy surface $V^{- 1} ( \lambda )$ near a critical point $c$.} \label{f1}
\end{figure}

\underline{Case 2} Consider now a critical point $c \in V^{- 1} ( E ) \cap X$. To simplify the notations, we assume, without lost of generality, that $E = 0$. We work in the spherical coordinates $x = c + r \theta$ with $r \in \R^{+}$ small and $\theta \in \S^{d - 1}$, and we denote $V ( r , \theta ) = V ( c + r \theta )$. We will describe the sets $M_{\lambda}$ and $\CM_{\lambda}$ in the angular directions close to $\theta_{0} \in \S^{d - 1}$ by distinguishing the cases $v_{c} ( \theta_{0} ) \neq 0$ and $v_{c} ( \theta_{0} ) = 0$, see Figure \ref{f1}.

\underline{Case 2.1} Let $\theta_{0} \in \S^{d - 1}$ be such that $v_{c} ( \theta_{0} ) \neq 0$. In the sequel we assume that $v_{c} ( \theta_{0} ) > 0$, the other case can be treated in the same way. We work in the variables $r , \theta$ and assume that $\theta$ is close to $\theta_{0}$ and $r \geq 0$ is small.

We first write $V$ in normal form. From \eqref{a4}, we have $V ( x ) = v_{c} ( \theta ) r^{\alpha_{c}} + \xi ( x )$ with $\xi ( c ) = 0$ and $\nabla \xi ( x ) = o ( r^{\alpha_{c} - 1} )$. Let $\widehat{v_{c}}$ denote a positive $C^{1} ( \S^{d - 1} )$ function such that $\widehat{v_{c}} ( \theta ) = v_{c} ( \theta )$ for $\theta$ near $\theta_{0}$. We then define the following change of variables from a pointed neighborhood of $c$ to a pointed neighborhood of $0$
\begin{equation} \label{d1}
x = c + r \theta \longmapsto \widehat{x} = \widehat{r} \widehat{\theta} ,
\end{equation}
with
\begin{equation} \label{d22}
\widehat{r} ( r , \theta ) = ( \widehat{v_{c}} ( \theta ) + \xi ( x ) r^{- \alpha_{c}} )^{1 / \alpha_{c}} r \qquad \text{and} \qquad \widehat{\theta} = \theta .
\end{equation}
Note that $\widehat{v_{c}}$ is bounded from below by a positive constant since $\widehat{v_{c}}$ is continuous and positive on the compact set $\S^{d - 1}$. On the other hand, \eqref{a57} shows that $\xi ( x ) r^{- \alpha_{c}} = o_{r \to  0} ( 1 )$. Thus, \eqref{d1} is well defined. Moreover, \eqref{a4} and direct computations give
\begin{align}
\partial_{r} \widehat{r} ( r , \theta ) &= \widehat{v_{c}} ( \theta )^{1 / \alpha_{c}} + o _{r \to 0} ( 1 ) ,  \label{d2}  \\
\partial_{\theta} \widehat{r} ( r , \theta ) &= \partial_{\theta} \widehat{v_{c}} ( \theta ) \frac{\widehat{v_{c}} ( \theta )^{\frac{1 - \alpha_{c}}{\alpha_{c}}}}{\alpha_{c}} r + o _{r \to 0} ( r ) .   \label{d3}
\end{align}
Since $\widehat{v_{c}} ( \theta )^{1 / \alpha_{c}} \neq 0$, this implies that \eqref{d1} is a local diffeomorphism near any point of a pointed neighborhood of $c$. Since this change of variables preserves $\theta$, \eqref{d22}, $\widehat{r} ( 0 , \theta ) = 0$ and $\partial_{r} \widehat{r} ( r , \theta ) > 0$ show that \eqref{d1} is a bijection from a pointed neighborhood of $c$ onto a pointed neighborhood of $0$. Summing up, \eqref{d1} is a diffeomorphism from a pointed neighborhood of $c$ onto a pointed neighborhood of $0$. Note that this is not a diffeomorphism at $x = 0$ when $\widehat{v_{c}}$ is not constant. Working in spherical coordinates and using \eqref{d22}, \eqref{d2}, \eqref{d3} and $\vert \nabla_{x} u \vert / C \leq \vert \partial_{r} u \vert + r^{- 1} \vert \partial_{\theta} u \vert \leq C \vert \nabla_{x} u \vert$ for $u \in C^{1} ( \R^{d} )$, one can check that there exists $C > 0$ such that
\begin{equation} \label{d6}
\Vert \nabla_{x} \widehat{x} ( x ) \Vert + \Vert \nabla_{\widehat{x} } x ( \widehat{x}  ) \Vert  \leq C ,
\end{equation}
for $x$ in a pointed neighborhood of $c$ and  $\widehat{x} $ in a pointed neighborhood of $0$, and
\begin{equation} \label{d10}
\vert \nabla_{\widehat{x}} \widehat{u} \vert / C \leq \vert \nabla_{x} u \vert \leq C \vert \nabla_{\widehat{x}} \widehat{u} \vert ,
\end{equation}
where $\widehat{u}$ denote the function $u ( x )$ in the new coordinates $\widehat{x}$. In particular, the Jacobian of the change of variables \eqref{d1} and of its inverse are uniformly bounded. Finally, we have by construction
\begin{equation} \label{d4}
\widehat{V} ( \widehat{x} ) = \widehat{r}^{\alpha_{c}} ,
\end{equation}
for $\theta$ near $\theta_{0}$.

We describe the sets $M_{\lambda} , \CM_{\lambda}$ with $\lambda , \delta > 0$ small enough and $\rho < \lambda / 2$ in the new variables $\widehat{x} = \widehat{r} \widehat{\theta}$. From \eqref{d4}, a point $x = c + r \theta$ with $\theta$ near $\theta_{0}$ belongs to $M_{\lambda}$  iff
\begin{equation*}
( \lambda - \rho )^{1 / \alpha_{c}} \leq \widehat{r} \leq ( \lambda + \rho )^{1 / \alpha_{c}} .
\end{equation*}
Since $\rho < \lambda / 2$, there exists $C > 0$ such that
\begin{equation} \label{d7}
\vert \widehat{r}  - \lambda^{1 / \alpha_{c}} \vert \leq C \rho \lambda^{\frac{1 - \alpha_{c}}{\alpha_{c}}} ,
\end{equation}
for all $x = c + r \theta \in M_{\lambda}$ with $\theta$ near $\theta_{0}$. We now study $\CM_{\lambda}$. Consider a point $x$ such that $\dist ( x , M_{\lambda} ) \leq \delta$, where $M_{\lambda}$ is (the part of) the levels sets described previously. By definition, there exists $y \in M_{\lambda}$ satisfying \eqref{d7} such that
\begin{equation} \label{d5}
\vert x - y \vert \leq \delta .
\end{equation}
By the mean value inequality and \eqref{d6}, we have $\vert \widehat{x} - \widehat{y} \vert \leq C \delta$ after the change of variables \eqref{d1}.  The triangle inequality yields $\vert \widehat{r} - \widehat{r}_{y} \vert \leq \vert \widehat{x} - \widehat{y} \vert \leq C \delta$ with the notations $\widehat{x} = \widehat{r} \widehat{\theta}$ and $\widehat{y} = \widehat{r}_{y} \widehat{\theta}_{y}$. Combining with \eqref{d7} for $y$, it implies
\begin{equation} \label{d8}
\vert \widehat{r}  - \lambda^{1 / \alpha_{c}} \vert \leq \vert \widehat{r} - \widehat{r}_{y}  \vert + \vert \widehat{r}_{y}  - \lambda^{1 / \alpha_{c}} \vert \leq C \delta + C \rho \lambda^{\frac{1 - \alpha_{c}}{\alpha_{c}}} .
\end{equation}

\underline{Case 2.2} Assume now that $v_{c} ( \theta_{0} ) = 0$. Note that this implies that $d \geq 2$. As in the previous case, we write $V$ in a normal form valid for $\theta$ near $\theta_{0}$ and $r \geq 0$ small.

The first step is to construct a (family of) global diffeomorphism of $\S^{d - 1}$. To simplify the exposition, we assume that $\theta_{0} = ( 1 , 0 , \ldots )$ and that $d_{\theta_{0}} v_{c} = \gamma ( 0 , 1 , 0 , \ldots )$ with $\gamma \neq 0$. Here, we identify the differential $d_{\theta_{0}} v_{c}$ with the unique vector $d_{\theta_{0}} v_{c} \in \theta_{0}^{\perp}$ such that $d_{\theta_{0}} v_{c} ( U ) = \< d_{\theta_{0}} v_{c} , U \>$ for all $U \in T_{\theta_{0}} \S^{d - 1} = \theta_{0}^{\perp}$. In the following, we assume that $\gamma > 0$. We use the notation $\theta^{\prime} = ( \theta_{3} , \ldots , \theta_{d} )$. The map
\begin{equation} \label{d12}
( \theta_{2} , \theta^{\prime} ) \longmapsto \big( \sqrt{1 - \theta_{2}^{2} - {\theta^{\prime}}^{2}} , \theta_{2} , \theta^{\prime} \big) ,
\end{equation}
is a diffeomorphism from a neighborhood of $0$ in $\R^{d - 1}$ to a neighborhood of $\theta_{0}$ in $\S^{d - 1}$. In the sequel, we identify $\theta \in \S^{d - 1}$ near $\theta_{0}$ with its coordinates $( \theta_{2} , \theta^{\prime} )  \in \R^{d - 1}$ using the application \eqref{d12}. Let $g ( \theta_{2} , \theta^{\prime} )$ be a $C^{\infty}$ function on $\R^{d - 1}$ such that $g ( \theta_{2} , \theta^{\prime} ) = \theta_{2}$ outside a small neighborhood of $0$, $g ( \theta_{2} , \theta^{\prime} ) = \gamma \theta_{2}$ in a smaller neighborhood of $0$ and $\partial_{\theta_{2}} g ( \theta_{2} , \theta^{\prime} ) > 0$ for all $( \theta_{2} , \theta^{\prime} ) \in \R^{d - 1}$. Then,
\begin{equation*}
\varphi_{0} : ( \theta_{2} , \theta^{\prime} ) \longmapsto ( g ( \theta_{2} , \theta^{\prime} ) , \theta^{\prime} ) ,
\end{equation*}
is a global diffeomorphism of $\R^{d - 1}$ which coincides with the identity outside a small neighborhood of $0$. For $\chi \in C^{\infty}_{0} ( \R^{d - 1} )$ equal to $1$ near $0$ and $\varepsilon > 0$, we define the function
\begin{equation} \label{d23}
\varphi : ( \theta_{2} , \theta^{\prime} ) \longmapsto \Big( \chi \Big( \frac{\theta_{2} , \theta^{\prime}}{\varepsilon} \Big) \big( v_{c} ( \theta ) + \xi ( r , \theta ) r^{- \alpha_{c}} \big) + \Big( 1 - \chi \Big( \frac{\theta_{2} , \theta^{\prime}}{\varepsilon} \Big) \Big) g ( \theta_{2} , \theta^{\prime} ) , \theta^{\prime} \Big) ,
\end{equation}
from $\R^{d - 1}$ to $\R^{d - 1}$. It can be written
\begin{equation} \label{d16}
\varphi = \varphi_{0}  + ( \phi , 0 ) \quad \text{with} \quad \phi ( \theta_{2} , \theta^{\prime} , r ) = \chi \Big( \frac{\theta_{2} , \theta^{\prime}}{\varepsilon} \Big) \big( v_{c} ( \theta ) + \xi ( r , \theta ) r^{- \alpha_{c}} - g ( \theta_{2} , \theta^{\prime} ) \big) .
\end{equation}
Since $v_{c} \in C^{1}$, $v_{c} ( \theta_{0} ) = 0$ and $\partial_{\theta_{2} ,  \theta^{\prime}} v_{c} ( \theta_{0} ) = \gamma ( 1 , 0 )$, we deduce $v_{c} ( \theta ) = \gamma \theta_{2} + o_{\theta \to \theta_{0}} ( \theta_{2} ,  \theta^{\prime} )$ and $\partial_{\theta_{2} ,  \theta^{\prime}} v_{c} ( \theta ) = \gamma ( 1 , 0 ) + o_{\theta \to \theta_{0}} ( 1 )$ (recall that $( \theta_{2} , \theta^{\prime} ) \in \R^{d - 1}$ is identified with $\theta \in \S^{d - 1}$ through \eqref{d12}). Together with $g ( \theta_{2} , \theta^{\prime} ) = \gamma \theta_{2}$ on the support of $\chi \big( ( \theta_{2} , \theta^{\prime} ) / \varepsilon \big)$ and
\begin{equation} \label{d14}
\xi ( r , \theta ) r^{- \alpha_{c}} = o_{r \to 0} ( 1 ) , \quad \partial_{\theta_{2} ,  \theta^{\prime}} \xi ( r , \theta ) r^{- \alpha_{c}} = o_{r \to 0} ( 1 ) , \quad \partial_{r} \xi ( r , \theta ) r^{- \alpha_{c}} = o_{r \to 0} ( r^{- 1} ) ,
\end{equation}
which follow from \eqref{a4}, it implies
\begin{align}
\partial_{\theta_{2} ,  \theta^{\prime}} \phi ( \theta_{2} , \theta^{\prime} , r ) ={}& \varepsilon^{- 1} ( \partial_{\theta_{2} ,  \theta^{\prime}} \chi ) \Big( \frac{\theta_{2} , \theta^{\prime}}{\varepsilon} \Big) \big( v_{c} ( \theta ) + \xi ( r , \theta ) r^{- \alpha_{c}} - g ( \theta_{2} , \theta^{\prime} ) \big)  \nonumber   \\
&+ \chi \Big( \frac{\theta_{2} , \theta^{\prime}}{\varepsilon} \Big) \partial_{\theta_{2} ,  \theta^{\prime}} \big( v_{c} ( \theta ) + \xi ( r , \theta ) r^{- \alpha_{c}} - g ( \theta_{2} , \theta^{\prime} ) \big) \nonumber  \\
={}&  \varepsilon^{- 1} ( \partial_{\theta_{2} ,  \theta^{\prime}} \chi ) \Big( \frac{\theta_{2} , \theta^{\prime}}{\varepsilon} \Big) \big( o_{\theta \to \theta_{0}} ( \theta_{2} ,  \theta^{\prime} ) + o_{r \to 0} ( 1 ) \big) \nonumber  \\
&+ \chi \Big( \frac{\theta_{2} , \theta^{\prime}}{\varepsilon} \Big) \big( o_{\theta \to \theta_{0}} ( 1 ) + o_{r \to 0} ( 1 ) \big) \nonumber   \\
={}& o_{\varepsilon \to 0} ( 1 ) + \varepsilon^{- 1} o_{r \to 0} ( 1 ) .   \label{d15}
\end{align}
In particular, taking $\varepsilon > 0$ small enough and then assuming that $r$ is small enough, we can guaranty that $\sup_{\R^{d - 1}} \vert \partial_{\theta_{2} ,  \theta^{\prime}} \phi \vert \leq ( \sup_{\R^{d - 1}} \Vert \partial_{\theta_{2} ,  \theta^{\prime}} \varphi_{0}^{- 1} \Vert )^{- 1} / 2$. Using $\varphi = \varphi_{0}  + ( \phi , 0 )$, it yields that $\varphi$ is a local diffeomorphism near any point of $\R^{d - 1}$. Since $\varphi_{0}$ is global diffeomorphism on $\R^{d - 1}$ which coincides with the identity outside a compact set, there exists $\nu > 0$ such that $\vert \varphi_{0} ( x ) - \varphi_{0} ( y ) \vert \geq \nu \vert x - y \vert$ for all $x , y \in \R^{d - 1}$. For $\varepsilon > 0$ and then $r$ small enough, we can assume that $\sup_{\R^{d - 1}} \vert \partial_{\theta_{2} ,  \theta^{\prime}} \phi \vert \leq \nu / 2$. The mean value inequality gives $\vert \phi ( x ) - \phi ( y ) \vert \leq \nu \vert x - y \vert / 2$ and then
\begin{equation} \label{d38}
\vert \varphi ( x ) - \varphi ( y ) \vert \geq \vert \varphi_{0} ( x ) - \varphi_{0} ( y ) \vert - \vert \phi ( x ) - \phi ( y ) \vert \geq \frac{\nu}{2} \vert x - y \vert .
\end{equation}
Thus, $\varphi$ is injective on $\R^{d - 1}$. Finally, the set $\varphi ( \R^{d - 1} )$ is closed thanks to \eqref{d38} and open since $\varphi$ is a local diffeomorphism. By connectedness, it shows that $\varphi( \R^{d - 1} ) = \R^{d - 1}$, that is $\varphi$ is surjective. Summing up, $\varphi$ is a global diffeomorphism on $\R^{d - 1}$. By conjugation with the local diffeomorphism \eqref{d12}, $\varphi$ can be seen as a local diffeomorphism on $\S^{d - 1}$ near $\theta_{0}$. Extended by the identity to the whole $\S^{d - 1}$, it provides a global diffeomorphism on $\S^{d - 1}$ for $r$ small enough denoted $\psi ( \theta , r ) : \S^{d - 1} \times [ 0 , r_{0} ] \longrightarrow \S^{d - 1}$. Moreover, there exists $C > 0$ such that
\begin{equation} \label{d33}
\vert \partial_{\theta} \psi ( \theta , r ) \vert + \vert \partial_{\theta} \psi^{- 1} ( \theta , r ) \vert \leq C ,
\end{equation}
uniformly for $r$ small enough. Note that when $\gamma < 0$, we can extend $\varphi$ with the symmetry $\theta \longmapsto ( \theta_{1} , - \theta_{2} , \theta_{3} , \ldots \theta_{d} )$ on $\S^{d - 1}$ instead of $\Id_{\S^{d - 1}}$. We now mimic the previous construction by replacing the function $v_{c} ( \theta ) + \xi ( r , \theta ) r^{- \alpha_{c}}$ by $\theta_{2} ( 1 - \theta_{2}^{2} - {\theta^{\prime}}^{2} )^{\frac{\alpha_{c} - 1}{2}}$ which satisfies similar properties. This provides a new diffeomorphism on $\S^{d - 1}$ denoted $\widehat{\psi}$ such that $\widehat{\psi}_{2} ( \theta ) = \theta_{2} \theta_{1}^{\alpha_{c} - 1}$. We eventually define
\begin{equation} \label{d26}
\widetilde{\psi} ( \theta , r ) = \widehat{\psi}^{- 1} ( \psi ( \theta , r ) ) : \S^{d - 1} \times [ 0 , r_{0} ] \longrightarrow \S^{d - 1} ,
\end{equation}
which is a global diffeomorphism on $\S^{d - 1}$ for $r$ small enough. Lastly, \eqref{d33} gives
\begin{equation} \label{d17}
\vert \partial_{\theta} \widetilde{\psi} ( \theta , r ) \vert + \vert \partial_{\theta} \widetilde{\psi}^{- 1} ( \theta , r ) \vert \leq C ,
\end{equation}
uniformly for $r$ small enough.

With the construction of the last paragraph, we define the following change of variables on $\R^{d}$ from a pointed neighborhood of $c$ to a pointed neighborhood of $0$
\begin{equation} \label{d13}
x = c + r \theta \longmapsto \widetilde{x} = \widetilde{r} \widetilde{\theta} ,
\end{equation}
with $\widetilde{r} = r$ and $\widetilde{\theta} = \widetilde{\psi}( \theta , r )$. Since $\widetilde{\psi}$ is a diffeomorphism on $\S^{d - 1}$, \eqref{d13} is a local diffeomorphism near any point of a pointed neighborhood of $c$. Since it preserves the radius and $\widetilde{\psi}$ is a bijection on $\S^{d - 1}$, it is a bijection from a pointed neighborhood of $c$ onto a pointed neighborhood of $0$. Summing up, \eqref{d13} is a diffeomorphism from a pointed neighborhood of $c$ onto a pointed neighborhood of $0$. On the other hand, \eqref{d16}, \eqref{d14} and \eqref{d26} give
\begin{equation} \label{d18}
\partial_{r} \widetilde{\theta} ( r , \theta ) = d \widehat{\psi}^{- 1} ( \partial_{r} \varphi ( r , \theta ) ) = o_{r \to 0} ( r^{- 1} ) .
\end{equation}
Working in spherical coordinates and using \eqref{d17}, \eqref{d18} and $\vert \nabla_{x} u \vert / C \leq \vert \partial_{r} u \vert + r^{- 1} \vert \partial_{\theta} u \vert \leq C \vert \nabla_{x} u \vert$ for $u \in C^{1} ( \R^{d} )$, one can check that there exists $C > 0$ such that
\begin{equation} \label{d20}
\Vert \nabla_{x} \widetilde{x} ( x ) \Vert + \Vert \nabla_{\widetilde{x}} x ( \widetilde{x} ) \Vert \leq C ,
\end{equation}
for $x$ in a pointed neighborhood of $c$ and  $\widetilde{x} $ in a pointed neighborhood of $0$, and
\begin{equation} \label{d19}
\vert \nabla_{\widetilde{x}} \widetilde{u} \vert / C \leq \vert \nabla_{x} u \vert \leq C \vert \nabla_{\widetilde{x}} \widetilde{u} \vert ,
\end{equation}
where $\widetilde{u}$ denote the function $u ( x )$ in the new coordinates $\widetilde{x}$. In particular, the Jacobian of the change of variables \eqref{d13} and of its inverse are uniformly bounded. Finally, since $V ( x ) = ( v_{c} ( \theta ) + \xi ( r , \theta ) r^{- \alpha_{c}} ) r^{\alpha_{c}}$ and $\widehat{\psi}_{2} ( \widetilde{\theta} ) = \widetilde{\theta}_{2} \widetilde{\theta}_{1}^{\alpha_{c} - 1}$, we have by construction (see \eqref{d23})
\begin{equation} \label{d21}
\widetilde{V} ( \widetilde{x} ) = \widetilde{\theta}_{2} \widetilde{\theta}_{1}^{\alpha_{c} - 1} \widetilde{r}^{\alpha_{c}} = \widetilde{x}_{2} \widetilde{x}_{1}^{\alpha_{c} - 1} ,
\end{equation}
for $\theta$ near $\theta_{0}$.

We describe the set $M_{\lambda}$ in the new variables $\widetilde{x} = \widetilde{r} \widetilde{\theta}$ with $r = \widetilde{r} > T \max ( \rho^{1 / \alpha_{c}} ,  \delta )$, $T \geq 1$ large enough. Since we work in the region $\theta$ close to $\theta_{0}$ and $r$ small, we have $\widetilde{\theta}$ close to $( 1 , 0 )$ and then
\begin{equation} \label{d34}
\widetilde{x}_{1} \leq \widetilde{r} \leq  2 \widetilde{x}_{1} .
\end{equation}
From \eqref{d21}, a point $x = c + r \theta$ with $\theta$ near $\theta_{0}$ and $r$ small belongs to $M_{\lambda}$  iff
\begin{equation} \label{d24}
\vert  \widetilde{x}_{2}  - \lambda \widetilde{x}_{1}^{1 - \alpha_{c}} \vert = \vert \widetilde{V} ( \widetilde{x} ) - \lambda \vert \widetilde{x}_{1}^{1 - \alpha_{c}}  \leq \rho \widetilde{x}_{1}^{1 - \alpha_{c}} .
\end{equation}
Using that $\rho \widetilde{x}_{1}^{- \alpha_ {c}} \leq C T^{- \alpha_{c}} \leq 1$ for $T$ large enough and $\vert \widetilde{x}_{2} \vert \leq \widetilde{x}_{1}$ (because $\widetilde{\theta}$ is close to $( 1 , 0 )$), the previous inequality gives
\begin{equation} \label{d35}
\vert  \lambda \widetilde{x}_{1}^{- \alpha_{c}} \vert \leq 2 .
\end{equation}
We now study $\CM_{\lambda}$. Consider a point $x$ such that $\dist ( x , M_{\lambda} ) \leq \delta$, where $M_{\lambda}$ is (the part of) the levels sets described previously. By definition, there exists $y \in M_{\lambda}$ satisfying \eqref{d24} such that $\vert x - y \vert \leq \delta$. Let $\widetilde{x} = \widetilde{r} \widetilde{\theta}$ and $\widetilde{y} = \widetilde{r}_{y} \widetilde{\theta}_{y}$ denote the points $x$ and $y$ after the change of variables \eqref{d13}. By the mean value inequality and \eqref{d20}, we have
\begin{equation} \label{d25}
\vert \widetilde{x} - \widetilde{y} \vert \leq C \delta .
\end{equation}
Since $\widetilde{r}_{y} > T  \delta$, the triangle inequality and \eqref{d34} for $\widetilde{y}$ give $\vert \widetilde{r} - \widetilde{r}_{y} \vert \leq \vert \widetilde{x} - \widetilde{y} \vert \leq C \delta \leq C T^{- 1} \widetilde{r}_{y} \leq \widetilde{r}_{y} / 2$ for $T$ large enough and then
\begin{equation} \label{d36}
\widetilde{y}_{1} / 2 \leq \widetilde{r} \leq 3 \widetilde{y}_{1} .
\end{equation}
The same way, \eqref{d34} for $\widetilde{y}$ and \eqref{d25} give $\vert \widetilde{x}_{1} - \widetilde{y}_{1} \vert \leq 2 C T^{- 1} \widetilde{y}_{1} \leq \widetilde{y}_{1} /2$ for $T$ large enough and then
\begin{equation} \label{d27}
\widetilde{y}_{1} / 2 \leq \widetilde{x}_{1} \leq 3 \widetilde{y}_{1} / 2 .
\end{equation}
On the other hand, the mean value inequality, \eqref{d35} for $\widetilde{y}$, \eqref{d25} and \eqref{d27} yield
\begin{equation} \label{d29}
\vert \lambda \widetilde{y}_{1}^{1 - \alpha_{c}} - \lambda \widetilde{x}_{1}^{1 - \alpha_{c}}  \vert \leq ( \alpha_{c} - 1 ) \lambda \sup_{z \in [ \widetilde{x}_{1} , \widetilde{y}_{1} ]} z^{- \alpha_{c}} \vert \widetilde{x}_{1} - \widetilde{y}_{1} \vert \leq C \delta.
\end{equation}
Summing up, \eqref{d24} for $\widetilde{y}$, \eqref{d36}, \eqref{d29} and $\vert \widetilde{x}_{2} - \widetilde{y}_{2} \vert \leq C \delta$ (see \eqref{d25}) give
\begin{align}
\vert  \widetilde{x}_{2}  - \lambda \widetilde{x}_{1}^{1 - \alpha_{c}} \vert &\leq \vert \widetilde{x}_{2} - \widetilde{y}_{2} \vert + \vert  \widetilde{y}_{2}  - \lambda \widetilde{y}_{1}^{1 - \alpha_{c}} \vert + \vert \lambda \widetilde{y}_{1}^{1 - \alpha_{c}}  - \lambda \widetilde{x}_{1}^{1 - \alpha_{c}} \vert    \nonumber \\
&\leq C \delta + \rho \widetilde{y}_{1}^{1 - \alpha_{c}} + C \delta     \nonumber \\
&\leq  C \delta + C \rho \widetilde{r}^{1 - \alpha_{c}} ,  \label{d31}
\end{align}
for all $x \in \CM_{\lambda}$.

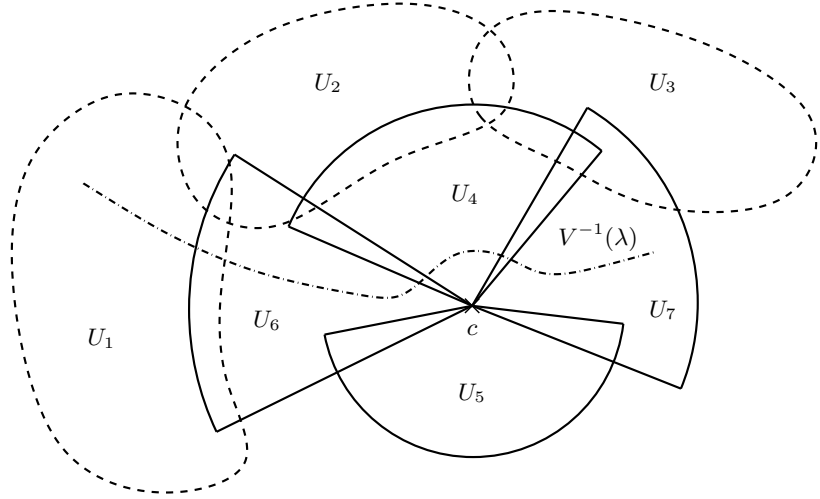
\begin{figure}
\begin{center}
\begin{picture}(0,0)%
\includegraphics{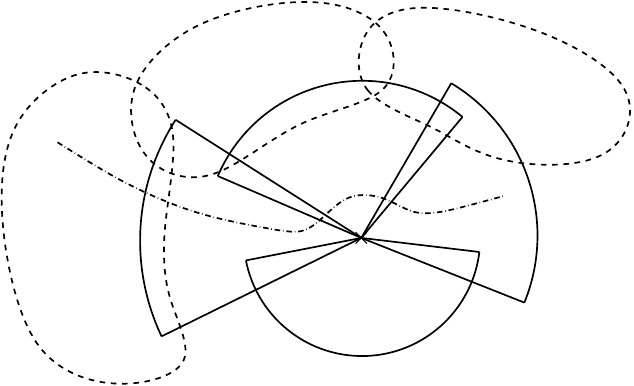}%
\end{picture}%
\setlength{\unitlength}{1184sp}%
\begingroup\makeatletter\ifx\SetFigFont\undefined%
\gdef\SetFigFont#1#2#3#4#5{%
  \reset@font\fontsize{#1}{#2pt}%
  \fontfamily{#3}\fontseries{#4}\fontshape{#5}%
  \selectfont}%
\fi\endgroup%
\begin{picture}(16845,10268)(-9633,-3093)
\put(  1,239){\makebox(0,0)[b]{\smash{{\SetFigFont{9}{10.8}{\rmdefault}{\mddefault}{\updefault}$c$}}}}
\put(-2999,5339){\makebox(0,0)[b]{\smash{{\SetFigFont{9}{10.8}{\rmdefault}{\mddefault}{\updefault}$U_{2}$}}}}
\put(3976,5339){\makebox(0,0)[b]{\smash{{\SetFigFont{9}{10.8}{\rmdefault}{\mddefault}{\updefault}$U_{3}$}}}}
\put(  1,-1111){\makebox(0,0)[b]{\smash{{\SetFigFont{9}{10.8}{\rmdefault}{\mddefault}{\updefault}$U_{5}$}}}}
\put(-149,3014){\makebox(0,0)[b]{\smash{{\SetFigFont{9}{10.8}{\rmdefault}{\mddefault}{\updefault}$U_{4}$}}}}
\put(-4274,389){\makebox(0,0)[b]{\smash{{\SetFigFont{9}{10.8}{\rmdefault}{\mddefault}{\updefault}$U_{6}$}}}}
\put(3976,614){\makebox(0,0)[b]{\smash{{\SetFigFont{9}{10.8}{\rmdefault}{\mddefault}{\updefault}$U_{7}$}}}}
\put(2626,2114){\makebox(0,0)[b]{\smash{{\SetFigFont{9}{10.8}{\rmdefault}{\mddefault}{\updefault}$V^{- 1} ( \lambda )$}}}}
\put(-7724, 14){\makebox(0,0)[b]{\smash{{\SetFigFont{9}{10.8}{\rmdefault}{\mddefault}{\updefault}$U_{1}$}}}}
\end{picture}%
\end{center}
\caption{Some elements of the covering in Lemma \ref{d37} with $U_{1} , U_{2} , U_{3}$ satisfying Case 1, $U_{4} , U_{5}$ satisfying Case 2.1 and $U_{6} , U_{7}$ satisfying Case 2.2.} \label{f5}
\end{figure}

Combining the three cases of this section and using the compacity of $\overline{X}$ and $\S^{d - 1}$, we have proven the following description of the potential $V$ illustrated in Figure \ref{f5}.

\begin{lemma}\sl \label{d37}
There exists a finite covering of $\overline{X} = \bigcup_{j \in J} U_{j}$ such that either

\underline{\rm Case 1} $U_{j}$ is an open set in $\R^{d}$ on which $\nabla_{x} V$ does not vanish and
\begin{equation*}
\forall x \in U_{j}, \qquad V ( x ) = C + y_{1} ( x ) ,
\end{equation*}
where $y ( x ) = \varphi^{- 1} ( x )$ is a $C^{1}$ diffeomorphism from $U_{j}$ to $\R^{d}$.

\underline{\rm Case 2.1} $U_{j}$ is a conical sector with vertex at a critical point $c$ and
\begin{equation*}
\forall x \in U_{j}, \qquad V ( x ) = \vert \widehat{x} ( x ) \vert^{\alpha_{c}} ,
\end{equation*}
where $\widehat{x} ( x )$ is a pointed $C^{1}$ diffeomorphism near $c$.

\underline{\rm Case 2.2} $U_{j}$ is a conical sector with vertex at a critical point $c$ and
\begin{equation*}
\forall x \in U_{j}, \qquad V ( x ) = \widetilde{x}_{2} ( x ) \widetilde{x}_{1}^{\alpha_{c} - 1} ( x ) ,
\end{equation*}
where $\widetilde{x} ( x )$ is a pointed $C^{1}$ diffeomorphism near $c$ (with $\widetilde{x}_{1} ( x ) \geq 0$ for $x \in U_{j}$).
\end{lemma}

\subsubsection{Resolvent estimate}

We can now prove the resolvent estimate for $Q$. We will take $\rho = S \delta^{\alpha}$ with $S > 0$ large enough and $\delta = s h^{\beta}$ with $s > 0$ small enough, but keeping the notations $\rho , \delta$ in the computations allows for a better understanding of the arguments. We first obtain the resolvent estimates on the imaginary axis and take $z = i \lambda$ with $\lambda \in [ E - \varepsilon , E + \varepsilon ]$. Following \cite{CoGa23_01}, we define the localization function on $X$
\begin{equation} \label{a64}
\chi ( x ) = \phi \Big( \frac{1}{\delta} \sgn \big( V ( x ) - \lambda \big) \dist \big( x , M_{\lambda} \big) \Big) ,
\end{equation}
where $\phi : \R \rightarrow [ - 1 , 1 ]$ is the unique odd function such that $\phi ( t ) = \min ( t , 1 )$ for $t \geq 0$. In particular, $\chi$ is a Lipschitz function which satisfies
\begin{gather}
\vert \chi ( x ) \vert \leq 1 \text{ and } \vert \nabla \chi ( x ) \vert \leq \delta^{- 1}  \text{ for all } x \in X ,   \label{a60}  \\
\chi ( x ) ( V ( x) - \lambda ) \geq 0 \text{ for all } x \in X ,   \label{a61}   \\
\chi ( x ) = \sgn ( V ( x ) - \lambda ) \text{ for all } x \in X \setminus \CM_{\lambda} .   \label{a62}
\end{gather}

For $u \in D ( P )$ we have
\begin{equation*}
\re \< u , ( Q - i \lambda ) u \> = h^{2} \Vert \nabla u \Vert^{2} + \Vert G u \Vert^{2} ,
\end{equation*}
and then
\begin{equation} \label{a10}
h^{2} \Vert \nabla u \Vert^{2} + \Vert G u \Vert^{2} \leq \Vert ( Q - i \lambda ) u \Vert \Vert u \Vert .
\end{equation}
On the other hand, \eqref{a12} and \eqref{a13} give
\begin{align*}
\im \< \chi u , ( Q - i \lambda ) u \> &= \im \< \chi u , - h^{2} \Delta u \> + \< u , \chi ( V - \lambda ) u \>  \\
&= h^{2} \im \< ( \nabla \chi ) u , \nabla u \> + \< u , \chi ( V - \lambda ) u \> .
\end{align*}
Using that $\vert \chi \vert \leq 1$ and $\vert \nabla \chi \vert \leq \delta^{- 1}$, this implies
\begin{align}
\< u , \chi ( V - \lambda ) u \> &\leq \Vert ( Q - i \lambda ) u \Vert \Vert u \Vert + h^{2} \delta^{- 1} \Vert \nabla u \Vert \Vert u \Vert \nonumber \\
&\leq \Vert ( Q - i \lambda ) u \Vert \Vert u \Vert + h \delta^{- 1} \Vert ( Q - i \lambda ) u \Vert^{1 / 2} \Vert u \Vert^{3 / 2} .  \label{a14}
\end{align}

To estimate the norm of $u$ in $L^{2} ( X )$, we decompose $X \subset X_{1} \cup X_{2} \cup X_{3}$ with
\begin{equation*}
X_{1} = \CC_{\max} ( E ) + B ( 0 , R h^{\beta} ) , \quad X_{2} = \CM_{\lambda} \setminus \big( \CC_{\max} ( E ) + B ( 0 , R h^{\beta} ) \big) , \quad X_{3} = X \setminus \CM_{\lambda} ,
\end{equation*}
and write
\begin{equation} \label{a16}
\Vert u \Vert^{2} \leq \CI_{1} + \CI_{2} + \CI_{3} \qquad \text{ where } \qquad \CI_{j} = \int_{X_{j}} \vert u ( x ) \vert^{2} d x .
\end{equation}
From \eqref{a9} and \eqref{a10}, we directly have
\begin{equation} \label{a15}
\CI_{1} \leq \Vert G u \Vert^{2} \leq \Vert ( Q - i \lambda ) u \Vert \Vert u \Vert .
\end{equation}
The contribution of $\CI_{2}$ is estimated by the following result proved in Section \ref{s5}.

\begin{lemma}\sl \label{b30}
For $R$ large enough and then $h$ small enough, we have
\begin{equation*}
\CI_{2} \leq C \delta \Vert u \Vert \Vert \nabla u \Vert + C \delta \Vert u \Vert^{2} .
\end{equation*}
\end{lemma}

\noindent
Combining Lemma \ref{b30} with \eqref{a10}, it comes
\begin{align}
\CI_{2} &\leq C \delta^{2} \Vert \nabla u \Vert^{2}  + \frac{1}{3} \Vert u \Vert^{2}   \nonumber \\
&\leq C \delta^{2} h^{- 2} \Vert ( Q - i \lambda ) u \Vert \Vert u \Vert + \frac{1}{3} \Vert u \Vert^{2}  , \label{a18}
\end{align}
for $R$ large enough and then $h$ small enough. For $x \in X_{3}$, we have $\vert V ( x ) - \lambda \vert \geq S \delta^{\alpha}$ and $\chi ( x ) = \sgn ( V ( x ) - \lambda )$ from \eqref{a63}, \eqref{a59} and \eqref{a62}. Then, \eqref{a61} and \eqref{a14} yield
\begin{align}
\CI_{3} &\leq S^{- 1} \delta^{- \alpha} \< u , \chi ( V - \lambda ) u \> \nonumber  \\
&\leq S^{- 1} \delta^{- \alpha} \Vert ( Q - i \lambda ) u \Vert \Vert u \Vert + S^{- 1} h \delta^{- 1 - \alpha} \Vert ( Q - i \lambda ) u \Vert^{1 / 2} \Vert u \Vert^{3 / 2}  \nonumber \\
&\leq \big( S^{- 1} \delta^{- \alpha} + S^{- 2} h^{2} \delta^{- 2 - 2 \alpha} \big) \Vert ( Q - i \lambda ) u \Vert \Vert u \Vert + \frac{1}{4} \Vert u \Vert^{2} . \label{a17}
\end{align}
Combining \eqref{a16} with the estimates \eqref{a15}, \eqref{a18} and \eqref{a17}, we deduce
\begin{equation*}
\Vert u \Vert^{2} \leq \big( 1 + C \delta^{2} h^{- 2} + S^{- 1} \delta^{- \alpha} + S^{- 2} h^{2} \delta^{- 2 - 2 \alpha} \big) \Vert ( Q - i \lambda ) u \Vert \Vert u \Vert + \frac{7}{12} \Vert u \Vert^{2} .
\end{equation*}
Using $\delta = s h^{\beta}$ and \eqref{a19}, we have
\begin{equation*}
1 + C \delta^{2} h^{- 2} + S^{- 1} \delta^{- \alpha} + S^{- 2} h^{2} \delta^{- 2 - 2 \alpha} = 1 + ( C s^{2} + S^{- 1} s^{- \alpha} + S^{- 2} s^{- 2 - 2 \alpha} ) h^{- \sigma} .
\end{equation*}
Taking $s > 0$ small enough, then $S \geq 1$ large enough and then $R \geq 1$ large enough, the previous equation becomes $\Vert u \Vert^{2} \leq M^{- 1} h^{- \sigma} \Vert ( Q - i \lambda ) u \Vert \Vert u \Vert$ for $h$ small enough. It implies that $Q$ has no spectrum in $i [ E - \varepsilon, E + \varepsilon ]$ and
\begin{equation} \label{a20}
\Vert ( Q - i \lambda )^{- 1} \Vert \leq M^{- 1} h^{- \sigma} ,
\end{equation}
for $\lambda \in [ E - \varepsilon, E + \varepsilon ]$ and $h$ small enough.

To estimate the resolvent of $Q$ in the complex plane, we write
\begin{equation*}
( Q - z ) = \big( 1 - ( z - i \lambda ) ( Q - i \lambda )^{- 1 } \big) ( Q - i \lambda ) ,
\end{equation*}
with $\lambda = \im z$. Together with \eqref{a20}, it shows that, for $h$ small enough, $Q$ has no spectrum in $[ 0 , M h^{\sigma} / 2 ] + i [ E - \varepsilon , E + \varepsilon ]$ and that
\begin{equation} \label{a22}
\Vert ( Q - z )^{- 1} \Vert \leq 2 M^{- 1} h^{- \sigma} ,
\end{equation}
for $z$ in that set. Moreover, using one more time \eqref{a10} leads to
\begin{equation*}
\Vert \nabla ( Q - i \lambda )^{- 1} u \Vert^{2} \leq h^{- 2} \Vert u \Vert \Vert ( Q - i \lambda )^{- 1} u \Vert \leq M^{-1} h^{- 2 - \sigma} \Vert u \Vert^{2} ,
\end{equation*}
and then, for $\lambda \in [ E - \varepsilon, E + \varepsilon ]$,
\begin{equation} \label{a21}
h^{\beta} \Vert \nabla ( Q - i \lambda )^{- 1} \Vert \leq M^{- 1 / 2} h^{- \sigma} ,
\end{equation}
since $1 + \sigma / 2 = \beta + \sigma$. Writing $h^{\beta} \nabla ( Q - z )^{- 1} = h^{\beta} \nabla ( Q - i \lambda )^{- 1} ( 1 - ( z - i \lambda ) ( Q - i \lambda )^{- 1 } )^{- 1}$, \eqref{a20} and \eqref{a21} give
\begin{equation} \label{a25}
h^{\beta} \Vert \nabla ( Q - z )^{- 1} \Vert \leq 2 M^{- 1 / 2} h^{- \sigma} .
\end{equation}
for $z \in [ 0 , M h^{\sigma} / 2 ] + i [ E - \varepsilon , E + \varepsilon ]$. Changing $M$, Proposition \ref{a33} follows from \eqref{a22} and \eqref{a25}.

\subsubsection{Proof of Lemma \ref{b30}} \label{s5}

From the finite covering of Lemma \ref{d37}, we write
\begin{equation*}
\CI_{2} = \int_{X_{2}} \vert u ( x ) \vert^{2} d x \leq \sum_{1 \leq j \leq J} \int_{X_{2} \cap U_{ j}} \vert u ( x ) \vert^{2} d x ,
\end{equation*}
and we estimate in the rest of this part each term of the sum using the description of the set $\CM_{\lambda}$ obtained in Section \ref{s3}.

\underline{Case A} We first consider the contribution to $\CI_{2}$ of a non-critical region $U_{j}$ (Case 1 of Section \ref{s3} and Lemma \ref{d37}). From \eqref{b10} and \eqref{b11}, the set $\CM_{\lambda}$ satisfies
\begin{equation} \label{b14}
\CM_{\lambda} \subset \{ y ; \ \vert y_{1} - ( \lambda - E ) \vert \leq \rho + L \delta \leq C \delta \} ,
\end{equation}
in the variables $y = \varphi^{- 1} ( x )$. Here, we used that $\rho = ( S \delta^{\alpha - 1} ) \delta \leq \delta$ for $h$ small enough. Let $\check{u} ( y ) \in H^{1} ( \R^{d} )$ denote an extension to $\varphi^{- 1} ( Y )$ (see the beginning of Section \ref{s1}) of $u ( \varphi^{- 1} ( y ) )$ truncated near $\varphi^{- 1} ( \overline{U_{j}} )$. We now use the following estimate

\begin{lemma}[Lemma B.1 of \cite{CoGa23_01}]\sl \label{b18}
for all $g \in H^{1} ( \R^{d} )$ and $0 \leq R_{1} \leq R_{2}$, we have
\begin{equation*}
\int_{R_{1} \leq \vert x \vert \leq R_{2}} \vert g  ( x ) \vert^{2} d x \leq 2 ( R_{2} - R_{1} ) \Vert g \Vert \Vert \nabla g \Vert .
\end{equation*}
\end{lemma}

\noindent
Using that the change of variables $x = \varphi ( y )$ is $C^{1}$, $\Vert \check{u} \Vert_{H^{1}} \leq C \Vert u \Vert_{H^{1}}$, \eqref{b14}, the Fubini theorem, the Cauchy--Schwarz inequality and Lemma \ref{b18} applied in the variable $y_{1}$, the contribution of $\CI_{2}$ in $U_{j}$ satisfies
\begin{align}
\int_{\CM_{\lambda} \cap U_{j}} \vert u ( x ) \vert^{2} d x &\leq C \int_{\varphi^{- 1} ( \CM_{\lambda} ) \cap \varphi^{- 1} ( U_{j} )} \vert \check{u} ( y ) \vert^{2} d y \nonumber \\
& \leq C \int d y_{2} \cdots d y_{d} \int_{\vert y_{1} - ( \lambda - E ) \vert \leq C \delta} \vert \check{u} ( y ) \vert^{2} d y_{1}   \nonumber \\
&\leq C \delta \int \Vert \check{u} \Vert_{y_{1}} \Vert \nabla \check{u} \Vert_{y_{1}} d y_{2} \cdots d y_{d} \leq C \delta \Vert \check{u} \Vert \Vert \nabla \check{u} \Vert   \nonumber  \\
&\leq  C \delta \Vert u \Vert \Vert \nabla u \Vert + C \delta \Vert u \Vert^{2}  . \label{b12}
\end{align}

\underline{Case B} We now study the contribution to $\CI_{2}$ of a neighborhood of a critical point $c \in \CC ( E ) \setminus \CC_{\max} ( E )$ and we use the spherical variables $x = c + r \theta$. Since $\S^{d - 1}$ is compact, it is enough to estimate this contribution in the angular directions close to a finite number of angles $\theta_{0} \in \S^{d -1}$. Following Section \ref{s3}, we distinguish between $v_{c} ( \theta_{0} ) \neq 0$ and $v_{c} ( \theta_{0} ) = 0$.

\underline{Case B.1}  Assume that $v_{c} ( \theta_{0} ) \neq 0$. In the sequel we suppose that $v_{c} ( \theta_{0} ) > 0$ (the other configuration can be treated in the same way) and we estimate the contribution to $\CI_{2}$ of a critical region $U_{j}$ as in Case 2.1 of Section \ref{s3} and Lemma \ref{d37}. We first consider the case $\lambda \leq \delta^{\alpha_{c}}$. Since $\alpha > \alpha_{c}$, we have $\lambda + \rho \leq 2 \delta^{\alpha_{c}}$ for $h$ small enough. Moreover, \eqref{a57} and $v_{c} ( \theta_{0} ) > 0$ give
\begin{equation} \label{d11}
V ( x ) \geq r^{\alpha_{c}} / C ,
\end{equation}
for $x = c + r \theta$ with $r$ near $0$ and $\theta$ near $\theta_{0}$. Then, a point $c+ r \theta \in M_{\lambda}$ as before satisfies $r \leq C ( \lambda + \rho )_{+}^{1 / \alpha_{c}} \leq C \delta$. It implies that a point $c+ r \theta \in \CM_{\lambda}$ satisfies $r \leq ( C + 1 ) \delta$. Here, $\CM_{\lambda}$ is restricted to the points at distance at most $\delta$ from the part of $M_{\lambda}$ with $\theta$ near $\theta_{0}$. Consider $\zeta_{j} \in C^{\infty}_{0} ( \R^{d} )$ with $\zeta_{j} = 1$ near $U_{j}$. Applying Lemma \ref{b18} to $\zeta_{j} u \in H^{1} ( \R^{d} )$ gives
\begin{align}
\int_{\CM_{\lambda} \cap U_{j}} \vert u ( x ) \vert^{2} d  x &\leq \int_{\vert x \vert \leq ( C + 1 ) \delta} \big\vert \zeta_{j} u ( x ) \big\vert^{2} d x  \leq C \delta \big\Vert \zeta_{j} u \big\Vert  \big\Vert \nabla \big( \zeta_{j} u \big) \big\Vert  \nonumber \\
&\leq C \delta \Vert u \Vert \Vert \nabla u \Vert + C \delta \Vert u \Vert^{2} ,   \label{b19}
\end{align}
if $\lambda \leq \delta^{\alpha_{c}}$. Assume now that $\lambda > \delta^{\alpha_{c}}$. Since $\rho \leq \delta^{\alpha_{c}} / 2$ for $h$ small enough because $\alpha > \alpha_{c}$, we have $\rho < \lambda / 2$ and we can apply \eqref{d8} which shows that a point $x \in \CM_{\lambda}$ satisfies
\begin{equation} \label{d9}
\vert \widehat{r}  - \lambda^{1 / \alpha_{c}} \vert \leq C \delta + C \rho \lambda^{\frac{1 - \alpha_{c}}{\alpha_{c}}} \leq C \delta ,
\end{equation}
where $\widehat{x} = \widehat{r} \widehat{\theta}$ are the spherical coordinates of $x$ after the change of variables \eqref{d1}. Here, we used that $\rho \lambda^{\frac{1 - \alpha_{c}}{\alpha_{c}}} \leq \delta^{\alpha_{c}} \delta^{1 - \alpha_{c}} \leq \delta$ for $h$ small enough.
Using that the Jacobian of the change of variables \eqref{d1} and its inverse are uniformly bounded and \eqref{d10}, Lemma \ref{b18} and \eqref{d9} give
\begin{align}
\int_{\CM_{\lambda} \cap U_{j}} \vert u ( x ) \vert^{2} d  x &= \int_{\CM_{\lambda} \cap U_{j}} \big\vert \zeta_{j} u ( x ) \big\vert^{2} d  x   \nonumber \\
&\leq C \int_{\lambda^{1 / \alpha_{c}} - C \delta \leq \vert \widehat{x} \vert \leq \lambda^{1 / \alpha_{c}} + C \delta} \big\vert \widehat{\zeta_{j}} \widehat{u} ( \widehat{x} ) \big\vert^{2} d  \widehat{x}   \nonumber \\
&\leq C \delta \big\Vert \widehat{\zeta_{j}} \widehat{u} \big\Vert \big\Vert \nabla \big( \widehat{\zeta_{j}} \widehat{u} \big) \big\Vert \leq C \delta \big\Vert \zeta_{j} u \big\Vert \big\Vert \nabla \big( \zeta_{j} u \big) \big\Vert \nonumber \\
&\leq  C \delta \Vert u \Vert \Vert \nabla u \Vert + C \delta \Vert u \Vert^{2}  . \label{b22}
\end{align}
Note that $\widehat{\zeta_{j}} \widehat{u} \in H^{1} ( \R^{d} )$ since the change of variables \eqref{d1} is Lipschitz continuous near $c$. Summing up, \eqref{b19} and \eqref{b22} show that we always have
\begin{equation} \label{b23}
\int_{\CM_{\lambda} \cap U_{j}} \vert u ( x ) \vert^{2} d  x \leq C \delta \Vert u \Vert \Vert \nabla u \Vert + C \delta \Vert u \Vert^{2} ,
\end{equation}
if $v_{c} ( \theta_{0} ) \neq 0$.

\underline{Case B.2}  We now consider the situation where $v_{c} ( \theta_{0} ) = 0$ and we work in a critical region $U_{j}$ as in Case 2.2 of Section \ref{s3} and Lemma \ref{d37}. We first estimate the contribution of $r \leq ( T + 1 ) \delta$ with $T \geq 1$ large enough. Let $\zeta_{j} \in C^{\infty}_{0} ( \R^{d} )$ be such that $\zeta_{j} = 1$ near $U_{j}$. Applying Lemma \ref{b18} and considering $T$ as being fixed, we directly obtain
\begin{align}
\int_{\CM_{\lambda} \cap U_{j} \cap \{ r \leq ( T + 1 ) \delta \}} \vert u ( x ) \vert^{2} d  x &\leq \int_{r \leq ( T + 1 ) \delta} \vert \zeta_{j} u ( x ) \vert^{2} d  x \leq C \delta \big\Vert \zeta_{j} u \big\Vert \big\Vert \nabla \big( \zeta_{j} u \big) \big\Vert   \nonumber \\
&\leq C \delta \Vert u \Vert \Vert \nabla u \Vert + C \delta \Vert u \Vert^{2} .  \label{b28}
\end{align}
It remains to estimate the contribution of $r > ( T + 1 ) \delta$. Since $\rho^{1 / \alpha_{c}} \leq \delta$ for $h$ small enough, we have $r > T \max ( \rho^{1 / \alpha_{c}} , \delta ) + \delta$ and such a point in $\CM_{\lambda}$ is at distance at most $\delta$ from a point $y \in M_{\lambda}$ with $\vert y \vert > T \max ( \rho^{1 / \alpha_{c}} , \delta )$. Then, we can apply \eqref{d31} which shows that such a point $x \in \CM_{\lambda}$ satisfies
\begin{equation} \label{d32}
\vert  \widetilde{x}_{2}  - \lambda \widetilde{x}_{1}^{1 - \alpha_{c}} \vert \leq C \delta + C \rho \widetilde{r}^{1 - \alpha_{c}} \leq C \delta ,
\end{equation}
where $\widetilde{x}$ are the coordinates of $x$ after the change of variables \eqref{d13}. Here, we used that $\rho \widetilde{r}^{1 - \alpha_{c}} \leq S \delta^{\alpha} ( T + 1 )^{1 - \alpha_{c}} \delta^{1 - \alpha_{c}} \leq \delta$ for $h$ small enough since $\alpha > \alpha_{c}$. Then, the Fubini theorem, the Cauchy--Schwarz inequality and Lemma \ref{b18} in the variable $\widetilde{x}_{2}$ give
\begin{align}
\int_{\CM_{\lambda} \cap U_{j} \cap \{ r > ( T + 1 ) \delta \}} \vert u ( x ) \vert^{2} d  x &= \int_{\CM_{\lambda} \cap U_{j} \cap \{ r > ( T + 1 ) \delta \}} \big\vert \zeta_{j} u ( x ) \big\vert^{2} d  x    \nonumber \\
&\leq C \int d \widetilde{x}_{1}  \, d \widetilde{x}_{3}  \cdots d \widetilde{x}_{d} \int_{\vert  \widetilde{x}_{2}  - \lambda \widetilde{x}_{1}^{1 - \alpha_{c}} \vert \leq C \delta}  \big\vert \widetilde{\zeta_{j}} \widetilde{u} ( \widetilde{x} ) \big\vert^{2} d  \widetilde{x}_{2}   \nonumber \\
&\leq C \delta \int \big\Vert \widetilde{\zeta_{j}} \widetilde{u} \big\Vert_{\widetilde{x}_{2}} \big\Vert \nabla \big( \widetilde{\zeta_{j}} \widetilde{u} \big) \big\Vert_{\widetilde{x}_{2}} d \widetilde{x}_{1}  \, d \widetilde{x}_{3}  \cdots d \widetilde{x}_{d}   \nonumber  \\
&\leq C \delta \big\Vert \widetilde{\zeta_{j}} \widetilde{u} \big\Vert \big\Vert \nabla \big( \widetilde{\zeta_{j}} \widetilde{u} \big) \big\Vert \leq C \delta \big\Vert \zeta_{j} u \big\Vert \big\Vert \nabla \big( \zeta_{j} u \big) \big\Vert    \nonumber  \\
&\leq  C \delta \Vert u \Vert \Vert \nabla u \Vert + C \delta \Vert u \Vert^{2}  . \label{b29}
\end{align}
Note that $\widetilde{\zeta_{j}} \widetilde{u} \in H^{1} ( \R^{d} )$ since the change of variables \eqref{d1} is Lipschitz continuous near $c$. Summing up,  \eqref{b28} and \eqref{b29} show that
\begin{equation} \label{b27}
\int_{\CM_{\lambda} \cap U_{j}} \vert u ( x ) \vert^{2} d  x \leq C \delta \Vert u \Vert \Vert \nabla u \Vert + C \delta \Vert u \Vert^{2} ,
\end{equation}
if $v_{c} ( \theta_{0} ) = 0$.

\underline{Case C} It remains to study the contribution to $\CI_{2}$ of a neighborhood of a critical point $c \in \CC_{\max} ( E )$. As before, we use the spherical variables $x = c + r \theta$ and work in the angular directions close to a finite number of angles $\theta_{0} \in \S^{d -1}$. In $X_{2}$, we have $r > R h^{\beta} = R s^{- 1} \delta$.

\underline{Case C.1} We assume that $v_{c} ( \theta_{0} ) \neq 0$ and we work in a critical region $U_{j}$ as in Case 2.1 of Section \ref{s3} and Lemma \ref{d37}. In the sequel we suppose that $v_{c} ( \theta_{0} ) > 0$, the other case can be treated in the same way. We first consider the situation $\lambda \leq R^{\alpha / 2} \delta^{\alpha}$. For $R$ large enough (depending on $S$), we have $\lambda + \rho \leq 2 R^{\alpha / 2} \delta^{\alpha}$. Then, \eqref{d11} shows that a point $c+ r \theta \in M_{\lambda}$ satisfies $r \leq C ( \lambda + \rho )_{+}^{1 / \alpha} \leq C R^{1 / 2} \delta$.  It implies that a point $c+ r \theta \in \CM_{\lambda}$ satisfies $r \leq ( C R^{1 / 2} + 1 ) \delta \leq R s^{- 1} \delta$ for $R$ large enough, where $\CM_{\lambda}$ is restricted to the points at distance at most $\delta$ from the part of $M_{\lambda}$ with $\theta$ near $\theta_{0}$. In other words, this region is not in $X_{2}$ and gives no contribution to $\CI_{2}$. Assume now that $\lambda > R^{\alpha / 2} \delta^{\alpha}$. For $R$ large enough, we have $\rho = S R^{- \alpha / 2} R^{\alpha / 2}\delta^{\alpha} < \lambda / 2$ and we can apply \eqref{d8} which shows that a point $x \in \CM_{\lambda}$ satisfies
\begin{equation*}
\vert \widehat{r}  - \lambda^{1 / \alpha} \vert \leq C \delta + C \rho \lambda^{\frac{1 - \alpha}{\alpha}} \leq C \delta ,
\end{equation*}
where $\widehat{x} = \widehat{r} \widehat{\theta}$ are the spherical coordinates of $x$ after the change of variables \eqref{d1}. Here, we used that $\rho \lambda^{\frac{1 - \alpha}{\alpha}} \leq S \delta^{\alpha} R^{\frac{1 - \alpha}{2}} \delta^{1 - \alpha} \leq \delta$ for $R$ large enough. Then, we are in a position similar to \eqref{d9}, and we can conclude as in \eqref{b22} that
\begin{equation}
\int_{X_{2} \cap U_{j}} \vert u ( x ) \vert^{2} d  x \leq C \delta \Vert u \Vert \Vert \nabla u \Vert + C \delta \Vert u \Vert^{2} ,
\end{equation}
if $v_{c} ( \theta_{0} ) \neq 0$.

\underline{Case C.2} The last situation to deal with is when $v_{c} ( \theta_{0} ) = 0$ (Case 2.2 of Section \ref{s3}). Since we work in $X_{2}$, we have $r \geq R \delta > T \max ( \rho^{1 / \alpha} , \delta ) + \delta = T S^{1 / \alpha} \delta + \delta $ for $R$ large enough. In particular, a point $x \in X_{2}$ is at distance at most $\delta$ from a point $y \in M_{\lambda}$ with $r_{y} \geq T \max ( \rho^{1 / \alpha} , \delta ) $. Thus, we can apply \eqref{d31} which shows that a point $x \in \CM_{\lambda}$ satisfies
\begin{equation}
\vert  \widetilde{x}_{2}  - \lambda \widetilde{x}_{1}^{1 - \alpha_{c}} \vert \leq  C \delta + C \rho \widetilde{r}^{1 - \alpha} \leq C \delta ,
\end{equation}
where $\widetilde{x}$ are the coordinates of $x$ after the change of variables \eqref{d13}. Here, we used that $r = \widetilde{r}$ and $\rho \widetilde{r}^{1 - \alpha} \leq S \delta^{\alpha} R^{1 - \alpha} \delta^{1 - \alpha} \leq \delta$ for $R$ large enough. Then, we are in a position similar to \eqref{d32}, and we can conclude as in \eqref{b27} that
\begin{equation}
\int_{X_{2} \cap U_{j}} \vert u ( x ) \vert^{2} d  x \leq C \delta \Vert u \Vert \Vert \nabla u \Vert + C \delta \Vert u \Vert^{2} ,
\end{equation}
if $v_{c} ( \theta_{0} ) = 0$.

\section{The case of monomial potential}

This section is devoted to the model operator $\CP = - \Delta + i \CV$ introduced in \eqref{a44}.

\subsection{Maximal accretivity} \label{s7}

We prove here Proposition \ref{a46}. A direct computation gives
\begin{equation*}
\re \< \CP u , u \> = \Vert \nabla u \Vert^{2} \geq 0 ,
\end{equation*}
for all $u \in C_{0}^{\infty} ( \R^{d} )$. Thus, $\CP$ defined on $C_{0}^{\infty} ( \R^{d} )$ is an accretive operator. Since $C_{0}^{\infty} ( \R^{d} )$ is dense in $L^{2} ( \R^{d} )$, Proposition 13.12 of \cite{He13_01} implies that $\CP$ is closable.

To show that its closed extension is maximal accretive, it is enough to show that $\ker ( \CP^{*} + a ) = \{ 0 \}$ for some $a > 0$ (see Theorem 13.14 of \cite{He13_01}). Then, suppose that $u \in \ker ( \CP^{*} + a )$. In particular, $u \in L^{2} ( \R^{d} )$ and $( \CP^{*} + a ) u = ( - \Delta - i \CV + a ) u = 0$ in the sense of distributions. The elliptic regularity implies that $u \in H^{2}_{\rm loc} ( \R^{d} )$. We will prove that $u = 0$ using a trick of spectral theory. We consider $\zeta \in C^{\infty}_{0} ( \R^{d} ; [ 0 , 1 ] )$ such that $\zeta = 1$ near $0$ and denote $\zeta_{k} ( x ) = \zeta ( x / k )$ for $k \in \N$. For $f \in H^{2}_{\rm loc} ( \R^{d} )$, a standard identity (see e.g. Equation (9.4.5) of \cite{He13_01}) gives
\begin{align}
0 ={}& \int \big( ( - \Delta + i \CV + a ) u \big) \zeta_{k}^{2} f \, d x = \int u \big( ( - \Delta + i \CV + a ) \zeta_{k}^{2} f \big) \, d x   \nonumber \\
={}& \int \nabla ( \zeta_{k} u ) \cdot \nabla ( \zeta_{k} f ) \, d x + \int \zeta_{k}^{2} ( i \CV + a ) u f \, d x   \nonumber \\
&- \int ( \nabla \zeta_{k} )^{2} u f \, d x + \sum_{1 \leq j \leq d} \int ( f \partial_{x_{j}} u - u \partial_{x_{j}} f ) \zeta_{k} \partial_{x_{j}} \zeta_{k} \, d x . \label{a87}
\end{align}
Applying this relation with $f = \overline{u}$ and taking the real part lead to
\begin{equation*}
0 = \Vert \nabla ( \zeta_{k} u ) \Vert^{2} + \int a \zeta_{k}^{2} \vert u \vert^{2} d x - \Vert ( \nabla \zeta_{k} ) u \Vert^{2} ,
\end{equation*}
and then
\begin{equation*}
a \Vert \zeta_{k}u \Vert^{2} \leq \Vert ( \nabla \zeta_{k} ) u \Vert^{2} .
\end{equation*}
Taking the limit $k \to + \infty$, we obtain
\begin{equation} \label{a89}
\Vert u \Vert^{2} = \lim_{k \to + \infty} \Vert \zeta_{k} u \Vert^{2} \leq \limsup_{k \to + \infty} \Vert ( \nabla \zeta_{k} ) u \Vert^{2} / a = 0 ,
\end{equation}
which shows that $u = 0$ and that (the closed extension of) $\CP$ is maximal accretive. In particular, the resolvent $( \CP  - z )^{- 1}$ is well defined for $\re z < 0$.

\begin{lemma}\sl \label {c1}
The operator $\CP$ has a compact resolvent
\end{lemma}

\begin{proof}
We use some arguments of the proof of Proposition \ref{a33}. For $u \in C^{\infty}_{0} ( \R^{d} )$, we have
\begin{equation} \label{c2}
\Vert \nabla u \Vert^{2} = \re \< u , \CP u \> \leq \Vert \CP u \Vert \Vert u \Vert ,
\end{equation}
which shows that $D ( \CP ) \subset H^{1} ( \R^{d} )$ and that \eqref{c2} holds for all $u \in D ( \CP )$. Mimicking \eqref{a64}, we define the Lipschitz function
\begin{equation} \label{c3}
\chi ( x ) = \phi \big( \sgn ( \CV ( x ) ) \dist \big( x , \CV^{- 1} ( 0 ) \big) \big) ,
\end{equation}
which satisfies $\vert \chi ( x ) \vert \leq 1$ and $\vert \nabla \chi ( x ) \vert \leq 1$. As in \eqref{a14}, one can write
\begin{equation*}
\im \< \chi u , \CP u \> = \im \< ( \nabla \chi ) u , \nabla u \> + \< u , \chi \CV u \> ,
\end{equation*}
which yields
\begin{equation} \label{c4}
\< u , \chi \CV u \> \leq \Vert \CP u \Vert \Vert u \Vert + \Vert \CP u \Vert^{1 / 2} \Vert u \Vert^{3 / 2} .
\end{equation}

We claim that, for all $\varepsilon > 0$, there exists $R \geq 1$ large enough such that
\begin{equation} \label{c5}
\int_{\vert x \vert \geq R} \vert u ( x )  \vert^{2} d x \leq \varepsilon ,
\end{equation}
for all $u \in D ( \CP )$ with $\Vert \CP u \Vert + \Vert u \Vert \leq 1$. To prove that, as in Section \ref{s3}, we use the spherical coordinates $x = r \theta$ and work near $\theta_{0} \in \S^{d - 1}$. If $v ( \theta_{0} ) \neq 0$, then $\chi \CV ( r \theta ) = \vert \CV ( r \theta ) \vert \geq r^{\alpha} / C$ for $\theta$ near $\theta_{0}$. Then \eqref{c4} gives
\begin{equation} \label{c6}
\int_{r \geq R \text{ and } \theta \text{ near } \theta_{0}} \vert u ( x )  \vert^{2} d x \leq C R^{- \alpha} \< u , \chi \CV u \> \leq C R^{- \alpha} .
\end{equation}
If now $v ( \theta_{0} ) = 0$, we use the changes of variables of Case 2.2 of Section \ref{s3}. Since $\xi ( r , \theta ) = 0$ in the present setting, the variables $\widetilde{x}$ constructed in \eqref{d13} provide a global diffeomorphism on $\R^{d} \setminus \{ 0 \}$. Thus, for $r \geq R$ and $\theta$ near $\theta_{0}$, we have $\CV ( x ) = \widetilde{x}_{2} \widetilde{x}_{1}^{\alpha - 1}$, $\CV^{- 1} ( 0 ) = \{ \widetilde{x}_{2} = 0 \}$, $\dist ( x , \CV^{- 1} ( 0 ) ) = \vert \widetilde{x}_{2} \vert$ and $\vert ( \widetilde{x}_{2} , \ldots , \widetilde{x}_{d} ) \vert \leq \widetilde{x}_{1} / C$. Then, $\chi \CV ( x ) \geq \min ( \widetilde{x}_{2}^{2} , \vert \widetilde{x}_{2} \vert ) \widetilde{x}_{1}^{\alpha -1}$. In particular, $\chi \CV ( x) \geq R^{( \alpha - 1) / 3} / C$ for $\vert \widetilde{x}_{2} \vert \geq R^{( 1 -\alpha ) / 3}$ and $\vert \widetilde{x} \vert = r \geq R$. Then, \eqref{c4} implies
\begin{equation} \label{c8}
\int_{r \geq R , \ \theta \text{ near } \theta_{0} \text{ and } \vert \widetilde{x}_{2} \vert \geq R^{( 1 -\alpha ) / 3}} \vert u ( x )  \vert^{2} d x \leq C R^{( 1 - \alpha ) / 3} \< u , \chi \CV u \> \leq C R^{( 1 - \alpha ) / 3} .
\end{equation}
Finally, Lemma \ref{b18}, the Fubini theorem and the properties of the change of variables $x \longmapsto \widetilde{x}$ give
\begin{align}
\int_{r \geq R , \ \theta \text{ near } \theta_{0} \text{ and } \vert \widetilde{x}_{2} \vert \leq R^{( 1 -\alpha ) / 3}} &\leq \int_{\widetilde{x}_{1} \geq R / C} d \widetilde{x}_{1}  \, d \widetilde{x}_{3}  \cdots d \widetilde{x}_{d} \int_{\vert \widetilde{x}_{2} \vert \leq R^{( 1 -\alpha ) / 3}} \vert \widetilde{u} ( \widetilde{x} )  \vert^{2} d\widetilde{x}_{2}  \nonumber \\
&\leq C R^{( 1 -\alpha ) / 3} \int_{\widetilde{x}_{1} \geq R / C} \Vert \widetilde{u} \Vert_{\widetilde{x}_{2}} \Vert \nabla \widetilde{u} \Vert_{\widetilde{x}_{2}}  d \widetilde{x}_{1}  \, d \widetilde{x}_{3}  \cdots d \widetilde{x}_{d}  \nonumber \\
&\leq C R^{( 1 -\alpha ) / 3} .  \label{c9}
\end{align}
Summing up, \eqref{c5} is a consequence of \eqref{c6}, \eqref{c8} and \eqref{c9}.

Using that $D ( \CP ) \subset H^{1} ( \R^{d} )$ and \eqref{c5}, the compacity of the resolvent of $\CP$ follows from the theorem of Kolmogorov--M. Riesz--Fr\'echet (see Corollary 4.27 and Proposition 9.3 of \cite{Br11_01}).
\end{proof}

It remains to show that $\sigma ( \CP ) \subset \{ z \in \C ; \ \re z > 0 \}$. The accretivity of $\CP$ yields $\sigma ( \CP ) \subset \{ z \in \C ; \ \re z \geq 0 \}$. Assume that $\CP$ has an eigenvalue $z \in i \R$ associated to an eigenvector $u \in D ( \CP )$. An integration by parts gives
\begin{equation*}
0 = \re \< ( \CP - z ) u , u \> = \Vert \nabla u \Vert^{2} ,
\end{equation*}
showing that $u = 0$. Thus, $\sigma ( \CP ) \subset \{ z \in \C ; \ \re z > 0 \}$ and Proposition \ref{a46} follows.

\subsection{Operators with signed potential}

In order to prove Proposition \ref{a65}, we need to study a slightly larger class of operators. Given $\rho \in \C^{*}$, we introduce the operator 
\begin{equation} \label{a50}
\CP^{\rho} = - \Delta + \rho \CV ( x ) ,
\end{equation}
acting on $L^{2} ( \R^{d} )$ with domain $D ( \CP^{\rho} ) = \{ u \in H^{2} ( \R^{d} ) ; \ \vert x \vert^{\alpha} u \in L^{2} ( \R^{d} ) \}$. In particular, $\CP = \CP^{i}$.

\begin{lemma}\sl \label{a69}
If $\CV$ has sign $\pm$ and $\rho \in \C \setminus \R^{\mp}$, the operator $\CP^{\rho}$ with domain $C_{0}^{\infty} ( \R^{d} )$ is closable. Its closed extension, still denoted $\CP^{\rho}$, has domain $D ( \CP^{\rho} ) = \{ u \in H^{2} ( \R^{d} ) ; \ \vert x \vert^{\alpha} u \in L^{2} ( \R^{d} ) \}$. If moreover $\pm \re \rho \geq 0$, then $\CP^{\rho}$ is maximal accretive.
\end{lemma}

\begin{proof}
We start with the closedness of the operator for which we follow the arguments of \cite{He13_01}, Section 14.2. We first show that for $\rho \in \C \setminus \R^{\mp}$ there exists $C > 0$ such that, for all $u \in C_{0}^{\infty} ( \R^{d} )$, one has
\begin{equation} \label{a70}
\Vert \Delta u \Vert^{2} + \Vert \vert x \vert^{\alpha} u \Vert^{2} \leq C ( \Vert \CP^{\rho} u \Vert^{2} + \Vert u \Vert^{2} ) .
\end{equation}
A direct computation gives
\begin{equation*}
\Vert \CP^{\rho} u \Vert^{2} = \Vert \Delta u \Vert^{2} + \vert \rho \vert^{2} \Vert \CV u \Vert^{2} - 2 \re \big( \rho \< \CV u , \Delta u \> \big) .
\end{equation*}
Denoting $\CF_{\rho} ( u ) = \Vert \Delta u \Vert^{2} + \vert \rho \vert^{2} \Vert \CV u \Vert^{2}$, it writes
\begin{equation*}
\Vert \CP^{\rho} u \Vert^{2} = \CF_{\rho} ( u ) - 2 ( \re \rho ) \re \< \CV u , \Delta u \> + 2 ( \im \rho ) \im \< \CV u , \Delta u \> .
\end{equation*}
Integrating by parts in the last term and taking advantage of the imaginary part, we get
\begin{equation*}
\Vert \CP^{\rho} u \Vert^{2} = \CF_{\rho} ( u ) - 2 ( \re \rho ) \re \< \CV u , \Delta u \> - 2 ( \im \rho ) \im \< ( \nabla \CV ) u , \nabla u \> .
\end{equation*}
Using that $\nabla \CV$ is homogeneous of order $\alpha - 1$, we have $\vert \nabla \CV ( x ) \vert \lesssim \vert x \vert^{\alpha - 1}$. On the other hand, $\vert \CV ( x ) \vert \gtrsim \vert x \vert^{\alpha}$ since $\CV$ has a sign. Then, given any parameter $\delta > 0$, there exists a constant $C = C_{\delta , \rho} > 0$ such that
\begin{align}
\vert \< ( \nabla \CV ) u , \nabla u \> \vert &\leq \Vert ( \nabla \CV ) u \Vert^{2} + \Vert \nabla u \Vert^{2} \leq \delta \vert \rho \vert^{2} \Vert \CV u \Vert^{2} + \delta \Vert \Delta u \Vert^{2} + C \Vert u \Vert^{2}   \nonumber  \\
&\leq \delta \CF_{\rho} ( u ) + C \Vert u \Vert^{2} ,  \label{a71}
\end{align}
since $\rho \neq 0$. Applying this estimate, we have for all $\nu > 0$
\begin{equation} \label{a72}
\Vert \CP^{\rho} u \Vert^{2} + C \Vert u \Vert^{2} \geq ( 1 - \nu ) \CF_{\rho} ( u ) - 2 ( \re \rho ) \re \< \CV u , \Delta u \> ,
\end{equation}
where $C$ depends on $\rho$ and $\nu$. Suppose first that $\rho \notin \R$. Using $2 a b \leq a^{2} + b^{2}$, \eqref{a72} yields
\begin{align*}
\Vert \CP^{\rho} u \Vert^{2} + C \Vert u \Vert^{2} &\geq ( 1 - \nu ) \CF_{\rho} ( u ) - 2 \sqrt{\frac{\vert \re \rho \vert}{\vert \rho \vert}} \vert \rho \vert \Vert \CV u \Vert \sqrt{\frac{\vert \re \rho \vert}{\vert \rho \vert}} \vert \Vert \Delta u \Vert  \\
&\geq \Big( 1 - \nu - \frac{\vert \re \rho \vert}{\vert \rho \vert} \Big) \CF_{\rho} ( u ) .
\end{align*}
For $\rho \notin \R$, we have $\vert \re \rho \vert < \vert \rho \vert$. Then, taking $\nu$ small enough, we get \eqref{a70} in that situation. It remains to consider the case $\rho \in \R^{\pm , *}$. An integration by parts in \eqref{a72} gives
\begin{align*}
\Vert \CP^{\rho} u \Vert^{2} + C \Vert u \Vert^{2} &\geq ( 1 - \nu ) \CF_{\rho} ( u ) + 2 \rho \Big( \int \CV \vert \nabla u \vert^{2} d x + \re \< ( \nabla \CV ) u , \nabla u \> \Big)   \\
&\geq ( 1 - \nu ) \CF_{\rho} ( u ) + 2 \rho \re \< ( \nabla \CV ) u , \nabla u \> ,
\end{align*}
since $\rho \CV \geq 0$. Using again \eqref{a71}, this implies $\Vert \CP^{\rho} u \Vert^{2} + C \Vert u \Vert^{2}\geq ( 1 - 2 \nu ) \CF_{\rho} ( u )$ for some new constant $C > 0$, which finishes the proof of \eqref{a70}.

On the other hand, we have
\begin{equation} \label{a88}
\Vert \CP^{\rho} u \Vert \leq \Vert \Delta u \Vert^{2} + C \Vert \vert x \vert^{\alpha} u \Vert^{2} ,
\end{equation}
since $v$ is bounded. Combining \eqref{a70} and \eqref{a88}, we deduce that $\CP^{\rho}$ is closable and that its closed extension has domain $D ( \CP^{\rho} ) = \{ u \in H^{2} ( \R^{d} ) ; \ \vert x \vert^{\alpha} u \in L^{2} ( \R^{d} ) \}$. It remains to show the maximal accretivity. Since 
\begin{equation*}
\re \< \CP^{\rho} u , u \> = \Vert \partial_{x} u \Vert^{2} + ( \re \rho ) \int \CV ( x ) \vert u(x) \vert^{2} d x ,
\end{equation*}
the accretivity of $\CP^{\rho}$ is immediate if $\pm \re \rho \geq 0$. To prove the maximal accretivity, it is enough to adapt the proof of Proposition \ref{a46}. Equation \eqref{a88} still hold with $i \CV$ replaced by $\rho \CV$. Using $\re ( \rho \CV ) \geq 0$, we can conclude as in \eqref{a89}. This finishes the proof of the lemma.
\end{proof}

To prove Proposition \ref{a65} it remains to compute the spectrum of $\CP$. For that, we follow the method of complex dilations and we use ideas coming from the theory of resonances (see \cite{AgCo71_01}). The complex dilation is defined by
\begin{equation*}
( \CU_{\theta} f ) ( x ) = e^{i \theta / 2} f ( e^{i \theta} x ) .
\end{equation*}
The operator $\CU_{\theta}$ is well defined, is unitary on $L^{2}( \R^{d} )$ for $\theta \in i \R$, and requires the analyticity of $f$ to make sense for $\theta \notin i \R$. Formally, $\CU_{\theta}^{- 1} = \CU_{- \theta}$ and $\CU_{\theta}^{*} = \CU_{\overline{\theta}}$. Since $\CU_{\theta} \CV \CU_{\theta}^{- 1} = e^{i \alpha \theta} \CV$ for $\theta \in i \R$, the dilated operator is defined as
\begin{equation*}
\SQ_{\theta} = - e^{- 2 i \theta} \Delta \pm e^{i \alpha \theta} \CV = e^{- 2 i \theta} ( - \Delta \pm e^{ i ( \alpha + 2) \theta} \CV ) ,
\end{equation*}
which satisfies formally $\SQ_{\theta} = \CU_{\theta} \CQ  \CU_{\theta}^{- 1}$. Note that
\begin{equation} \label{b85}
\SQ_{0} = \CQ \qquad \text{and} \qquad \SQ_{\theta_{0}} = e^{\mp i \frac{\pi}{\alpha + 2}} \CP \quad \text{with} \quad \theta_{0} = \pm \frac{\pi}{2 ( \alpha + 2 )} .
\end{equation}
Thus, we will work for $\theta$ near $[ 0 , \theta_{0} ]$. Thanks to Lemma \ref{a69}, the operator $\SQ_{\theta}$ is closed on $D ( \CP )$ for such $\theta$. Since moreover the coefficients of $\SQ_{\theta}$ are holomorphic in $\theta$, $\SQ_{\theta}$ is a holomorphic family of type (A) of Kato (see Section VII.2 of \cite{Ka76_01}).

Since $\CP$ is maximal accretive, $- 1 \notin \sigma ( \SQ_{\theta_{0}} )$. By the perturbation theory (see Theorem VII.1.3 of \cite{Ka76_01}), there exists $\varepsilon > 0$ such that $\sigma ( \SQ_{\theta} ) \cap B ( - 1 , \varepsilon ) = \emptyset$ for all $\theta \in B ( \theta_{0} , \varepsilon )$. On the other hand, $\re e^{ i ( \alpha + 2) \theta} > 0$ for all $\theta \in [ 0 , \theta_{0}[$. Using again Lemma \ref{a69} and combining with the previous argument, there exists (a new) $\varepsilon > 0$ such that
\begin{equation}
\sigma ( \SQ_{\theta} ) \cap B ( - 1 , \varepsilon ) = \emptyset \text{ for all } \theta \in \Theta = [ 0 , \theta_{0} ] + B ( 0 , \varepsilon ) .
\end{equation}
In particular, the resolvent set of $\SQ_{\theta}$ is not empty and $\SQ_{\theta}$ has compact resolvent for $\theta \in \Theta$.

For $z \notin \sigma ( \SQ_{\theta} )$ and two functions $u , v$ such that $\theta \mapsto \CU_{\theta} u$ (resp. $\theta \mapsto \CU_{- \overline{\theta}} v$) is an holomorphic (resp. antiholomorphic) function on $\Theta$ with values in $L^{2} ( \R )$, we define
\begin{equation*}
R_{\theta} ( z ) = \< ( \SQ_{\theta} - z )^{- 1} \CU_{\theta} u , \CU_{- \overline{\theta}} v \> .
\end{equation*}
We first fix $z \in B ( - 1 , \varepsilon )$. By the perturbation theory, $\theta \mapsto R_{\theta} ( z )$ is holomorphic on $\Theta$. For $\theta \in \Theta \cap i \R$, the operator $\CU_{\theta}$ is just a scaling operator on $L^{2}( \R )$ and a direct and licit computation gives
\begin{equation*}
R_{\theta} ( z ) = \< ( \CU_{\theta} \CQ \CU_{\theta}^{- 1} - z )^{- 1} \CU_{\theta} u , \CU_{- \overline{\theta}} v \> = \< \CU_{\theta} ( \CQ - z )^{- 1} \CU_{\theta}^{- 1} \CU_{\theta} u , \CU_{- \overline{\theta}} v \> = \< ( \CQ - z )^{- 1} u , v \> .
\end{equation*}
Since the holomorphic function $\theta \mapsto R_{\theta} ( z )$  is constant on $\Theta \cap i \R = i [ - \varepsilon , \varepsilon ]$, it is constant on the whole $\Theta$. We have just proven that
\begin{equation*}
R_{0} ( z ) = R_{\theta_{0}} ( z ) ,
\end{equation*}
for all $z \in B ( - 1 , \varepsilon )$. Since $\SQ_{0}$ and $\SQ_{\theta_{0}}$ have compact resolvent, the functions $z \mapsto R_{0} ( z )$ and $z \mapsto R_{\theta_{0}} ( z )$ are meromorphic on $\C$. Since they coincide in $B ( - 1 , \varepsilon )$, they coincide everywhere. In other words,
\begin{equation} \label{b86}
\< ( \SQ_{0} - z )^{- 1} u ,v \> = \< ( \SQ_{\theta_{0}} - z )^{- 1} \CU_{\theta_{0}} u , \CU_{- \overline{\theta_{0}}} v \> ,
\end{equation}
for all $z \in \C \setminus ( \sigma ( \SQ_{0} ) \cup \sigma ( \SQ_{\theta_{0}} ) )$.

Let $z_{0}$ be an eigenvalue of $\SQ_{0}$. Integrating \eqref{b86} on a small circle around $z_{0}$, we get
\begin{equation} \label{b87}
\< \Pi_{0} u ,v \> = \< \Pi_{\theta_{0}} \CU_{\theta_{0}} u , \CU_{- \overline{\theta_{0}}} v \> ,
\end{equation}
where $\Pi_{0}$ (resp. $\Pi_{\theta_{0}}$) is the generalized spectral projector of $\SQ_{0}$  (resp. $\SQ_{\theta_{0}}$) associated to $z_{0}$ (with $\Pi_{\theta_{0}} = 0$ if $z_{0}$ is not an eigenvalue of $\SQ_{\theta_{0}}$). Let $( \widetilde{u}_{j} )_{j = 1 , \ldots , \rank \Pi_{0}}$ and $( \widetilde{v}_{k} )_{k = 1 , \ldots , \rank \Pi_{0}}$ in $L^{2}( \R^{d} )$ be such that the matrix $( \< \Pi_{0} \widetilde{u}_{j} , \widetilde{v}_{k} \> )_{j , k = 1 , \ldots , \rank \Pi_{0}}$ is invertible. Since it contains the eigenvectors of the harmonic oscillator, the space $\C [ x_{1} , \ldots , x_{d}  ] e^{- x^{2}}$ is dense in $L^{2}( \R^{d} )$. Approximating $\widetilde{u}_{j}$ and $\widetilde{v}_{j}$, one can find $( u_{j} )_{j = 1 , \ldots , \rank \Pi_{0}}$ and $( u_{k} )_{k = 1 , \ldots , \rank \Pi_{0}}$ in $\C [ x_{1} , \ldots , x_{d} ] e^{- x^{2}}$ such that the matrix $( \< \Pi_{0} u_{j} , v_{k} \> )_{j , k = 1 , \ldots , \rank \Pi_{0}}$ is invertible. Since $\vert \theta_{0} \vert \leq \pi / 6$, $\theta \mapsto \CU_{\theta} u_{j}$ (resp. $\theta \mapsto \CU_{- \overline{\theta}} v_{k}$) is an holomorphic (resp. antiholomorphic) function on $\Theta$ with values in $L^{2} ( \R )$ or $L^{2} ( \R^{\pm} )$. Then, \eqref{b87} implies that $( \< \Pi_{\theta_{0}} \CU_{\theta_{0}} u_{j} , \CU_{- \overline{\theta_{0}}} v_{k} \> )_{j , k = 1 , \ldots , \rank \Pi_{0}}$ is invertible. This shows that $z_{0}$ is an eigenvalue of $\SQ_{\theta_{0}}$ and that
\begin{equation*}
\rank \Pi_{0} \leq \rank \Pi_{\theta_{0}} .
\end{equation*}
Since this argument can be performed in the other direction, we conclude that $\sigma ( \SQ_{0} ) = \sigma ( \SQ_{\theta_{0}} )$ with the same multiplicity. Taking into account \eqref{b85}, this shows Proposition \ref{a65}.

\subsection{Resolvent estimates} \label{s8}

The proof of Proposition \ref{a47} is similar to those of Proposition \ref{a33} and of Lemma \ref{c1} and we use again the ideas of \cite{CoGa23_01}. As in \eqref{a63} and \eqref{a59}, we set for $\lambda \geq 1$ large
\begin{equation*}
M_{\lambda} = \{ x \in \R^{d} ; \ \vert \CV ( x ) - \lambda \vert \leq \rho \}  \qquad \text{ and } \qquad 
\CM_{\lambda} = \{ x \in \R^{d} ; \ \dist ( x , M_{\lambda} ) \leq \delta \} ,
\end{equation*}
but now with $\rho \geq 1$ large and $\delta > 0$ small. The case $\lambda \leq  - 1$ large can be treated the same way. Note that since $\vert \CV ( x ) \vert \leq C \vert x \vert^{\alpha}$, we have $\vert x \vert \geq \lambda^{1 / \alpha} / C$ for $x \in \CM_{\lambda}$ where $C > 0$ is independent of $\rho , \delta , \lambda$ for $\rho , \delta$ fixed and $\lambda$ large enough. As in \eqref{a64}, we define the Lipschitz function
\begin{equation}
\chi ( x ) = \phi \Big( \frac{1}{\delta} \sgn \big( \CV ( x ) - \lambda \big) \dist \big( x , M_{\lambda} \big) \Big) ,
\end{equation}
which satisfies $\vert \chi ( x ) \vert \leq 1$ and $\vert \nabla \chi ( x ) \vert \leq \delta^{- 1}$. Mimicking \eqref{a10} and \eqref{a14}, we have
\begin{align}
\Vert \nabla u \Vert^{2} &\leq \Vert ( \CP - i \lambda ) u \Vert \Vert u \Vert ,  \label{a52} \\
\< u , \chi ( \CV - \lambda ) u \> &\leq \Vert ( \CP - i \lambda ) u \Vert \Vert u \Vert + \delta^{- 1} \Vert ( \CP - i \lambda ) u \Vert^{1 / 2} \Vert u \Vert^{3 / 2} .  \label{a53}
\end{align}
We decompose
\begin{equation} \label{a95}
\Vert u \Vert^{2} = \CJ_{1} + \CJ_{2} \quad \text{where} \quad \CJ_{1} = \int_{\CM_{\lambda}} \vert u ( x ) \vert^{2} d x \quad \text{and} \quad \CJ_{2} = \int_{\R^{d} \setminus \CM_{\lambda}} \vert u ( x ) \vert^{2} d x .
\end{equation}

We estimate $\CJ_{1}$ following the arguments of Section \ref{s5}. Then, we use the spherical coordinates $x = r \theta$ and work near $\theta_{0} \in \S^{d - 1}$. We first assume that $v ( \theta_{0} ) \neq 0$. The discussions and estimates of Case 2.1 of Section \ref{s1} still hold in the present case. More precisely, the changes of variables \eqref{d1} is now global on $\R^{d} \setminus \{ 0 \}$ since $\xi ( r , \theta ) = 0$ for $\CV$. Working as in \eqref{d8}, we get that a point $r \theta \in \CM_{\lambda}$  with $\theta$ near $\theta_{0}$ satisfies
\begin{equation*}
\vert \widehat{r}  - \lambda^{1 / \alpha} \vert \leq C \delta + C \rho \lambda^{\frac{1 - \alpha}{\alpha}}  \leq 2 C \delta ,
\end{equation*}
for $\lambda$ large enough. Here, $\widehat{x} = \widehat{r} \widehat{\theta}$ are the spherical coordinates of $x$ after the change of variables \eqref{d1}. Eventually, the proof of \eqref{b23} yields
\begin{equation} \label{a76}
\int_{x \in \CM_{\lambda} \text{ and } \theta \text{ near } \theta_{0}} \vert u ( x ) \vert^{2} d  x \leq C \delta \Vert u \Vert \Vert \nabla u \Vert + C \delta \Vert u \Vert^{2} .
\end{equation}
We now treat the case $v ( \theta_{0} ) = 0$. The geometric setting is similar to Case 2.2 of Section \ref{s1} and we can recycle some arguments of this part. In the variables $\widetilde{x}$ of \eqref{d13} which are now global since $\xi ( r , \theta ) = 0$, we have $\CV ( \widetilde{x} ) = \widetilde{x}_{2} \widetilde{x}_{1}^{\alpha - 1}$. As in \eqref{d24}, it implies that a point $x = r \theta \in M_{\lambda}$ with $\theta$ near $\theta_{0}$ satisfies $\vert  \widetilde{x}_{2}  - \lambda \widetilde{x}_{1}^{1 - \alpha} \vert \leq \rho \widetilde{x}_{1}^{1 - \alpha}$. Following the proof of \eqref{d31}, one can show that a point $x = r \theta \in \CM_{\lambda}$  with $\theta$ near $\theta_{0}$ satisfies
\begin{equation} \label{a77}
\vert  \widetilde{x}_{2}  - \lambda \widetilde{x}_{1}^{1 - \alpha} \vert \leq  C \delta + C \rho \widetilde{r}^{1 - \alpha} \leq 2 C \delta ,
\end{equation}
for $\lambda$ large enough. Here, we use $\widetilde{r} = r \geq \lambda^{1 / \alpha} / C$ and then $\rho \widetilde{r}^{1 - \alpha} \leq ( C \rho \delta^{- 1} \lambda^{\frac{1 - \alpha}{\alpha}} ) \delta \leq \delta$ for $\lambda$ large enough. Therefore, working as in \eqref{b27}, we obtain
\begin{equation} \label{a78}
\int_{x \in \CM_{\lambda} \text{ and } \theta \text{ near } \theta_{0}} \vert u ( x ) \vert^{2} d  x \leq C \delta \Vert u \Vert \Vert \nabla u \Vert + C \delta \Vert u \Vert^{2} .
\end{equation}
Summing up, \eqref{a52}, \eqref{a76} and \eqref{a78} imply
\begin{equation} \label{a79}
\CJ_{1} \leq C \delta \Vert ( \CP - i \lambda ) u \Vert \Vert u \Vert + C \delta \Vert u \Vert^{2} .
\end{equation}

It remains to estimate $\CJ_{2}$. Using that $\chi ( x ) ( \CV ( x ) - \lambda ) \geq 0$ for all $x \in \R^{d}$ and $\chi ( x ) ( \CV ( x ) - \lambda ) = \vert \CV ( x ) - \lambda \vert \geq \rho $ for all $x \in \R^{d} \setminus \CM_{\lambda}$, \eqref{a53} gives
\begin{align}
\CJ_{2} &\leq \rho^{- 1} \< u , \chi ( V - \lambda ) u \>   \nonumber  \\
&\leq \rho^{- 1} \Vert ( \CP - i \lambda ) u \Vert \Vert u \Vert + \rho^{- 1}  \delta^{- 1} \Vert ( \CP - i \lambda ) u \Vert^{1 / 2} \Vert u \Vert^{3 / 2}   \nonumber  \\
&\leq \big( \rho^{- 1} + \rho^{- 2} \delta^{- 2} \big) \Vert ( \CP - i \lambda ) u \Vert \Vert u \Vert + \frac{1}{4} \Vert u \Vert^{2} .  \label{a80}
\end{align}

Combining \eqref{a95} together with \eqref{a79} and \eqref{a80}, we get
\begin{equation*}
\Vert u \Vert^{2} \leq \big( C \delta + \rho^{- 1} + \rho^{- 2} \delta^{- 2} \big) \Vert ( \CP - i \lambda ) u \Vert \Vert u \Vert + \Big( C \delta + \frac{1}{4} \Big) \Vert u \Vert^{2} .
\end{equation*}
It implies that, for all $M \geq 1$, one can choose $\delta > 0$ small enough and then $\rho \geq 1$ large enough such that
\begin{equation*}
\Vert ( \CP - i \lambda )^{-1} \Vert \leq 2^{- 1} M^{-1} ,
\end{equation*}
for $\lambda \geq 1$ large enough. Using $( \CP - z )^{- 1} = ( \CP - i \lambda )^{- 1} ( 1 - ( z - i \lambda ) ( \CP - i \lambda )^{- 1 } )^{- 1}$ and $\Vert \partial_{x} ( \CP - i \lambda )^{- 1} \Vert^{2} \leq \Vert ( \CP - i \lambda )^{- 1} \Vert$ from \eqref{a52}, we obtain Proposition \ref{a47} $i)$ for $\vert \re z \vert \leq M$. On the other hand, the maximal accretivity implies $\Vert ( \CP - z )^{- 1} \Vert \leq M^{-1}$ for $\re z \leq - M$ and Proposition \ref{a47} $i)$ follows in that case.

Let $K > 0$. Since $\CP$ has discrete spectrum in $B ( 0 , 2 K )$ from Proposition \ref {a46}, there exist $C , N > 0$ such that
\begin{equation} \label{a56}
\Vert ( \CP - z )^{- 1} \Vert \leq C + \frac{C}{\dist ( z , \sigma ( \CP ) )^{N}} ,
\end{equation}
for all $z \in B ( 0 , K ) \setminus \sigma ( \CP )$. Moreover, an integration by parts and the Cauchy--Schwarz inequality show that
\begin{equation*}
\Vert \partial_{x} u \Vert^{2} = \re \< u , \CP u \> \leq \Vert \CP u \Vert \Vert u \Vert \leq \Vert ( \CP - z ) u \Vert^{2} + C \Vert u \Vert^{2} ,
\end{equation*}
for all $z \in B ( 0 , K )$. It gives $\Vert \partial_{x} ( \CP - z )^{- 1} \Vert \leq 1 + \sqrt{C} \Vert ( \CP - z )^{- 1} \Vert$. Combined with \eqref{a56}, this implies $ii)$. This completes the proof of the Proposition \ref{a47}.


\begin{thebibliography}{10}

\bibitem{AgCo71_01}
J.~Aguilar and J.-M. Combes, \emph{A class of analytic perturbations for
  one-body {S}chr\"odinger {H}amiltonians}, Comm. Math. Phys. \textbf{22}
  (1971), 269--279.

\bibitem{AlBeNo22_01}
D.~Albritton, R~Beekie, and M.~Novack, \emph{Enhanced dissipation and
  {H}\"ormander's hypoellipticity}, J. Funct. Anal. \textbf{283} (2022), no.~3.

\bibitem{Al08_01}
Y.~Almog, \emph{The stability of the normal state of superconductors in the
  presence of electric currents}, SIAM J. Math. Anal. \textbf{40} (2008),
  no.~2, 824--850.

\bibitem{AlGrHe18_01}
Y.~Almog, D.~S. Grebenkov, and B.~Helffer, \emph{Spectral semi-classical
  analysis of a complex {S}chr\"odinger operator in exterior domains}, J. Math.
  Phys. \textbf{59} (2018), no.~4.

\bibitem{AlGrHe19_01}
Y.~Almog, D.~S. Grebenkov, and B.~Helffer, \emph{On a {S}chr\"odinger operator with a purely imaginary potential
  in the semiclassical limit}, Comm. Partial Differential Equations \textbf{44}
  (2019), no.~12, 1542--1604.

\bibitem{AlHe20_01}
Y.~Almog and B.~Helffer, \emph{The spectrum of a {S}chr\"odinger operator in a
  wire-like domain with a purely imaginary degenerate potential in the
  semiclassical limit}, M\'em. Soc. Math. Fr. (N.S.) (2020), no.~166, vi+94.

\bibitem{AlHe16_01}
Y.~Almog and R.~Henry, \emph{Spectral analysis of a complex {S}chr\"odinger
  operator in the semiclassical limit}, SIAM J. Math. Anal. \textbf{48} (2016),
  no.~4, 2962--2993.

\bibitem{Ar25_01}
V.~Arnaiz, J.-F. Bony, and L.~Michel, \emph{Asymptotic distribution of
  eigenvalues for a class of $\mathscr{PT}$-symmetric {S}chrödinger
  operators}, in preparation.

\bibitem{Ar76_01}
V.~I. Arnold, \emph{Local normal forms of functions}, Invent. Math. \textbf{35}
  (1976), 87--109.

\bibitem{AvFrHeRa25_01}
M.~Averseng, N.~Frantz, F.~H\'erau, and N.~Raymond, \emph{Semiclassical
  tunneling for some 1{D} {S}chr\"odinger operators with complex-valued
  potentials}, arXiv:2510.04296 (2025).

\bibitem{BeHeHeRo15_01}
K.~Beauchard, B.~Helffer, R.~Henry, and L.~Robbiano, \emph{Degenerate parabolic
  operators of {K}olmogorov type with a geometric control condition}, ESAIM
  Control Optim. Calc. Var. \textbf{21} (2015), no.~2, 487--512.

\bibitem{BeCoZ17_01}
J.~Bedrossian and M.~Coti~Zelati, \emph{Enhanced dissipation, hypoellipticity,
  and anomalous small noise inviscid limits in shear flows}, Arch. Ration.
  Mech. Anal. \textbf{224} (2017), no.~3, 1161--1204.

\bibitem{Br11_01}
H.~Brezis, \emph{Functional analysis, {S}obolev spaces and partial differential
  equations}, Universitext, Springer, New York, 2011.

\bibitem{CoDuSe83}
J.-M. Combes, P.~Duclos, and R.~Seiler, \emph{Kre\u{\i}n's formula and
  one-dimensional multiple-well}, J. Funct. Anal. \textbf{52} (1983), no.~2,
  257--301.

\bibitem{CoZDr21_01}
M.~Coti~Zelati and T.~D. Drivas, \emph{A stochastic approach to enhanced
  diffusion}, Ann. Sc. Norm. Super. Pisa Cl. Sci. (5) \textbf{22} (2021),
  no.~2, 811--834.

\bibitem{CoGa23_01}
M.~Coti~Zelati and T.~Gallay, \emph{Enhanced dissipation and {T}aylor
  dispersion in higher-dimensional parallel shear flows}, J. Lond. Math. Soc.
  (2) \textbf{108} (2023), no.~4, 1358--1392.

\bibitem{Da99_01}
E.~B. Davies, \emph{Pseudo-spectra, the harmonic oscillator and complex
  resonances}, R. Soc. Lond. Proc. Ser. A Math. Phys. Eng. Sci. \textbf{455}
  (1999), no.~1982, 585--599.

\bibitem{DeSjZw04_01}
N.~Dencker, J.~Sj\"ostrand, and M.~Zworski, \emph{Pseudospectra of
  semiclassical (pseudo-) differential operators}, Comm. Pure Appl. Math.
  \textbf{57} (2004), no.~3, 384--415.

\bibitem{DiSj99_01}
M.~Dimassi and J.~Sj{\"o}strand, \emph{Spectral asymptotics in the
  semi-classical limit}, London Mathematical Society Lecture Note Series, vol.
  268, Cambridge University Press, 1999.

\bibitem{DrRe81_01}
P.~G. Drazin and W.~H. Reid, \emph{Hydrodynamic stability}, second ed.,
  Cambridge Mathematical Library, Cambridge University Press, Cambridge, 2004,
  With a foreword by John Miles.

\bibitem{He88_01}
B.~Helffer, \emph{Semi-classical analysis for the {S}chr\"odinger operator and
  applications}, Lecture Notes in Mathematics, vol. 1336, Springer-Verlag,
  1988.

\bibitem{He13_01}
B.~Helffer, \emph{Spectral theory and its applications}, Cambridge Studies in
  Advanced Mathematics, vol. 139, Cambridge University Press, Cambridge, 2013.

\bibitem{HeSj84_01}
B.~Helffer and J.~Sj\"{o}strand, \emph{Multiple wells in the semiclassical
  limit. {I}}, Comm. Partial Differential Equations \textbf{9} (1984), no.~4,
  337--408.

\bibitem{HeSj10_01}
B.~Helffer and J.~Sj\"{o}strand, \emph{From resolvent bounds to semigroup
  bounds}, arXiv:1001.4171 (2010).

\bibitem{HeSj21_01}
B.~Helffer and J.~Sj\"{o}strand, \emph{Improving semigroup bounds with resolvent estimates}, Integral
  Equations Operator Theory \textbf{93} (2021), no.~3, Paper No. 36, 41.

\bibitem{He13_02}
R.~Henry, \emph{Spectre et pseudospectre d’op\'erateurs non-autoadjoints},
  Th\`ese Universit\'e Paris Sud (2013).

\bibitem{HeHiSj08_02}
F.~H{\'e}rau, M.~Hitrik, and J.~Sj{\"o}strand, \emph{Tunnel effect for
  {K}ramers-{F}okker-{P}lanck type operators}, Ann. Henri Poincar\'e \textbf{9}
  (2008), no.~2, 209--274.

\bibitem{Ka76_01}
T.~Kato, \emph{Perturbation theory for linear operators}, second ed.,
  Grundlehren der Mathematischen Wissenschaften, vol. 132, Springer, 1976.

\bibitem{Ma26_01}
S.~Malkov, \emph{Harmonic approximation and resolvent estimates for
  non-self-adjoint operators}, arXiv:2601.17643 (2026).

\bibitem{MaRo88_01}
A.~Martinez and M.~Rouleux, \emph{Effet tunnel entre puits d\'eg\'en\'er\'es},
  Comm. Partial Differential Equations \textbf{13} (1988), no.~9, 1157--1187.

\bibitem{MoRaVu23_01}
L.~Morin, N.~Raymond, and S.~V{\~u}~Ng{\d{o}}c, \emph{Eigenvalue asymptotics
  for confining magnetic {S}chr\"odinger operators with complex potentials},
  Int. Math. Res. Not. IMRN (2023), no.~17, 14547--14593.

\bibitem{Sh02_01}
K.~C. Shin, \emph{On the reality of the eigenvalues for a class of {$\mathscr
  P\mathscr T$}-symmetric oscillators}, Comm. Math. Phys. \textbf{229} (2002),
  no.~3, 543--564.

\bibitem{Si75_01}
Y.~Sibuya, \emph{Global theory of a second order linear ordinary differential
  equation with a polynomial coefficient}, North-Holland Math. Stud., vol.~18,
  Elsevier, Amsterdam, 1975 (English).

\bibitem{Si83_01}
B.~Simon, \emph{Semiclassical analysis of low lying eigenvalues. {I}.
  {N}ondegenerate minima: asymptotic expansions}, Ann. Inst. H. Poincar\'e{}
  Sect. A (N.S.) \textbf{38} (1983), no.~3, 295--308.

\bibitem{Sj19_01}
J.~Sj\"{o}strand, \emph{Non-self-adjoint differential operators, spectral
  asymptotics and random perturbations}, Pseudo-Differential Operators. Theory
  and Applications, vol.~14, Birkh\"{a}user/Springer, Cham, 2019.

\bibitem{Tr97_01}
L.~Trefethen, \emph{Pseudospectra of linear operators}, SIAM Rev. \textbf{39}
  (1997), no.~3, 383--406.

\bibitem{We21_01}
D.~Wei, \emph{Diffusion and mixing in fluid flow via the resolvent estimate},
  Sci. China Math. \textbf{64} (2021), no.~3, 507--518.

\bibitem{Wh34_01}
H.~Whitney, \emph{Functions differentiable on the boundaries of regions}, Ann.
  of Math. \textbf{35} (1934), no.~3, 482--485.

\end{thebibliography}

\bibliographystyle{amsplain}

\end{document}